\documentclass[a4, 11pt]{amsart}
\usepackage{amssymb,latexsym,hyperref}
\usepackage{tikz}
\usepackage{tikz-3dplot}
\usepackage{pgfplots} % Required for the 'axis' environment
\pgfplotsset{compat=1.18}

\usepackage{color}
\usepackage{comment}
\usepackage{xcolor}
\usepackage{mathtools}
\usepackage[dvipsnames]{xcolor}

\usepackage{graphicx}
\usepackage[left=25mm, right=25mm, top=20mm, bottom=20mm, includefoot, includehead]{geometry}
\usepackage{tikz-cd}
\usepackage{cancel}
\theoremstyle{plain}
\newtheorem{theorem}{Theorem}[section]
\newtheorem{corollary}[theorem]{Corollary} 
\newtheorem{proposition}[theorem]{Proposition}
\newtheorem{lemma}[theorem]{Lemma}
\newtheorem{conjecture}[theorem]{Conjecture}

\newtheorem*{cor}{Corollary}
\theoremstyle{definition}
\newtheorem{definition}[theorem]{Definition}

\newtheorem{example}[theorem]{Example}

\newtheorem{remark}[theorem]{Remark}

\newcommand{\enm}[1]{\ensuremath{#1}}          %
\renewcommand{\bar}[1]{\overline{#1}}

\newcommand{\CC}{\enm{\mathbb{C}}}

\newcommand{\RR}{\enm{\mathbb{R}}}

\newcommand{\ZZ}{\enm{\mathbb{Z}}}

\newcommand{\PP}{\enm{\mathbb{P}}}

\renewcommand{\theta}{\vartheta}
\renewcommand{\epsilon}{\varepsilon}

\DeclareMathOperator{\Gr}{Gr}
\DeclareMathOperator{\red}{red}
\DeclareMathOperator{\Res}{Res}

\DeclareMathOperator{\reg}{reg}
\DeclareMathOperator{\sort}{sort}
\DeclareMathOperator{\spec}{Spec}
\DeclareMathOperator{\proj}{Proj}

\renewcommand{\to}[1][]{\xrightarrow{\ #1\ }}

\newcommand{\Bl}{\mathit{Bl}}

\newcommand{\old}[1]{}

\newcommand{\rk}{\mathrm{rk}}

\DeclareMathOperator{\itr}{int}
\DeclareMathOperator{\Sing}{Sing}

\makeatletter
\@namedef{subjclassname@2020}{%
\textup{2020} Mathematics Subject Classification}
\makeatother



\date{}

\title{The ABCT Variety $V(3,n)$ is a Positive Geometry}

\author{Dawei Shen}
\address{University of Michigan, Department of Mathematics, 530 Church Street, Ann Arbor, USA}
\email{dwshen@umich.edu}

\author{Emanuele Ventura}
\address{Politecnico di Torino, Dipartimento di Scienze Matematiche ``G. L. Lagrange'', Corso Duca degli Abruzzi 24\\
	10129 Torino, Italy}
\email{emanuele.ventura@polito.it}

\begin{document}

\maketitle

\begin{abstract}
The {\it ABCT variety} $V(3,n)$ is the image closure of the rational Veronese map from the Grassmannian $\Gr(2,n)$ to the Grassmannian $\Gr(3,n)$. It was studied by Arkani-Hamed--Bourjaily--Cachazo--Trnka \cite{ABCT} %and by Nandan--Volovich--Wen
in the context of  tree-level scattering amplitudes arising in planar $\mathcal N=4$ supersymmetric Yang-Mills theory and Witten's twistor string theory. From this perspective, $V(3,n)$ is conjectured to be a positive 
geometry by Lam \cite{Lam24}.

In this paper, we study the combinatorial and algebraic geometry aspects of $V(3,n)$ and its subvarieties induced by iteratively taking analytic boundaries of the totally nonnegative part. 
We interpret these subvarieties as point configurations on $\mathbb{P}^2$ by the Gelfand-MacPherson correspondence. We construct a top-degree meromorphic form on $V(3,n)$ and show that it is a positive 
geometry, proving Lam's conjecture.
\end{abstract}

\tableofcontents

\section{Introduction}

\textit{Scattering amplitudes} give a measure of the probability of outcomes of quantum particles' collisions and are the central objects in quantum field theory \cite{BHPZ24,EH15,Schwartz14}. Perturbative expansions lead to expressions of the amplitudes as a gigantic sum indexed by combinatorial objects called \textit{Feynman diagrams},  representing all possible sequences of  events. While Feynman’s method yields experimental data with great accuracy, the number of Feynman diagrams increases super-exponentially with the number of particles involved, rendering computations intractable. Surprisingly, the final answers in some quantum field theories are simple despite the complicated computations involved. This hints at the existence of deeper structures.

Over the past two decades, multiple geometric structures are revealed to bypass Feynman diagrammatic computations for the $\mathcal{N} = 4$ supersymmetric Yang-Mills theory (SYM) scattering amplitude. Arguably the most notable one starts out with Hodge \cite{Hodges} pointing out that certain tree-level amplitudes can be interpreted as volumes of polytopes in projective space. This was then generalized to higher loop-level in the {\it amplituhedron} construction of Arkani-Hamed and Trnka \cite{BT14} inside Grassmannians.
Soon after, more geometric structures appeared for various other quantum field theories and cosmology, such as \textit{correlahedra} \cite{EHM17}, \textit{cosmological polytopes} \cite{AHBP17}, \textit{ABHY associahedra} \cite{AHBHY18}, and more recently \textit{cosmohedra} \cite{AHFV24}. Some notion of ``volume" can be defined for each of these geometric objects, which gives the corresponding scattering amplitude. 
%Furthermore, {\it triangulations} of these geometric objects correspond to decompositions of the scattering amplitude that coincide with different expansions, whereas they were traditionally hard to be compared with each other; see \cite{ PSTW24,EZLPTSW24,E-ZLT25} for a deep study of the relation between triangulations and the Britto-Cachazo-Feng-Witten recurrence, and the combinatorial structure of tilings. Also, taking boundaries of these geometric objects explains the factorization properties of scattering amplitudes.

\textit{Positive geometries}, introduced by Arkani-Hamed, Bai, and Lam \cite{AHBL17}, capture the underlying combinatorics and geometry, generalizing all the above examples. A \textit{positive geometry} is a triple of a projective algebraic variety, a closed semi-algebraic subset of its real points, and a top-degree rational differential form (the {\it canonical form}), satisfying a subtle recursion determined by taking boundaries of the geometry and residues of the canonical form; see Definition \ref{def: positive geom} for details. From the physics perspective, this notion is designed to reflect the properties of {\it locality} (the presence of simple poles in the amplitude at particular locations) and {\it unitarity} (residue at those simple poles factorizes) of scattering amplitudes. 

Besides the ones above, the prototypical example of a positive geometry is the {\it positive Grassmannian} \cite{Postnikov06, AHBCGPT16}. This wonderful object lies at the crossroads of algebraic combinatorics, algebraic geometry, and theoretical physics with deep connections to important topics in each of these domains such as cluster algebras \cite{FWZ25} and the KP equation \cite{KW14}, an important PDE in mathematical physics. Other important basic examples include the moduli space $\mathcal M_{0,n}$ \cite[\S 1]{Lam24}, {\it toric varieties} \cite[\S 5.6]{AHBL17}, {\it polytopes} and blowups thereof \cite{brauner2024wondertopes}, and {\it polypols} \cite{Kohn25}. For more on the profound and fast growing connections between positivity, combinatorial algebraic geometry and particle physics we refer to the recent informative surveys \cite{DPSV25, FS25, Lam25, RST25}. Recently, another interesting and fertile definition of positive geometry was given by Brown and Dupont, based on mixed Hodge theory \cite{BD24}. 
In this paper, we shall follow the former definition.  

The problem of computing the  $\mathcal{N} = 4$ SYM amplitude has yet another approach by realizing it as contour integrals over the complex Grassmannian $\Gr(k,n)$, the space of  $k$-planes in $\mathbb{C}^n$. In this framework, residues of specific meromorphic forms correspond to Yangian-invariant building blocks of scattering amplitudes \cite[\S 9]{EH15}. A breakthrough result of Arkani-Hamed {\it et al.} \cite{ABCT} showed that these residues can be organized into a single algebraic variety -- a subvariety of $\Gr(k,n)$ -- onto which the full tree-level amplitude localizes. This construction provides another unified picture of amplitude computation and gives a geometric interpretation to physical features such as the BCFW recursion.

This subvariety, known in this context as the \emph{ABCT variety} and denoted by $V(k,n)$, arises as the image closure of a rational Veronese map from $\Gr(2,n)$ to $\Gr(k,n)$, see Section \ref{intro V(k,n)}. In the case $k=3$, the algebraic geometry of the ABCT variety $V(3,n)$ is studied in \cite{ARS24}: it is a reduced, irreducible, and Cohen-Macaulay subvariety of $\Gr(3,n)$ of expected codimension $n-5$. It can be realized as the degeneracy locus of a vector bundle morphism, and its radical ideal of vanishing inside $\Gr(3,n)$ is known, see Theorems \ref{prop:interpretation of V} and \ref{Ideal of ABCT}. In this paper, we will establish that it is also rational and normal, see Theorem \ref{geometry of V and its boundaries}.

The intersection theory of the ABCT variety is also related to computing the tree-level SYM amplitude from twistor string theory introduced by Witten. Roiban, Spradlin, and Volovich \cite{RSV04} showed that the amplitude can be expressed as a sum over points in the intersection of $V(k,n)$ and a specific Schubert variety. From this perspective, the ABCT variety $V(k,n)$ is conjecturally a subvariety of $\Gr(k,n)$ of expected dimension $2(n-2)$ and its fundamental class $[V(k,n)] \in H^{2(k-2)(n-k-2)}(\Gr(k,n))$ is conjectured such that its Schubert expansion has coefficient in front of the Schubert class represented by a $k-2$ by $n-k-2$ rectangle given by the Eulerian number $E_{n-3,k-2}$. This matches the number of solutions to scattering equations considered by Cachazo, He, and Yuan \cite{CHY13}, which has an interpretation from the spinor-helicity formalism \cite{Lam24}. For the $k=3$ case, \cite{ARS24} gives a recursive formula in symmetric functions for the fundamental class $[V(3,n)]$, which recovers the Eulerian number coefficient as a corollary.

From a combinatorial algebraic geometry point of view, $V(k,n)$ has natural connections with {\it point configurations} on a rational normal curve \cite{CGMS18, ARS24}. Any point in the image of the rational Veronese map $\Gr(2,n)\dashrightarrow \Gr(k,n)$ satisfies that its matrix representative has $n$ columns giving a point configuration in $\mathbb{P}^{k-1}$ lying on a common rational normal curve of degree $k-1$. The ABCT variety $V(k,n)$ is by definition the Zariski closure of the image and it is shown in \cite{ARS24} that $V(3,n)$ is exactly given by the condition that the nonzero columns give points lying on a (not necessarily normal) conic in $\mathbb{P}^2$, see Remark \ref{rmk:point configuration for V}. In particular, the extra points needed to form the Zariski closure have an easy description in that zero columns are allowed or the conic is allowed to degenerate to a union of two lines. The perspective of point configurations will be relevant for us not only as set-theoretic descriptions and dimension counts of $V(3,n)$ and its ``boundaries," but also when we relate $V(3,n)$ in $V(n-3,n)$ in Section \ref{paozhuan}.

From a physics perspective, $V(k,n)$ captures the configurations where external particles (encoded as marked points on $\mathbb{P}^1$) are mapped into $\mathbb{P}^{k-1}$ via the Veronese embedding, reflecting the localization of the path integral to specific algebraic curves in twistor space. This variety simultaneously makes manifest the cyclic symmetry and soft limits of amplitudes. In the NMHV case (i.e., $ k = 3 $), this means that the associated $n $ points lie on a conic in $\mathbb{P}^2$.
%Another important combinatorial structure for scattering amplitudes is the {\it momentum amplituhedron} \cite{DFLP19}. Yang-Mills amplitudes may be computed as an integral over the this object. 
There are hints of a deep connection between $V(k,n)$ to the momentum amplituhedron \cite{HKZ22}, which is very far from clear.  

In this paper, we focus on the case $k=3$, building on the results by Agostini, Ramesh and the first author \cite{ARS24}, to study the combinatorics and algebraic geometry of $V(3,n)$, along with its iterated boundaries induced by the totally nonnegative part $V(3,n)_{\geq0}$. Our main result proves Lam's conjecture \cite{Lam24} that $V(3,n)$ is a positive geometry and the interior of $\Gr(2,n)_{\geq0}$ maps to the interior of $V(3,n)_{\geq0}$ so that the two canonical forms are related by pushforward. This builds a parallel between $V(3,n)$ and other positive geometries, such as toric varieties, Grassmannians, polytopes, etc. It establishes the foundation for the further conjecture that $V(3,n)_{\geq0}$ maps diffeomorphically to the interior of the momentum amplituhedron and their canonical forms are related by pushforward \cite{HKZ22}. This conjecture can be interpreted as that $V(3,n)$ plays a similar role for the momentum amplituhedron as a toric variety plays for a polytope,
and it bridges the two existing approaches of computing the  $\mathcal{N} = 4$  SYM scattering amplitude using the ABCT variety and momentum amplituhedron, respectively. \\

\noindent {\bf An Example and the Main Result.}

Iterated boundaries of $V(3,7)$ correspond to imposing positroid conditions on point configurations on a common conic in $\mathbb{P}^2$. By ``positroid conditions," we mean that columns are set to $0$, or give projective points that collide or become collinear. When points are collinear, the conic must be the union of two lines, and thus all other points lie on the other line. The figure below shows the poset of iterated boundaries of codimensions $0,1,2$ up to dihedral symmetry: to get the actual poset, we possibly reflect the diagram and then label points in counterclockwise order. A red dot enclosed by $1$ or $2$ circles denotes $2$ or $3$ colliding points. Two different configurations are in a box when they are related by a non-dihedral relabeling of points. Forgetting the conic condition gives a rank $3$ matroid on $7$ vertices, which is a positroid because the points are labeled counterclockwise in the actual poset. This identifies iterated boundaries with a subset of rank $3$ positroids on $7$ vertices. Theorem \ref{thm:mainintro} states that the poset structure on iterated boundaries is the induced sub-poset. Black covering relations come from covering relations in positroids. Pink covering relations come from non-covering relations in positroids that become covering relations in the induced sub-poset. The smooth conics in row $3$ are thin because they impose no conditions, since $5$ points determine a conic. The only iterated boundaries of $V(3,7)$ that are not positroid varieties are $V(3,7)$ itself and the codimension $1$ boundary given by the right-most configuration on the second row. We omit codimension $\geq3$ boundaries in the picture. We also omit the codimension $2$ boundary given by setting one column to zero. It is child of only the rightmost codimension $1$ boundary and it is isomorphic to $V(3,6)$ by forgetting the zero column. 

\begin{tikzpicture}

% Draw small sub-pictures, for demonstration:
\begin{scope}[shift={(0,0)}]
% Node A: Small diagram of 7 points on a hyperbola, encapsulated in a node
\node (A) {
% Embedded TikZ for hyperbola points
\begin{tikzpicture}[scale=0.25]
% Draw the hyperbola xy=3 (parametric approach)
\draw[domain=0.13:4.9,smooth,variable=\x,blue,thick]
plot ({\x},{0.5/\x});

% Mark 7 specific points on the hyperbola
\foreach \x in {0.17, 0.25 ,0.5,1,2,3,4} {
\pgfmathsetmacro\y{0.5/\x}
\filldraw[red] (\x,\y) circle (5pt); % Draw a point
}
\end{tikzpicture}
};
\end{scope}

\begin{scope}[shift={(5,-2.5)}]
\node (B) {
% Embedded TikZ for hyperbola points

\begin{tabular}{c}
\begin{tikzpicture}[scale=0.25]
\draw[domain=0.13:4.9,smooth,variable=\x,blue,thick]
plot ({\x},{0.5/\x});

% Mark 7 specific points on the hyperbola
\foreach \x in {0.17, 0.25,0.8,2,3,4} {
\pgfmathsetmacro\y{0.5/\x}
\filldraw[red] (\x,\y) circle (5pt); % Draw a point
}
% Outer circle
\draw[black, thick] (0.17,3) circle (0.4);
\end{tikzpicture}
\end{tabular}
};
\end{scope}

\begin{scope}[shift={(0,-2.5)}]
\node (C) {
% Embedded TikZ for hyperbola points
\begin{tikzpicture}[scale=0.25]
\draw[domain=-1:5,smooth,variable=\x,blue,thick]
plot ({\x},{0});

\draw[domain=-1:4,smooth,variable=\x,blue,thick]
plot ({0},{\x});

% Mark 7 specific points on the hyperbola
\foreach \x in {1,2,3} {
\filldraw[red] (0,\x) circle (5pt); 
}

\foreach \x in {1,2,3,4} {
\filldraw[red] (\x,0) circle (5pt); 
}
\end{tikzpicture}
};
\end{scope}

\begin{scope}[shift={(-5,-2.5)}]
\node (D) {
% Embedded TikZ for hyperbola points
\begin{tikzpicture}[scale=0.25]
\draw[domain=-1:6,smooth,variable=\x,blue,thick]
plot ({\x},{0});

\draw[domain=-1:3,smooth,variable=\x,blue,thick]
plot ({0},{\x});

% Mark 7 specific points on the hyperbola
\foreach \x in {1,2} {
\filldraw[red] (0,\x) circle (5pt); 
}

\foreach \x in {1,2,3,4,5} {
\filldraw[red] (\x,0) circle (5pt); 
}
\end{tikzpicture}
};
\end{scope}

\begin{scope}[shift={(-6.2,-5.6)}]
\node (E1) {
% Embedded TikZ for hyperbola points
\begin{tikzpicture}[scale=0.25]
\draw[domain=-1:6,smooth,variable=\x,blue,thick]
plot ({\x},{0});

\draw[domain=-1:2,smooth,variable=\x,blue]
plot ({0},{\x});

% Mark 7 specific points on the hyperbola
\foreach \x in {1,0} {
\filldraw[red] (0,\x) circle (5pt); 
}

\foreach \x in {1,2,3,4,5} {
\filldraw[red] (\x,0) circle (5pt); 
}
\end{tikzpicture}
};
\end{scope}

\begin{scope}[shift={(-4.5,-5.8)}]
\node (E2) {
% Embedded TikZ for hyperbola points
\begin{tikzpicture}[scale=0.25]
\draw[domain=-1:5,smooth,variable=\x,blue,thick]
plot ({\x},{0});

\draw[domain=-1:3,smooth,variable=\x,blue,thick]
plot ({0},{\x});

% Mark 7 specific points on the hyperbola
\foreach \x in {0,1,2} {
\filldraw[red] (0,\x) circle (5pt); 
}

\foreach \x in {1,2,3,4} {
\filldraw[red] (\x,0) circle (5pt); 
}
\end{tikzpicture}
};
\end{scope}

\begin{scope}[shift={(-3,-5.9)}]
\node (E3) {
% Embedded TikZ for hyperbola points
\begin{tikzpicture}[scale=0.25]
\draw[domain=-1:4,smooth,variable=\x,blue,thick]
plot ({\x},{0});

\draw[domain=-1:4,smooth,variable=\x,blue,thick]
plot ({0},{\x});

% Mark 7 specific points on the hyperbola
\foreach \x in {0,1,2,3} {
\filldraw[red] (0,\x) circle (5pt); 
}

\foreach \x in {1,2,3} {
\filldraw[red] (\x,0) circle (5pt); 
}
\end{tikzpicture}
};
\end{scope}

\begin{scope}[shift={(-1.5,-6.9)}]
\node (F1) [draw, rectangle, inner ysep=1pt, inner xsep=-4pt]{
\begin{tabular}{c}
\begin{tikzpicture}[scale=0.25]
\draw[domain=-1:3,smooth,variable=\x,blue,thick]
plot ({\x},{0});

\draw[domain=-1:5,smooth,variable=\x,blue,thick]
plot ({0},{\x});

% Mark 7 specific points on the hyperbola
\foreach \x in {1,2,3,4} {
\filldraw[red] (0,\x) circle (5pt); 
}

\foreach \x in {1,2} {
\filldraw[red] (\x,0) circle (5pt); 
}
% Outer circle
\draw[black, thick] (0,4) circle (0.4);

\end{tikzpicture}
  \\

 \begin{tikzpicture}[scale=0.25]
\draw[domain=-1:3,smooth,variable=\x,blue,thick]
plot ({\x},{0});

\draw[domain=-1:5,smooth,variable=\x,blue,thick]
plot ({0},{\x});

% Mark 7 specific points on the hyperbola
\foreach \x in {1,2,3,4} {
\filldraw[red] (0,\x) circle (5pt); 
}

\foreach \x in {1,2} {
\filldraw[red] (\x,0) circle (5pt); 
}
% Outer circle
\draw[black, thick] (0,3) circle (0.4);

\end{tikzpicture}

 \end{tabular}
% Embedded TikZ for hyperbola points
};
\end{scope}

\begin{scope}[shift={(0,-6.7)}]
\node (F2) [draw, rectangle, inner ysep=1pt, inner xsep=-5pt]{
\begin{tabular}{c}
\begin{tikzpicture}[scale=0.25]
\draw[domain=-1:4,smooth,variable=\x,blue,thick]
plot ({\x},{0});

\draw[domain=-1:4,smooth,variable=\x,blue,thick]
plot ({0},{\x});

% Mark 7 specific points on the hyperbola
\foreach \x in {1,2,3} {
\filldraw[red] (0,\x) circle (5pt); 
}

\foreach \x in {1,2,3} {
\filldraw[red] (\x,0) circle (5pt); 
}
% Outer circle
\draw[black, thick] (0,3) circle (0.4);

\end{tikzpicture}
  \\

 \begin{tikzpicture}[scale=0.25]
\draw[domain=-1:4,smooth,variable=\x,blue,thick]
plot ({\x},{0});

\draw[domain=-1:4,smooth,variable=\x,blue,thick]
plot ({0},{\x});

% Mark 7 specific points on the hyperbola
\foreach \x in {1,2,3} {
\filldraw[red] (0,\x) circle (5pt); 
}

\foreach \x in {1,2,3} {
\filldraw[red] (\x,0) circle (5pt); 
}
% Outer circle
\draw[black, thick] (0,2) circle (0.4);

\end{tikzpicture}

 \end{tabular}
% Embedded TikZ for hyperbola points
};
\end{scope}

\begin{scope}[shift={(1.6,-5.8)}]
\node (F3){
\begin{tabular}{c}
\begin{tikzpicture}[scale=0.25]
\draw[domain=-1:5,smooth,variable=\x,blue,thick]
plot ({\x},{0});

\draw[domain=-1:3,smooth,variable=\x,blue,thick]
plot ({0},{\x});

% Mark 7 specific points on the hyperbola
\foreach \x in {1,2} {
\filldraw[red] (0,\x) circle (5pt); 
}

\foreach \x in {1,2,3,4} {
\filldraw[red] (\x,0) circle (5pt); 
}
% Outer circle
\draw[black, thick] (0,2) circle (0.4);

\end{tikzpicture}
 \end{tabular}
% Embedded TikZ for hyperbola points
};
\end{scope}

\begin{scope}[shift={(3.3,-5.6)}]
\node (F4){
\begin{tabular}{c}
\begin{tikzpicture}[scale=0.25]
\draw[domain=-1:6,smooth,variable=\x,blue,thick]
plot ({\x},{0});

% Mark 7 specific points on the hyperbola
\foreach \x in {1} {
\filldraw[red] (0,\x) circle (5pt); 
}

\foreach \x in {1,2,3,4,5} {
\filldraw[red] (\x,0) circle (5pt); 
}
% Outer circle
\draw[black, thick] (0,1) circle (0.4);

\end{tikzpicture}
 \end{tabular}
% Embedded TikZ for hyperbola points
};
\end{scope}

\begin{scope}[shift={(5,-6.4)}]
\node (F5) [draw, rectangle,inner ysep=1pt, inner xsep=-5pt]{
% Embedded TikZ for hyperbola points

\begin{tabular}{c}
\begin{tikzpicture}[scale=0.25]
\draw[domain=0.13:4.9,smooth,variable=\x,blue]
plot ({\x},{0.5/\x});

% Mark 7 specific points on the hyperbola
\foreach \x in {0.17, 0.3,1,2.5,4} {
\pgfmathsetmacro\y{0.5/\x}
\filldraw[red] (\x,\y) circle (5pt); % Draw a point
}
% Outer circle
\draw[black, thick] (0.17,3) circle (0.4);
\draw[black, thick] (0.3,1.67) circle (0.4);
\end{tikzpicture} \\
\begin{tikzpicture}[scale=0.25]
\draw[domain=0.13:4.9,smooth,variable=\x,blue]
plot ({\x},{0.5/\x});

% Mark 7 specific points on the hyperbola
\foreach \x in {0.17, 0.3,1,2.5,4} {
\pgfmathsetmacro\y{0.5/\x}
\filldraw[red] (\x,\y) circle (5pt); % Draw a point
}
% Outer circle
\draw[black, thick] (0.17,3) circle (0.4);
\draw[black, thick] (1,0.5) circle (0.4);
\end{tikzpicture} 
\end{tabular}
};
\end{scope}

\begin{scope}[shift={(6.5,-5.7)}]
\node (F6) {
% Embedded TikZ for hyperbola points

\begin{tabular}{c}
\begin{tikzpicture}[scale=0.25]
\draw[domain=0.15:4.9,smooth,variable=\x,blue]
plot ({\x},{0.5/\x});

% Mark 7 specific points on the hyperbola
\foreach \x in {0.2,1,2.5,4} {
\pgfmathsetmacro\y{0.5/\x}
\filldraw[red] (\x,\y) circle (5pt); % Draw a point
}
% Outer circle
\draw[black, thick] (0.2,2.5) circle (0.4);
\draw[black, thick] (0.2,2.5) circle (0.6);
\end{tikzpicture}
\end{tabular}
};
\end{scope}

% Edges between shifted nodes
\draw[-,pink,thick] (A.south) -- (B.north);
\draw[-,pink, thick] (A.south) -- (C.north);
\draw[-,pink, thick] (A.south) -- (D.north);

\draw[-, thick] (D.south) -- (E1.north);
\draw[-, thick] (D.south) -- (E2.north);
\draw[-, thick] (D.south) -- (F1.north);
\draw[-, thick] (D.south) -- (F4.north);

\draw[-, thick] (C.south) -- (E2.north);
\draw[-, thick] (C.south) -- (E3.north);
\draw[-, thick] (C.south) -- (F2.north);
\draw[-, thick] (C.south) -- (F3.north);

\draw[-,pink,thick] (B.south) -- (F1.north);
\draw[-,pink,thick] (B.south) -- (F2.north);
\draw[-,pink,thick] (B.south) -- (F3.north);
\draw[-,pink,thick] (B.south) -- (F4.north);

\draw[-,pink,thick] (B.south) -- (F5.north);
\draw[-,pink,thick] (B.south) -- (F6.north);
\end{tikzpicture}

Our main result is: 

\begin{theorem}\label{thm:mainintro}
The following properties of the ABCT variety $V(3,n)$ hold true:

\begin{enumerate}

\item[(i)] The variety $(V(3,n), V(3,n)_{\geq0}, \Omega_{V(3,n)})$ is a positive geometry. The canonical form $\Omega_{V(3,n)}$ equals the pushforward of $\Omega_{\Gr(2,n)}$ along the rational Veronese map.\\

\item[(ii)] The (analytic) iterated boundaries of the ABCT variety $V(3,n)$ are (the totally nonnegative parts of) the reduced subschemes of intersections of $V(3,n)$ and certain positroid varieties in $\Gr(3,n)$ described in Definition \ref{dfn:iterated boundaries two types}.\\

\item[(iii)] The poset structure of iterated boundaries is the induced sub-poset on the positroids as above. The poset is naturally graded by the dimension of the iterated boundary. \\

\item[(iv)] The variety $V(3,n)$ and all its iterated algebraic boundaries are reduced, irreducible, rational, normal, Cohen-Macaulay, regular in codimension $1$, and have expected dimensions from the perspective of point configurations.
\end{enumerate}
\end{theorem}

\begin{remark}
    It is well known that the totally nonnegative Grassmannian $(\Gr(k,n),\Gr_{\geq0}(k,n))$  is a positive geometry with iterated boundaries given by positroid varieties. For more details, see Subsections \ref{Positive Grassmannians and Positroid Varieties} to \ref{The Positive Grassmannian is a Positive Geometry}. In particular, the canonical form $\Omega_{\Gr(2,n)}$ in Theorem \ref{thm:mainintro}.(i) makes sense.

    Theorem \ref{thm:mainintro}.(ii), (iii), (iv) are analogous to, but not a consequence of, the fact that positroid varieties are iterated boundaries of the Grassmannian satisfying all of the same properties as listed above. Our poset of iterated boundaries can be embedded in the poset of positroid varieties, but is far from a lower order ideal. The intersection between $V(3,n)$ and positroid varieties is in general ill-behaved. 
    
    As shown in the above example, some positroid varieties are contained in $V(3,n)$ and some are not. For those that are not contained in $V(3,n)$, the scheme theoretic intersection is often not reduced and we need to take the reduced subscheme to form the iterated boundary. We characterize exactly when they are not reduced and identify their reduced subscheme in  Theorem \ref{thm:matrix description of colliding boundary} and \ref{Ideal of colliding boundary}. The intersection is far from proper and the codimensions in different positroid varieties differ, even for those positroid varieties that we take intersection with $V(k,n)$ to form iterated boundaries. For the rest of the positroid varieties, the intersection can have multiple components, which is the case for positroid divisors as in the proof of Theorem \ref{decomposition of boundary}. Our result that all iterated boudnaries are reduced, irreducible, normal, and Cohen-Macaulay is very surprising, given the geometrically ill-behaved intersection. For example, Cohen-Macaulayness fails if we do not take the reduced subscheme of the intersection, and normality fails trivially with non-reduced schemes. Also, our result that the dimension of the intersection can be predicted from the perspective of point configurations is particularly desired, given the wild behavior of dimensions of the intersections.

\end{remark}

\noindent {\bf Organization of the article}. \\
In \S \ref{sec: pos geom}, we give the definition of positive geometries and see an important family of examples given by Grassmannians and positroid varieties. In \S \ref{intro V(k,n)}, we give the definition of ABCT variety $V(k,n)$ and review known results on the geometry of $V(3,n)$ from \cite{ARS24}. In \S\ref{section: candidate form}, we study our candidate canonical form $\Omega_{V(3,n)}$ in local coordinates. In \S\ref{sec:boundary components}, we study the structure of the analytic boundary components of $V(3,n)_{\geq0}$. In \S \ref{rational normal}, we identify a family of subvarieties inside $V(3,n)$, which turn out to be all the iterated boundaries. We then study their algebro-geometric properties. In \S \ref{n=5,6 case}, we prove the base case that $V(3,5)$ and $V(3,6)$ equipped with the candidate canonical form from \S\ref{section: candidate form} are positive geometries. Our central findings regarding the general theory of positive geometries is in \S\ref{sec: pos geoms as fibrations over p^1}, where we study positive geometries arising as fibrations over a $\mathbb P^1$ and retract. In \S\ref{sec: canonical form of Z(A|B)}, we explicitly compute the canonical form for certain positroid varieties. We prove our main result that $V(3,n)$ is a positive geometry in \S\ref{sec:final proof}. We prove that $V(n-3,n)$ is a positive geometry is study its boundary structures in \S\ref{paozhuan}. We state open problems in the last section.\\

\noindent {\bf Acknowledgements}. \\
We are grateful to Thomas Lam and Bernd Sturmfels for suggesting this project. We are grateful to Thomas Lam for his meaningful guidance throughout the project. We would like to 
warmly thank Vincenzo Antonelli, Claudia Fevola and Mateusz Micha\l{}ek for extensive important discussions during several stages of this project. D.S. thanks Yucong Lei and Shend Zhjeqi for the helpful discussions. We thank the organizers of the {\it Combinatorics of Fundamental Physics Workshop} held at the IAS in Princeton in November 2024, where the work on this project started. D.S was partially supported by NSF grant DMS-2348799. E.V. is a member of the research group GNSAGA of INdAM and acknowledges that the preparation of this article was partially funded by the Italian Ministry of University and Research, through the project {\it Applied Algebraic Geometry of Tensors}, PRIN 2022 Protocol no. 2022NBN7TL.

\section{Positive Geometries, Positive Grassmannians, and Positroid Varieties}\label{sec: pos geom}

\subsection{Positive geometries}
In this subsection, we give the definition of positive geometries, following \cite{lam2022invitationpositivegeometries}. The data of a positive geometry consists of a triple $(X,X_{\geq0},\Omega_X)$, where $X$ is a projective variety, $X_{\geq0}$ is a closed semi-algebraic subset, and $\Omega_X$ is a top-degree rational differential form on $X$. We first introduce the ingredients required to define a positive geometry.

\begin{definition}[{\bf Iterated boundary components}]\label{def:iterated boundary}
Let $X$ be an irreducible complex projective variety of dimension $d$. Let $X_{\geq0}$ be a closed semi-algebraic subset of the real points $X(\RR)\subset X$ such that $X_{\geq0}$ equals the analytic closure of its analytic interior $X_{>0}:=\itr X_{\geq0}$. Moreover, $X_{>0}$ is an orientable real manifold of dimension $d$. 

Now, consider the analytic boundary $\partial X_{\geq0}:=X_{\geq0}\setminus X_{>0}$ and denote by $C$ its Zariski closure in $X$, which is a divisor on $X$. Let $C=C_1\cup C_2\cup \cdots \cup C_r$ be the irreducible decomposition of $C$. Let $C_{i,\geq 0 }$ denote the analytic closure of the intersection $C_i\cap X_{>0}$. We call the $C_i$ (resp., $C_{i,\geq 0}$) the {\it algebraic} (resp., {\it analytic}) {\it boundary components} of the pair $(X,X_{\geq0})$. We drop the word algebraic or analytic when it is clear from the context whether we are referring to the algebraic variety or to the semi-algebraic subset. Starting from $(X,X_{\geq0})$, we can continue the above construction and take boundary components of the pair $(C_i,C_{i,\geq0})$ and iterate. We call all pairs $(C,C_{\geq0})$ obtained this way the {\it iterated boundary components} of $(X,X_{\geq 0})$. 
\end{definition}

Let $X_{\reg}:=X\setminus \Sing(X)$ be the smooth locus of $X$. By a rational (or meromorphic) {\it differential $d$-form} on $X$, we mean a differential $d$-form on an open subset of the smooth manifold $X_{\reg}$. This is a local section of the canonical bundle $\bigwedge^{\dim X} T^{\ast} X_{\reg}$, the line bundle given by the top exterior power of the cotangent bundle of $X_{\reg}$.  We now define the {\it Poincar\'e residue} of a differential form.

\begin{definition}[{\bf Residue of a form}]
Let $\omega$ be a differential $d$-form on $X$ with a pole of order one along
a prime divisor $Y\subset X$ not contained in the singular locus $\Sing(X)$. A pole of order one is
called a {\it simple} pole. Locally at a smooth point of $Y\setminus \Sing(X)$, we have
\begin{equation*}
    \omega=\eta\wedge \frac{df}{f}+\eta',
\end{equation*}where $\eta$ is a $(d-1)$-form, $\eta'$ is a $d$-form, both having no poles along $Y$, and $f$ is a local coordinate which vanishings to order one along $Y$. The {\it (Poincar\'e) residue} of $\omega$ along $Y$ is defined as the $(d-1)$-form on $Y$ given by the restriction
\begin{equation*}
\Res_{Y} \omega = \eta_{|_Y},
\end{equation*}
not depending on the choices of $\eta, \eta', f$.
\end{definition}

\begin{example}
Let $\omega(z)=\frac{g(z)}{h(z)}dz_1\wedge \cdots \wedge dz_n$ have a simple pole along $\{z_j=0\}$, so that $h=z_j\cdot \frac{\partial h}{\partial z_j}$.
Then 
\[
\Res_{\{z_j=0\}} \omega =(-1)^{n-j} \frac{g}{\partial h/\partial z_j} dz_1\wedge \cdots \wedge\widehat{dz_j}\wedge\cdots \wedge dz_n,
\]
where $\widehat{dz_j}$ means that this differential is omitted. Any choice of $j$ would give cohomologous residue forms on $Y$.
\end{example}

Roughly speaking, the notion of a positive geometry is describing the phenomenon of the existence of a unique top-degree rational differential form $\Omega_X$ on $X$, called the {\it canonical form}, such that it is ``compatible" with the semi-algebraic set $X_{\geq0}$ and its iterated boundaries in the sense that  
\begin{enumerate}
    \item We recover exactly all the iterated boundary components $(C,C_{\geq0})$ of $(X,X_{\geq0})$ by iteratedly taking the poles and residues of $\Omega_X$.
    \item In each iteration, the dimension of the iterated boundary component decreases by one. When we reach the points, the iterated residue of $\Omega_X$ is a $0$-form on the point that equals $\pm 1$, depending on the orientation. 
    \item There is a unique such differential form $\Omega_X$ on $X$. Furthermore, on each iterated component $(C,C_{\geq0})$, the iterated residue $\Omega_C$ of $\Omega_X$ is also the unique form on $C$ satisfying the previous two properties on $C$. 
\end{enumerate}

We now give the formal definition in the case when $X$ and all iterated boundaries of $X$ are normal. In this case, $(X, X_{\geq0}, \Omega_X)$ is called a {\it normal positive geometry}. We omit the more complicated general definition without normality, which involves possibly taking normalizations when taking iterated boundaries, because our result will show the stronger property of normal positive geometry. 

\begin{definition}[\bf Normal Positive Geometry]\label{def: positive geom}
Let $(X,X_{\geq0})$ be as in Definition \ref{def:iterated boundary}. Assume that $X$ and all iterated boundaries of $X$ are normal. A {\it normal positive geometry} is a triple $(X,X_{\geq 0}, \Omega_X)$, where $\Omega_X$ is the unique rational $d$-form on $X$, called the {\it canonical form}, satisfying the following recursive property. 

\begin{enumerate}
\item[(i)] When $d=0$, $X=X_{\geq 0}$ is a point and $\Omega_X=\pm 1$,
depending on orientation. 

\item[(ii)] When $d>0$, let $(C_i,C_{i,\geq0})$ for $i=1,2,\cdots,r$ be the boundary components of $(X,X_{\geq0})$ as in Definition \ref{def:iterated boundary}. The form $\Omega_X$ has simple poles exactly along the boundary components $\bigcup C_i$ and nowhere else. Furthermore, each $(C_i,C_{i,\geq 0},\Omega_{C_i})$ is a positive geometry with the unique canonical form $\Omega_{C_i}$ and the residue $\Res_{C_i}\Omega_X$ coincides with the canonical form $\Omega_{C_i}$.
\end{enumerate}

\end{definition}

\begin{remark}\label{uniqueness of canonical form}
    The uniqueness of the canonical form (on $X$ and on all its iterated boundary components, respectively) is guaranteed by regularity in codimension $1$ and rationality of $X$ and of all its iterated algebraic boundaries. If two rational top forms both satisfy the recursive definition, their difference must be zero because it is a holomorphic top form on $X_{\reg}$, which has geometric genus $0$, because it is rational \cite[II.8.19]{Hartshorne}.

\end{remark}

It is  remarkable that such a canonical form exists on a variety. It is generally very difficult to prove that a given differential form is the canonical form, as the definition requires to understand the complex and real geometry of all iterated boundary components and how to take poles and residues of the canonical form on each iterated boundary component. All known examples of positive geometries are proved under a case-by-case basis. Prototypical examples include  projective polytopes and their blowups \cite{brauner2024wondertopes}, the Deligne-Mumford compactification $\bar{\mathcal{M}_{0,n}}$ \cite{Lam24}, and the Grassmannians\cite{AHBCGPT16}.

One can define a semialgebraic subset $(M_{0,n})_{\geq 0}$ in such a way that $(\overline{\mathcal M}_{0,n}, (\mathcal M_{0,n})_{\geq 0})$ is a positive geometry with a degree $n-3$ differential form \emph {Parke-Taylor form} $\Omega_{0,n}$; see \cite{Lam24} for detailed descriptions of these spaces. It is worth mentioning that the name of this form is inspired by the {\it Parke-Taylor} formula for the $\mathrm{MHV}$ gluon scattering amplitudes in the tree-level Yang-Mills theory. Gluons are massless particles with helicities $\pm 1$ (in units of $\hbar$) appearing in quantum chromodynamics, the fundamental gauge theory of the Standard Model. Employing the spinor helicity formalism, one can express the scatterning amplitude of $n$ gluons 
with total helicity $h=\sum_{i=1}^n h_i = n-4$, called $\mathrm{MHV}$ amplitude $A(1^{+},\ldots, i^{-}, \ldots, j^{-},\ldots, n^{+})$, where only the particles $i$ and $j$ have negative helicities. The point is that this quantity {\it can be extracted} from the canonical form on $\Gr(2,n)$.

In later subsections, we present more details in the case of Grassmannians as a positive geometry, as it will be useful for our purpose.

\subsection{Positive Grassmannians and Positroid Varieties}\label{Positive Grassmannians and Positroid Varieties}
The Grassmannians are known to be positive geometries and its iterated boundary components are given by closed positroid varieties and positroid cells \cite{AHBCGPT16}.

\begin{definition}[\bf Total Nonnegativity on Grassmannians] \label{Total Nonnegativity on Grassmannians}
    The {\it positive} (resp., {\it totally nonnegative}) {\it Grassmannian} $\Gr_{>0}(k,n)$ (resp., $\Gr_{\geq 0}(k,n)$) is the semi-algebraic subset of the real points of $\Gr(k,n)$ such that all Pl\"{u}cker coordinates are positive (resp., nonnegative).

    For a subvariety $V$ of $\Gr(k,n)$, we define its positive part to be $V_{>0}:=V(\mathbb{R})\cap \Gr_{>0}(k,n)$ and its totally nonnegative part to be $V_{\geq0}:=V(\mathbb{R})\cap \Gr_{\geq0}(k,n) $.
\end{definition}

The Grassmannian variety $\Gr(k,n)$ is a classical object of study in algebraic geometry. A comprehensive expository of positroid varieties can be found in \cite{lam2015,speyer2024richardsonvarietiesprojectedrichardson}. The Grassmannians have a stratification by {\it positroid varieties} \cite{KLS13} that restricts to a well behaved stratification of the totally nonnegative Grassmannian by {\it positroid cells} \cite{postnikov2006totalpositivitygrassmanniansnetworks}. It is worth noting that the positroid stratification refines the Schubert (and Richardson) stratification and it is the finest known stratification of Grassmannians having good geometric and topological properties. The well-known matroid stratification \cite{Gelfand_Goresky_MacPherson_Serganova_1987} for instance has topologically badly behaved cells. 

\begin{theorem}\cite{KLS13}
    There exists a family of reduced, irreducible, normal, Cohen-Macaulay closed subvarieties $\Pi_f$ called {\it (closed) positroid varieties} with the following property: inside each closed positroid variety, we have an open subvariety $\Pi_f^\circ$ called open positroid subvarieties, which  form a stratification of the Grassmannian $$\Gr(k,n)=\bigsqcup \Pi_f^\circ,$$
    whose totally nonnegative parts form a stratification of the totally nonnegative Grassmannian 
    $$\Gr_{\geq0}(k,n)=\bigsqcup \Pi_{f,\geq0}^\circ.$$
\end{theorem}

\begin{example}\label{ideal of positroid}
We give an explicit description of positroid varieties in the case $k=3$ using cyclic rank matrices. All positroid varieties in $\Gr(3,n)$ correspond to imposing a collection of cyclic rank conditions on the matrix $M$ representing a point $[M]\in \Gr(3,n)$. By a cyclic rank condition, we mean that we specify a cyclic interval $I$ in $[n]$ and insist that the submatrix $M_I$ of $M$ has rank at most $0,1,$ or $2$. The rank conditions on $M_I$ can be translated into the vanishing of Pl\"ucker coordinates as follows.
\begin{enumerate}
    \item $\rk M_I\leq 2$ if and only if $p_{J}=0$ for all $J\subset I$.
    \item $\rk M_I\leq 1$ if and only if $p_{J}=0$ for all $\#(J\cap I)\geq 2$.
    \item $\rk M_I=0$ if and only if $p_{J}=0$ for all $J\cap I\not= \emptyset$.
\end{enumerate}

It was shown in \cite{KLS13} that the vanishing of these Pl\"ucker coordinates cut out the positroid varieties as a subscheme of $\Gr(3,n)$.
\end{example}

\begin{remark}\label{rmk:positroid as point config}
The cyclic rank conditions translate to the perspective of point configurations in $\mathbb{P}^2$ under the Gelfand-MacPherson correspondence from classical invariant theory.

For any $3$ by $n$ matrix $M$, each nonzero column gives a point in $\mathbb{P}^2$ after projectivization. The nonzero columns of $M$ give a configuration of points in $\mathbb{P}^2$ up to $GL(3)$. Under this description, we have that
\begin{enumerate}
    \item $\rk M_I\leq 2$ if and only if nonzero columns in $I$ give collinear points in $\mathbb{P}^2$.
    \item $\rk M_I\leq 1$ if and only nonzero columns in $I$ give colliding points in $\mathbb{P}^2$.
    \item $\rk M_I=0$ if and only columns in $I$ are all zero and do not give points after projectivization.
\end{enumerate}
    Throughout the paper, we will repeatedly appeal to the Gelfand-MacPherson correspondence to interpret set-theoretically subschemes of $\Gr(3,n)$.   
\end{remark}

\subsection{Positroids, Affine permutations, and Grassmann necklaces}

Besides Cyclic rank conditions, which we introduced through an example in the last subsection, we now review other combinatorial objects that index positroid varieties.
\begin{definition}\label{def: necklace}
A $(k,n)$-{\it Grassmann necklace} is a collection of $k$-element subsets $\mathcal I = (I_1,I_2,\ldots, I_n)$ such that, for each $a\in [n]$, the following conditions hold:
\begin{enumerate}
\item $I_{a+1}=I_a$ if $a\notin I_a$;
\item $I_{a+1}=I_a\setminus \lbrace a\rbrace \cup \lbrace a'\rbrace$ if $a\in I_a$ for some $a'\in [n]$.
\end{enumerate}
\begin{comment}
We say that $\mathcal I\leq \mathcal J$ if $I_a\leq J_a$ for all $a\in [n]$. 
\end{comment}
\end{definition}

\begin{definition}\label{def: bounded affine perm}
Let $n\geq 2$. An {\it affine permutation}
is a bijection $f:\ZZ\rightarrow \ZZ$ satisfying the periodicity condition $f(i+n) = f(i)+n$ for all $i\in \ZZ$. A {\it bounded affine permutation} is an affine permutation $f$ with the additional bounded condition $i\leq f(i)\leq i+n$. In its {\it window notation}, an affine permutation is denoted $[f(1),f(2),\ldots,f(n)]$. We often write $k+n$ as $\bar{k}$ for simplicity of notations.

For a bounded affine permutation $f$ with period $n$, the sum 
$\sum_{i=1}^n (f(i)-i)=kn$ for some $0\leq k\leq n$. Let $\mathcal B(k,n)$ denote the set of bounded affine permutations $f$ such that $\sum_{i=1}^n (f(i)-i)=kn$. 
\end{definition}

\begin{comment}
Let $I=\lbrace i_1<i_2<\cdots < i_k\rbrace$ and $J = \lbrace j_1<j_2<\cdots < j_k \rbrace$ be two $k$-element subsets of $[n]$. We say that $I\leq J$ if $i_r\leq j_r$ for every $1\leq r\leq k$. This defines a partial order on the set of $k$-element subsets of $[n]$. Let $\leq_a$ be the cyclically rotated order $a<a+1<\cdots < n < 1<\cdots <a-1$ on $[n]$. This new order induces a new partial order $I\leq_a J$ between $k$-element subsets of $[n]$. 
\end{comment}

There exist bijections between any two of the following
\begin{itemize}
    \item Bounded affine permutations in $\mathcal B(k,n)$.
    \item $(k,n)$-Grassmann necklaces.
    \item positroids of rank $k$ on $[n]$.
\end{itemize}
Each of these combinatorial objects can be used to index positroid varieties in $\Gr(k,n)$. We provide an example of such a bijection. Details about these families and bijections among them can be found in \cite{lam2015}.

Given $f\in \mathcal B(k,n)$, we define a sequence $\mathcal I(f)=(I_1,\ldots,I_n)$ of $k$-element subsets by
\[
I_a = \lbrace f(b) \ | \ b<a \mbox{ and } f(b)\geq a\rbrace \mbox{ mod } n,
\]
where $\mbox{mod } n$ means that one takes representatives in $[n]$.

\subsection{The Positive Grassmannian is a Positive Geometry}\label{The Positive Grassmannian is a Positive Geometry}

It was discovered later that the positroid stratification exactly gives the iterated boundary components of the totally nonnegative Grassmannian $\Gr(k,n)$ in the sense of Definition \ref{def:iterated boundary} and, with a suitable choice of the canonical form, makes $(\Gr(k,n),\Gr_{\geq 0}(k,n), \Omega(k,n))$ a positive geometry \cite[\S 7]{AHBCGPT16}. 

\begin{theorem}\cite[\S 7]{AHBCGPT16}\label{Grassmannian is a positive geometry}
The triple $(\Gr(k,n),\Gr_{\geq 0}(k,n), \Omega(k,n))$ a positive geometry, where $\Omega(k,n)$ is given by
\begin{equation*}
    \frac{vol(\Gr(k,n))}{\prod_{i=1}^n p_{i,i+1,\cdots ,i+k-1}}.
\end{equation*}

The iterated boundary components are exactly the positroid varieties and their totally nonnegative parts 
$(\Pi_f,\Pi^\circ_{f,\geq 0})$. 
\end{theorem}

\begin{example} 
 The canonical form on $\Gr(2,n)$ may be written as follows. 
\begin{equation}\label{equ:canonical form on G(2,n)}
\Omega(\Gr(2,n)) = \frac{vol(\Gr(2,n))}{\prod_{i=1}^n p_{i,i+1}},
\end{equation}
where $p_{i,i+1}$ are the Pl\"{u}cker coordinates on $\Gr(2,n)$ with $p_{n,n+1}=p_{n,1}$, and $vol(\Gr(2,n))$ is the standard $\mathrm{GL}(2,n)$-invariant measure on $\Gr(2,n)$. We now describe the invariant measure in this case. Let $M\in \Gr(2,n)$, denote with $M_i$ the rows of $M$ and let $d M_i$ be the row vector whose $j$th entry is the differential $d m_{ij}$, where $m_{ij}$ is the $(i,j)$th entry of $M_i$. 
Let $\langle M_1M_2 d^{n-2} M_i\rangle$ be the formal determinant of the $n\times n$ matrix with rows $M_1, M_2$ and $(n-2)$-copies of the row $d M_i$. This determinant is a differential of degree $(n-2)$ in $d m_{ij}$. Then the invariant measure on $\Gr(2,n)$ is 
\[
vol(\Gr(2,n)) = \frac{\prod_{i=1}^2 \langle M_1M_2 d^{n-2} M_i\rangle}{((n-2)!)^2}
\]

Similar formulas give the invariant measure and the canonical form on any $\Gr(k,n)$ \cite[\S 5.5.2]{AHBL17}. On the coordinate chart $U=\lbrace p_{1,2}=1\rbrace$, where we fix the identity matrix in the first two columns and take the rest of the matrix coordinates to be $x_{1j},x_{2j}$ (for $1\leq j\leq n-2$), the $2(n-2)$-form
\eqref{equ:canonical form on G(2,n)} becomes
\[
\Omega(\Gr(2,n)_{\geq 0})_{|U} = \frac{\bigwedge_{j=1}^{n-2} dx_{1j}\wedge \bigwedge_{j=1}^{n-2} dx_{2j}}{\prod_{i=1}^n p_{i,i+1}},
\]
where $p_{1,2}=1, p_{1,j} = x_{2j}, p_{2,j} = -x_{1j}$, and so on. 

\end{example}

\subsection{Actions by $S_n$, $C_n$, $D_n$, and $(\mathbb{C}^*)^{n-1}$}\label{actions}
The symmetric group $S_n$ acts on the Grassmannian $\Gr(k,n)$ by permuting columns. However, $S_n$ do not act on the totally nonnegative Grassmannian $\Gr_{\geq0}(k,n)$, as permuting columns might change the sign of Pl\"ucker coodinates.

Instead, the cyclic group $C_n$ acts on the totally nonnegative Grassmannian $\Gr_{\geq0}(k,n)$. A generator of $C_n$ acts by sending a matrix with columns $v_1,\cdots, v_n$ to the matrix with columns $(-1)^{k-1}v_n, v_1,\cdots,v_{n-1}$. The sign $(-1)^{k-1}$ ensures that the resulting matrix has all positive maximal minors.

The positive geometry $(\Gr(k,n),\Gr_{\geq0}(k,n))$ has a natural $C_n$-symmetry. The canonical form $\Omega_{\Gr(k,n)}$ is invariant under the $C_n$ action, but not under the $S_n$ action. Also, $C_n$, but not $S_n$, acts on the collection of positroid varieties. These two facts are mirrored by the fact that the canonical form determines iterated boundary divisors by iteratedly taking poles and residues. 

In this paper, we focus on the case $k=3$. The above sign disappears. Also, it is worth noting that there is actually the dihedral $D_n$-symmetry, because flipping all columns (the flip maps column $i$ to column $n+1-i$, for $1\leq i\leq n$) also preserves positivity.

Finally, notice that the algebraic torus $(\mathbb{C}^*)^n$ acts on $\Gr(k,n)$ by rescaling columns (matrix multiplication from the right), and the diagonal subgroup $\mathbb{C}^*$ acts trivially. The totally nonnegative Grassmannian $\Gr_{\geq0}(k,n)$ and its decomposition into positroid cells $\Pi_{f,\geq 0}^{\circ}$ are invariant under the positive torus $(\mathbb{R}_{>0})^{n-1}$. And it turns out that each positroid variety is invariant under the $(\mathbb{C}^*)^{n-1}$-action by rescaling columns. 

For any subvariety of $\Gr(k,n)$ invariant under the $(\mathbb{C}^*)^{n-1}$-action, we can interpret it as a closed condition on the corresponding point configuration in $\mathbb{P}^k$ under the Gelfand-MacPherson correspondence as in Remark \ref{rmk:positroid as point config}.

\section{The ABCT Variety} \label{intro V(k,n)}

Consider the Veronese map $\theta: \PP^1\longrightarrow \PP^{k-1}$ induced
by the map $\theta: \CC^2\longrightarrow S^{k-1}\CC^2\cong \CC^{k}$, 
sending $v\mapsto v\otimes v\otimes \cdots \otimes v$. The representation 
$S^{k-1}\CC^2$ of $\mathrm{GL}(2)$ gives a homomorphism $\theta: \mathrm{GL}(2)\longrightarrow \mathrm{GL}(k)$. The Veronese map is equivariant 
with respect to the actions of $\mathrm{GL}(2)$, i.e. $\theta(g\cdot v) = \theta(g)\cdot \theta(v)$. The map $\theta\times \cdots \times \theta: (\CC^2)^{\times n}\longrightarrow (\CC^k)^n$ descends to a {\it rational Veronese map}
$\theta_{k-1}: \Gr(2,n)\dashrightarrow \Gr(k,n)$. In matrix coordinates, this map is given 
by the following formula
\[
\begin{pmatrix}
a_{11} & a_{12} & \cdots & a_{1n} \\
a_{21} & a_{22} & \cdots & a_{2n} \\
\end{pmatrix}
\mapsto 
\begin{pmatrix}
a_{11}^{k-1} & a_{12}^{k-1} & \cdots & a_{1n}^{k-1} \\
a_{11}^{k-2}a_{21} & a_{12}^{k-2}a_{22} & \cdots & a_{1n}^{k-2}a_{2n} \\
\vdots & \vdots & \vdots & \vdots \\
a_{21}^{k-1} & a_{22}^{k-1} & \cdots & a_{2n}^{k-1} \\
\end{pmatrix}.
\]
The map $\theta_{k-1}$ between the two Grassmannians is not defined for instance on $\langle e_i,e_j\rangle\in \Gr(2,n)$,
because the rank of the matrix in the image is lower than $k$. 

\begin{definition}\label{def:ABCT variety}
The {\it ABCT variety} $V(k,n)$ is the Zariski closure of the image of $\theta_{k-1}$, i.e. $V(k,n)= \overline{\theta_{k-1}(\Gr(2,n))}\subset \Gr(k,n)$.
\end{definition}

\begin{conjecture}[{\cite[Conjecture~4.10]{Lam24}}]\label{conj lam}
The ABCT variety $(V(k,n),V(k,n)_{\geq 0})$, where $V(k,n)_{\geq 0}=\overline{\theta_{k-1}(\Gr(2,n)_{>0})}$ denotes the analytic closure, is a positive
geometry. 
\end{conjecture}

\begin{remark}
The notation of positive geometry is a condition about boundaries. The above conjecture is nontrivial in the sense that the face structure of $V(k,n)_{\geq 0}$ and $\Gr(2,n)_{\geq 0}$
differ, even though we have a holomorphic map on the interiors. 
For instance, $\Gr(2,n)_{\geq 0}$ has $0$-dimensional faces that are the $\binom{n}{2}$ torus fixed points $W_{ij} = \langle e_i,e_j\rangle\in\Gr(2,n)$. But these torus fixed points are in the indeterminacy locus of the rational map $\theta_{k-1}: \Gr(2,n)\dashrightarrow \Gr(k,n)$, for $k>2$.  Indeed, we will see in our main theorem that $V(3,n)_{\geq0}$ has different boundary structures than $\Gr_{\geq0}(2,n)$.
\end{remark}

In the case $k=3$, the geometric properties of $V(3,n)$ as a subvariety of $\Gr(3,n)$ have been understood \cite{ARS24}. In this paper, we leverage their results to show that $V(3,n)$ is a positive geometry. We briefly review results from \cite{ARS24} that will be relevant for our purpose.

\begin{theorem}[\cite{ARS24}]\label{prop:interpretation of V}
The ABCT variety $V(3,n)$ is reduced, irreducible, and Cohen-Macaulay of expected codimension $n-5$ in $\Gr(3,n)$. Moreover, let $\theta:\operatorname{Mat}_{3\times n}\to  \operatorname{Mat}_{6\times n}$ be defined by 
$$\begin{pmatrix} 
x_{11} & x_{12} & \dots & x_{1n} \\
x_{21} & x_{22} & \dots & x_{2n} \\
x_{31} & x_{32} & \dots & x_{3n}
\end{pmatrix}\mapsto
\begin{pmatrix} 
x_{11}^2 & x_{12}^2 & \dots & x_{1n}^2 \\
x_{21}^2 & x_{22}^2 & \dots & x_{2n}^2 \\
x_{31}^2 & x_{32}^2 & \dots & x_{3n}^2 \\
x_{11}x_{21} & x_{12}x_{22}&\dots & x_{1n}x_{2n}\\
x_{11}x_{31} & x_{12}x_{32}&\dots & x_{1n}x_{3n}\\
x_{21}x_{31} & x_{22}x_{32}&\dots & x_{2n}x_{3n}\\
\end{pmatrix}.
$$
Then, one has the scheme-theoretic equality
$
V(3,n)=\left\{ [M] \in \Gr(3,n) \,|\, \operatorname{rk}\theta(M) < 6 \right\}.$

\end{theorem}

\begin{theorem}[{\cite[Corollary 4.2]{ARS24}}] \label{Ideal of ABCT}
 If $n\geq 6$. Then $V(3,n)$ is cut out as a scheme by the following quartic equations in the Pl\"ucker coordinates on $\Gr(3,n)$:
    \[p_{i_1,i_2,i_3}p_{i_1,i_5,i_6}p_{i_2,i_4,i_6}p_{i_3,i_4,i_5}-p_{i_2,i_3,i_4}p_{i_1,i_2,i_6}p_{i_1,i_3,i_5}p_{i_4,i_5,i_6} = 0,\] where $i_1<i_2<i_3<i_4<i_5<i_6$ vary through all the $6$-element subsets of $\{1,\dots,n\}$.   
\end{theorem}

\begin{remark}\label{S_n symmetry}
The above theorems provide the description of $V(3,n)$ as a subscheme of $\Gr(3,n)$ by specifying the radical vanishing ideal in terms of matrix and Pl\"ucker coordinates. 

The description in Pl\"ucker coordinates makes it manifest that $V(3,n)$ is invariant under the $S_n$-action by permuting columns and the $(\mathbb{C}^*)^{n-1}$-action by column scaling.

The description in matrix coordinates works well after fixing any standard chart of $\Gr(3,n)$ by fixing some minor of $M$ to be the identity matrix. Then, the rank condition translates to an explicit ideal in the matrix coordinates cutting $V(3,n)$ out in that chart.
\end{remark}

\begin{remark}\label{rmk:point configuration for V}
The ABCT variety is invariant under the $(\mathbb{C}^*)^{n-1}$-action by column scaling. Following the Gelfand-MacPherson correspondence as in Remark \ref{rmk:positroid as point config}, the description of $V(3,n)$ in matrix coordinates translates to a set-theoretic description of $V(3,n)$ in terms of the point configurations in $\mathbb{P}^2$.
We see that $[M]\in \Gr(3,n)$ lies in $V(3,n)$ if and only if the configuration of points in $\mathbb{P}^2$ given by the nonzero columns of $M$ lie on a common conic in $\mathbb{P}^2$: if $\theta(M)$ is not full-rank, any nonzero vector in the kernel of $\theta(M)$ is a six-tuple of coefficients for a conic in $\mathbb{P}^2$ containing all points corresponding to nonzero columns of $M$.
\end{remark}

We can naturally write the map $\theta_{k-1}$ in terms of Pl\"{u}cker coordinates. 

\begin{proposition}\label{prop:map in terms of Pluckers}
The map $\theta_{k-1}$ can be lifted between Pl\"{u}cker projective spaces. Namely, $\theta_{k-1}$ can be defined in terms of Pl\"{u}cker coordinates as follows. 
\[
p_{i_1,\cdots,i_k} = \prod_{ j<\ell} p_{i_j,i_{\ell}},
\]
where $p_{j,{\ell}}$ (and resp., $p_{i_1,\cdots,i_k}$) is a Pl\"{u}cker coordinate on $\Gr(2,n)$ (and resp., $\Gr(k,n)$).
\begin{proof}
A $k\times k$ minor of $\theta_{k-1}(M)$ corresponding to the coordinate $p_{i_1,\cdots,i_k}$
is a Vandermonde determinant, where each factor is the quadratic polynomial $(a_{1,i_j}a_{2,i_\ell}-a_{1,i_\ell}a_{2,i_j})=p_{i_j,i_\ell}$. 
%More invariantly, the minors must be invariant under the $\mathrm{GL}(2)$-actions induced from the ones on the $n$ copies of $\CC^2$. So they must be divisible by these $2\times 2$ determinants. 
\end{proof}
\end{proposition}

\begin{remark}\label{D_n symmetry}
From the above proposition, it is clear that the rational map $\theta$ is $S_n$-invariant. Also, it sends points in $\Gr_{\geq0}(2,n)$, if the image is well-defined, to points in $\Gr_{\geq0}(3,n)$. In particular, it sends every point in $\Gr_{>0}(2,n)$, the subset where all Pl\"ucker coordinates are strictly positive, to $\Gr_{>0}(3,n)$. Thus, $V(3,n)\cap \Gr_{\geq0}(3,n)$ is invariant under the $D_n$-action by cyclically rotating and flipping the columns. Furthermore, the $D_n$-symmetry of the canonical form on $\Gr(2,n)$ passes to its pushforward via $\theta$.

\end{remark}

\section{Candidate Canonical Form by Pushing-forward}\label{section: candidate form}

\subsection{Pushforwards of differentials}
In the category of complex algebraic varieties, the pushforward of a form is defined as follows \cite[\S 3.9]{Lam24}. 

\begin{definition}[\bf Pushforward]\label{def: pushforward Lam}
Suppose we have a dominant rational map $f: X\rightarrow Y$ of complex varieties, 
generically of degree $d$. 
Let $\omega$ be a rational $r$-form on $X$. 
Let $W\subset Y$ be an open set such that $f^{-1} = U_1\sqcup U_2\sqcup \cdots \sqcup U_d$ is a disjoint union of open sets $U_1,U_2,\ldots, U_d\subset X$ where $f_{|U_i}: U_i\rightarrow W$ is an isomorphism. Then the {\it pushforward} of $\omega$ through $f$ is
\[
f_{*}\omega = \sum_{i=1}^d ((f_{|U_i})^{-1})^{*}\omega_{|U_i},
\]
where the form on the right-hand side is a rational form on $W$, extended to a rational form on $Y$. 
\end{definition}

\begin{remark}
We shall use Definition \ref{def: pushforward Lam} on local affine charts of complex algebraic varieties. It is easy to generalize this definition for schemes over arbitrary rings. However since we do not need such a general setup, we will not introduce it.  
\end{remark}

\subsection{Pushforward of the canonical form on $\Gr(2,n)$ to $V(3,n)$ via the rational map $\theta$} 
Consider the affine open $\{p_{01}p_{12}p_{02}\not=0\}\subset \Gr(2,n)$ which gets mapped inside $\{p_{012}\not=0\}\subset \Gr(3,n)$ under $\theta$. We fix coordinates \begin{equation}\label{eq:matrices_chart}
    \begin{pmatrix}
    1&0&a&x_1&x_2 &\cdots &x_{n-3}\\
    0&1&b&y_1& y_2&\cdots &y_{n-3}
\end{pmatrix} \qquad \text{and} \qquad \begin{pmatrix}
    1&0&0&\alpha_1&\cdots&\alpha_{n-3}\\
    0&0&1&\beta_1&\cdots&\beta_{n-3}\\
    0&1&0&\gamma_1&\cdots &\gamma_{n-3}
\end{pmatrix}
\end{equation}
on $\Gr(2,n)$ and $\Gr(3,n)$, respectively. Then, the restriction of $\theta$ to these affine opens is described by a ring homomorphism as follows. From now on, set $m\coloneqq n-3$. Define the rings $$S\coloneqq k[a^{\pm 1},b^{\pm 1},x_i,y_i \,|\, i\in [m]] \qquad \text{and}\qquad R' \coloneqq k[\alpha_i , \beta_i,\gamma_i \,|\, i\in [m]].$$ Note that 
$S$ is the coordinate ring for $\Gr(2,n)\cap \{p_{01}p_{12}p_{02}\not=0\}$. Let $I_{\det}$ denote the ideal in $R'$ generated by the maximal minors of the matrix \begin{equation}\label{variant matrix}
   \theta'(M) = \begin{pmatrix}
        \alpha_1\beta_1&\dots& \alpha_{m}\beta_{m}\\
    \alpha_1\gamma_1&\dots& \alpha_{m}\gamma_{m}\\
    \beta_1\gamma_1&\dots& \beta_{m}\gamma_{m}
\end{pmatrix}
\end{equation}
From \cite[Corollary~3.8]{ARS24}, we know that the ideal $I_{\det}$ is prime in $R'$. Hence, the quotient $R\coloneqq R'/I_{\det}$ is a domain, and serves as the coordinate ring for $V(3,n)\cap \{p_{012}\not=0\}$. Let $\Tilde{\varphi}: R' \to S $ be the ring homomorphism defined by \[\alpha_i\,\mapsto\, x_i^2-\frac{a}{b}x_iy_i, \qquad \beta_i\,\mapsto\, \frac{1}{ab}x_i y_i,\qquad\gamma_i\,\mapsto\, y_i^2-\frac{b}{a}x_iy_i ,\qquad i\in [m].\]

\begin{proposition}\label{obvious kernel}
    We have $\Tilde{\varphi} (I_{\det})=0$. Therefore, $\Tilde{\varphi}$ induces a ring homomorphism $R \to S$, giving the restriction of $\theta$ to $\mathrm{Spec}(S)=\{p_{01}p_{12}p_{02}\not=0\}\subset \Gr(2,n)\to \mathrm{Spec}(R)= \{p_{012}\not=0\}\cap V(3,n)$.
\end{proposition}
\begin{proof}
    It is enought to verify that for any choice of $i,j,k\in [m]$, $\varphi$ maps the determinant of the $3\times 3$ matrix 
    \begin{equation}\label{eq:Mijk}
    \theta'(M)_{ijk} = \begin{pmatrix}
        \alpha_i\beta_i&\alpha_j\beta_j&\alpha_k\beta_k\\
        \alpha_i\gamma_i&\alpha_j\gamma_j&\alpha_k\gamma_k\\
        \beta_i\gamma_i&\beta_j\gamma_j&\beta_k\gamma_k\\
  
    \end{pmatrix}
    \end{equation}
    to zero. This is a straightforward computation.
\end{proof}

\begin{proposition}\label{ring map injective}
The map $\varphi:R\to S$ is injective with an explicit partial inverse.
\end{proposition}
\begin{proof}
    We construct a ring extension $i:R\hookrightarrow\Tilde{R}$ and a ring homomorphism $\psi:S\to \Tilde{R}$ such that $\psi\circ\varphi=i$ as ring homomorphisms $R\hookrightarrow\Tilde{R}$. Recall that $R$ is a domain and let us first define $\Tilde{R}$. 

 For $i\not=j\in [m]$, define the elements $f_{ij},g_{ij}\in\operatorname{Frac}(R)$ given by 
 $$f_{ij}\coloneqq\frac{\alpha_i\alpha_j}{\beta_i\beta_j}\frac{\beta_i\gamma_j-\beta_j\gamma_i}{\gamma_i\alpha_j-\gamma_j\alpha_i} \qquad \text{and} \qquad  g_{ij}\coloneqq\frac{\gamma_i\gamma_j}{\beta_i\beta_j}\frac{\alpha_i\beta_j-\alpha_j\beta_i}{\gamma_i\alpha_j-\gamma_j\alpha_i}.$$
For $i,j,k\in [m]$ distinct, by direct computation, we verify that $f_{ij}=f_{ik}$ and $g_{ij}=g_{ik}$ modulo the determinant of the matrix $\theta'(M)_{ijk}$ in~\eqref{eq:Mijk} in the ideal $I_{\det}$ of $R'$. As $f_{ij}=f_{ji}$ and $g_{ij}=g_{ji}$, we have well defined elements $f=f_{ij}$ and $g=g_{ij}$ for any $i\not=j\in [m]$ in $\operatorname{Frac}(R)$. Importantly, we have 
$$\alpha_i+\beta_if=-\frac{\alpha_i\beta_i}{\gamma_i}g \qquad\text{and}\qquad \gamma_i+\beta_i g=-\frac{\beta_i\gamma_i}{\alpha_i}f \quad \text{for any }i\in [m].$$
Let $R\to \Tilde{R}$ be the ring extension obtained by adjoining the square roots of $f,g,\beta_i,-\frac{\gamma_i}{\alpha_i}$ (for $i\in [m]$) and localizing so that all these elements become invertible. The ring
thus $\Tilde{R}$ contains special elements $$\sqrt{f},\,\sqrt{g},\,\sqrt{\beta_i},\,\sqrt{-\frac{\gamma_i}{\alpha_i}}.$$
When multiplying these ``square root" elements, we shall abuse notation and write it as the square root of the product. One should, however, be aware of the following caveat: \begin{equation}\label{eq:caveat}
\sqrt{\beta_i^2}=\sqrt{\beta_i}\sqrt{\beta_i}=\beta_i, \qquad \text{but}\qquad \sqrt{\left(\frac{\alpha_i}{\gamma_i}\right)^2}=\sqrt{-\frac{\alpha_i}{\gamma_i}}\sqrt{-\frac{\alpha_i}{\gamma_i}}=-\frac{\alpha_i}{\gamma_i}.
\end{equation}
We define the ring homomorphism $\psi:S\to \Tilde{R}$ on a set of generators as follows.
\begin{equation}\label{eq:psi_inverse}
    a\mapsto \sqrt{f},\qquad b\mapsto \sqrt{g}, \qquad
    x_i \mapsto \sqrt{-\frac{\alpha_i\beta_i}{\gamma_i}g},
    \qquad
    y_i \mapsto \sqrt{-\frac{\beta_i\gamma_i}{\alpha_i}f}, \quad i\in [m].        
\end{equation}

Finally, let us verify that $\psi\circ\varphi=i$. It suffices to check thin on the generators $\{\alpha_i,\beta_i,\gamma_i\,|\,i\in[m]\}.$

\[\begin{aligned}
    \psi\circ\varphi(\alpha_i)
    =\psi\left(x_i^2-\frac{a}{b}x_i y_i\right)
    &=\psi(x_i)^2-\frac{\psi(a)}{\psi(b)}\psi(x_i)\psi( y_i)\\
    &=\left(-\frac{\alpha_i\beta_i}{\gamma_i}g\right)- \sqrt{\frac{f}{\cancel{g}}\left(-\frac{\cancel{\alpha_i}\beta_i}{\cancel{\gamma_i}}\cancel{g}\right)\left(-\frac{\beta_i\cancel{\gamma_i}}{\cancel{\alpha_i}}f\right)}\\
    &=\alpha_i+\beta_i f - \sqrt{\beta_i^2f^2}=\alpha_i,
\end{aligned}
\]

\[\begin{aligned}
    \psi\circ\varphi(\beta_i)
    =\psi\left(\frac{1}{ab}x_iy_i\right)
    &=\sqrt{\left(-\frac{{\alpha_i}\beta_i}{{\gamma_i}}{g}\right)\left(-\frac{\beta_i{\gamma_i}}{{\alpha_i}}f\right)\frac{1}{fg}}
    =\sqrt{\beta_i^2}=\beta_i,
\end{aligned}
\]

\[\begin{aligned}
    \psi\circ\varphi(\gamma_i)
    =\psi\left(y_i^2-\frac{b}{a}x_i y_i\right)
    &=\psi(y_i)^2-\frac{\psi(b)}{\psi(a)}\psi(x_i)\psi( y_i)\\
    &=\left(-\frac{\beta_i\gamma_i}{\alpha_i}f\right)- \sqrt{\frac{g}{\cancel{f}}\left(-\frac{\cancel{\alpha_i}\beta_i}{\cancel{\gamma_i}}{g}\right)\left(-\frac{\beta_i\cancel{\gamma_i}}{\cancel{\alpha_i}}\cancel{f}\right)}\\
    &=\gamma_i+\beta_i g - \sqrt{\beta_i^2g^2}=\gamma_{i}.
\end{aligned}
\]
This concludes the proof.
\end{proof}
For simplicity of notations, set 
\[
A_{ij}:=\beta_i\beta_j(\gamma_i\alpha_j-\alpha_i\gamma_j), \qquad
B_{ij}:=\alpha_i\alpha_j(\beta_i\gamma_j-\gamma_i\beta_j),\qquad 
C_{ij}:=\gamma_i\gamma_j(\alpha_i\beta_j-\beta_i\alpha_j).
        \]
Recall that, by definition of $f$ and $g$, we have $f=f_{ij}=B_{ij}/A_{ij}$ and $g=g_{ij}=C_{ij}/A_{ij}$ and these expressions are independent of the choice of indices $i,j$ in the quotient ring $R=R'/I_{\det}$. 

In the quotient ring $R=R'/I_{\det}$, the determinants of the matrix $\theta'(M)_{ijk}$ are zero. This determinant can be expanded as 
\[P_{ijk}\coloneqq \det(\theta'(M)_{ijk}) = \alpha_i\gamma_iA_{jk}+\beta_i\gamma_iB_{jk}+\alpha_i\beta_i C_{jk}=0.\]
Since the triple $(A_{ij},B_{ij},C_{ij})$ is independent from the index up to a common scalar, we may write $A,B,C$ without indices when taking ratios.

Let $\Omega^{ijk}_{\alpha_i}$ denote the $8$-form obtained by omitting $d\alpha_i$ from the $9$-form $d\alpha_i\wedge d\beta_i\wedge d\gamma_i\wedge \cdots \wedge d\gamma_k$. Here, $1$-forms are pulled-back to $V(3,n)$ via the inclusion map $V(3,n)\to \Gr(3,n)$. The cotangent bundle on $V(3,n)$ is a quotient of that on $\Gr(3,n)$ by relations of the form $dP_{ijk}=0$.
That is, we have the relation $\frac{\partial P_{ijk}}{\partial \alpha_i}d\alpha_i+\cdots+ \frac{\partial P_{ijk}}{\partial \gamma_k}d\gamma_k=0$. Taking the wedge with the wedge product of any $7$ variables, we obtain the relation

\begin{equation*}
   \Omega^{ijk}:=\frac{\Omega^{ijk}_{\alpha_i}}{sgn(\alpha_i)\frac{\partial P_{ijk}}{\partial \alpha_i}}=\frac{\Omega^{ijk}_{\alpha_j}}{sgn(\alpha_j)\frac{\partial P_{ijk}}{\partial \alpha_j}}=\cdots =\frac{\Omega^{ijk}_{\gamma_k}}{sgn(\gamma_k)\frac{\partial P_{ijk}}{\partial \gamma_k}},
\end{equation*}
where the $sgn=sgn^{ijk}$ function is determined by $sgn(\alpha_i)=1$ and alternates in the sequence $\alpha_i, \beta_i, \gamma_i,\alpha_j, \beta_j, \gamma_j, \alpha_k, \beta_k, \gamma_k$. The partial derivatives of $P_{ijk}$ can be explicitly computed, namely, we have
$$\frac{\partial P_{123}}{\partial \alpha_1}=\beta_1C_{23} +\gamma_1 A_{23}=\frac{-\beta_1\gamma_1 B_{23}}{\alpha_1},$$
and the partial derivative with respect to any other variable can be obtained by symmetry.

\begin{lemma} \label{pick 8 out of 9}
 Fix $1<i<n-3$, we have the following identity. 
\begin{align*}
    d\alpha_1 \wedge d\beta_1 \wedge d\gamma_1\wedge  d(\frac{\alpha_i}{\gamma_i})\wedge d\beta_i\wedge d\alpha_{n-3}\wedge d\beta_{n-3}\wedge  d\gamma_{n-3}=\frac{\alpha_i A_{1,n-3}}{\gamma_i}\Omega^{1,i,n-3} =-\frac{1}{\beta_ig}\Omega_{\gamma_i}^{1,i,n-3}=-\frac{\alpha_i^2}{\gamma_i^2\beta_if}\Omega_{\alpha_i}^{1,i,n-3}
\end{align*}
\end{lemma}
\begin{proof}
        First, notice that for any $2\leq i\leq n-4$, we have that 
\begin{align*}  
&d\alpha_1 \wedge d\beta_1 \wedge d\gamma_1\wedge  d(\frac{\alpha_i}{\gamma_i})\wedge d\beta_i\wedge d\alpha_{n-3}\wedge d\beta_{n-3}\wedge  d\gamma_{n-3}\\
&= d\alpha_1 \wedge d\beta_1 \wedge d\gamma_1 \wedge  (\frac{\gamma_id\alpha_i-\alpha_id\gamma_i}{\gamma_i^2})\wedge d\beta_i\wedge d\alpha_{n-3}\wedge d\beta_{n-3}\wedge  d\gamma_{n-3}\\
&=\frac{1}{\gamma_i}\Omega^{1,i,n-3}_{\gamma_i}+\frac{\alpha_i}{\gamma_i^2} \Omega^{1,i,n-3}_{\alpha_i} =-(\frac{1}{\gamma_i}\frac{\partial P_{1,i,n-3}}{\partial \gamma_i}+\frac{\alpha_i}{\gamma_i^2}\frac{\partial P_{1,i,n-3}}{\partial \alpha_i})\Omega^{1,i,n-3}  
\end{align*}
We compute the polynomial
\begin{align*}
    &\frac{1}{\gamma_i}\frac{\partial P_{1,i,n-3}}{\partial \gamma_i}+\frac{\alpha_i}{\gamma_i^2}\frac{\partial P_{1,i,n-3}}{\partial \alpha_i}=\frac{1}{\gamma_i}\frac{\alpha_i\beta_iC_{1,n-3}}{ \gamma_i}+\frac{\alpha_i}{\gamma_i^2}\frac{\beta_i\gamma_iB_{1,n-3}}{ \alpha_i}\\
     &=\frac{\alpha_i\beta_iC_{1,n-3}}{ \gamma_i^2}+\frac{\beta_i\gamma_iB_{1,n-3}}{\gamma_i^2}=-\frac{\alpha_i\gamma_iA_{1,n-3}}{ \gamma_i^2}=-\frac{\alpha_i A_{1,n-3}}{\gamma_i}
\end{align*}
Now, notice that $\Omega^{1,i,n-3} =\frac{\Omega^{1,i,n-3}_{\alpha_i}}{sgn(\alpha_i)\frac{\partial P_{1,i,n-3}}{\partial \alpha_i}}$. Thus, we can replace $d(\frac{\alpha_i}{\gamma_i})$ by $$\frac{\alpha_i A_{1,n-3}}{\gamma_i}/\frac{\beta_i\gamma_i B_{1,n-3}}{\alpha_i}d\gamma_i=\frac{\alpha_i^2}{\gamma_i^2\beta_if} d\gamma_i.$$ By a similar computation, we can replace $d(\frac{\alpha_i}{\gamma_i})$ by $$-\frac{\alpha_i A_{1,n-3}}{\gamma_i}/\frac{\alpha_i\beta_i C_{1,n-3}}{\gamma_i}d\alpha_i=-\frac{1}{\beta_ig} d\alpha_i.$$

We take the expression from Proposition \ref{Canonical form on V(3,n)}, replace $d(\frac{\alpha_i}{\gamma_i})$ by $-\frac{\alpha_i^2}{\gamma_i^2\beta_if} d\gamma_i$ for $2\leq i\leq m-2$ and by $-\frac{1}{\beta_ig} d\alpha_i$ for $m-1\leq i\leq n-4$.

\end{proof}

%We denote the above rational $8$-form by $\Omega^{ijk}$.
As we have the inverse of $\phi:R\to S$ from Proposition~\ref{ring map injective}, we can compute the pushforward of the canonical form on $\Gr(2,n)$ given by 
\begin{equation}\label{eq:can_form_2n_chart}
\frac{da\wedge db\wedge (\bigwedge_{i=1}^m dx_i)\wedge (\bigwedge_{i=1}^m dy_i)}{-ay_{n-3} (ay_1-bx_1) \displaystyle{\prod_{i=1}^{n-4}} (x_iy_{i+1}-x_{i+1}y_i) }
\end{equation}

To compute the pushforward on \eqref{eq:can_form_2n_chart}, we need to sum over all preimages. It is easy to see from the definition of the inverse map $\psi$ in \eqref{eq:psi_inverse} that choosing signs for $\sqrt{f},$ and $\sqrt{g}$ gives us arbitrary choice of signs for $a,b$. However, once this sign is fixed, the signs choice for $\sqrt{\beta_i}$ and $\sqrt{\alpha_i/\gamma_i}$ simultaneously changes the signs of $x_i$ and $y_i$. In other words, we can arbitrarily choose the sign of $x_i$, while the sign of $y_i$ is determined by that of $x_i$. On matrix coordinates, the $n-1$ choices of signs are reflected by the fact that changing the sign of a single column (from column $3$ to column $n$) of the $2\times n$ matrix in \eqref{eq:matrices_chart} does not change the image under $\theta$. Furthermore, changing the sign of $b$ and all $y_i$'s simultaneously does not change the image under $\theta$, which is equivalent to changing the sign of the second column and gauge fix the first $2\times 2$ block to be identity. Hence, $\theta$ is a $2^{n-1}$-covering map. It is also easy to see that the contribution from all $2^{n-1}$ preimages have the same sign. This is essentially due to the fact that changing the sign of a column changes sign twice in the denominator of the canonical form \eqref{eq:can_form_2n_chart} for the two minors containing that column, and changes sign twice in the numerator of the canonical form for the two total differentials of the column entries. Thus, it suffices to compute the pushforward from one preimage and multiply by $2^{n-1}$.

\begin{proposition}\label{pushforward formula}
     The pushforward of the canonical form on $\Gr(2,n)$ via $\theta:\Gr(2,n)\dashrightarrow \Gr(3,n)$ equals
     \[\frac{-1}{f\cdot
\bigg(\displaystyle\prod_{i=1}^{m}\beta_i\bigg)}  \left(\frac{\beta_1\gamma_1}{\alpha_1\gamma_{m}}\right) \bigg(\prod_{i=1}^{m-1}\frac{\gamma_{i+1}^2}{\alpha_i\gamma_{i+1}-\alpha_{i+1}\gamma_i}
\bigg)  df\wedge dg \wedge \bigg(\bigwedge_{i=1}^{m} d\left(\frac{\alpha_i}{\gamma_i}\right)\wedge d\beta_i\bigg). \]

We denote the pullback of this form on $\Gr(3,n)$ to its closed subscheme $V(3,n)$ by $\Omega_{V(3,n)}$.
\end{proposition}
\begin{proof}
 It suffices to check on the affine chart $\{p_{01}p_{12}p_{02}\not=0\}\subset \Gr(2,n) \to \{p_{012}\not=0\}\subset \Gr(3,n)$, which is given by the ring map $\phi:R\to S$.

    We fix a specific preimage, which is the same as fixing a sign choice for the square root elements in $\Tilde{R}$. As discussed above, we will multiply by $2^{n-1}$ in the end to account for the $2^{n-1}$-fold covering. In order to compute the pushforward, we take the canonical form on $\Gr(2,n)$ given in~\eqref{eq:can_form_2n_chart}
    and substitute in according to the partial inverse from \eqref{eq:psi_inverse}. 
    %Namely, $$a\mapsto \sqrt{f},\qquad b\mapsto \sqrt{g},\qquad  x_i \mapsto \sqrt{-\frac{\alpha_i\beta_i}{\gamma_i}g},\qquad y_i \mapsto \sqrt{-\frac{\beta_i\gamma_i}{\alpha_i}f}.$$
    We first compute the terms contributing to the denominator, keeping in mind the equalities in~\eqref{eq:caveat}:
\begin{align*}
\psi(ay_1-bx_1)&=\sqrt{-\frac{\beta_1\gamma_1}{\alpha_1}f^2}-\sqrt{-\frac{\alpha_1\beta_1}{\gamma_1}g^2}=\sqrt{-\frac{\beta_1\gamma_1}{\alpha_1}}\left(f-g\sqrt{\bigg(\frac{\alpha_1}{\gamma_1}\bigg)^2}\right)\\
&=\sqrt{-\frac{\beta_1\gamma_1}{\alpha_1}}\left(f+g\frac{\alpha_1}{\gamma_1}\right) =\sqrt{-\frac{\beta_1\gamma_1}{\alpha_1}}\frac{B\gamma_1+C\alpha_1}{A\gamma_1}=-\sqrt{-\frac{\beta_1\gamma_1}{\alpha_1}}\frac{\alpha_1}{\beta_1},
\end{align*}
where in the second to last equality we used $\alpha_1\gamma_1A+\beta_1\gamma_1B+\alpha_1\beta_1 C=0$.
%We conclude that 
%$$ay_1-bx_1\mapsto -\sqrt{-\frac{\beta_1\gamma_1}{\alpha_1}}\frac{\alpha_1}{\beta_1}$$
Moreover, we have 

\begin{align*}
\psi(x_iy_{i+1}-x_{i+1}y_i)&=
\sqrt{\frac{\alpha_i\beta_i}{\gamma_i}g\frac{\beta_{i+1}\gamma_{i+1}}{\alpha_{i+1}}f}-\sqrt{\frac{\alpha_{i+1}\beta_{i+1}}{\gamma_{i+1}}g\frac{\beta_i\gamma_i}{\alpha_i}f}\\
&=\sqrt{\beta_i\beta_{i+1} fg \frac{\alpha_i}{\gamma_i}\frac{\gamma_{i+1}}{\alpha_{i+1}}}\left(1-\sqrt{\left(\frac{\alpha_i}{\gamma_i}\frac{\gamma_{i+1}}{\alpha_{i+1}}\right)^{-2}}\right)\\
&=\sqrt{\beta_i\beta_{i+1}fg \frac{\alpha_i}{\gamma_i}\frac{\gamma_{i+1}}{\alpha_{i+1}}}\frac{\alpha_i\gamma_{i+1}-\alpha_{i+1}\gamma_i}{\alpha_i\gamma_{i+1}},
\end{align*}
where $i\in[m-1].$
%And $x_1y_2-x_2y_1$ turns into
%\[\sqrt{\frac{\alpha_1\beta_1}{\gamma_1}g\frac{\beta_2\gamma_2}{\alpha_2}f}-\sqrt{\frac{\alpha_2\beta_2}{\gamma_2}g\frac{\beta_1\gamma_1}{\alpha_1}f}=\sqrt{\beta_1\beta_2fg \frac{\alpha_1}{\gamma_1}\frac{\gamma_2}{\alpha_2}}\left(1-\sqrt{\left(\frac{\alpha_1}{\gamma_1}\frac{\gamma_2}{\alpha_2}\right)^{-2}}\right)=\sqrt{\beta_1\beta_2fg \frac{\alpha_1}{\gamma_1}\frac{\gamma_2}{\alpha_2}}\frac{\alpha_1\gamma_2-\alpha_2\gamma_1}{\alpha_1\gamma_2}
%\]
%Similarly, we have 
%$$\psi(x_iy_{i+1}-x_{i+1}y_i)= \sqrt{\beta_i\beta_{i+1}fg \frac{\alpha_i}{\gamma_i}\frac{\gamma_{i+1}}{\alpha_{i+1}}}\frac{\alpha_i\gamma_{i+1}-\alpha_{i+1}\gamma_i}{\alpha_i\gamma_{i+1}}$$
Combining the above, the image of the denominator via $\psi$ is given by 
%\begin{align*}
%    -ay_{n-3} (ay_1-bx_1) \prod_{i=1}^{n-4} (x_iy_{i+1}-x_{i+1}y_i)    
%\end{align*}
\begin{align*}
&\sqrt{f} \sqrt{-\frac{\beta_{n-3}\gamma_{n-3}}{\alpha_{n-3}}f} \sqrt{-\frac{\beta_1\gamma_1}{\alpha_1}}\frac{\alpha_1}{\beta_1} \prod_{i=1}^{n-4}\left(\sqrt{\beta_i\beta_{i+1}fg \frac{\alpha_i}{\gamma_i}\frac{\gamma_{i+1}}{\alpha_{i+1}}}\frac{\alpha_i\gamma_{i+1}-\alpha_{i+1}\gamma_i}{\alpha_i\gamma_{i+1}}
\right)\\
&=\sqrt{f^{n-2}g^{n-4}  \frac{\gamma_{n-3}^2}{\alpha_{n-3}^2} \prod_{i=1}^{n-3} \beta_i^2} \frac{\alpha_1}{\beta_1}\bigg(\prod_{i=1}^{n-4}\frac{\alpha_i\gamma_{i+1}-\alpha_{i+1}\gamma_i}{\alpha_i\gamma_{i+1}}
\bigg)\\
&= f\sqrt{(fg)^{n-4}} \bigg(\!\!-\frac{\gamma_{n-3}}{\alpha_{n-3}}\bigg) \bigg(\prod_{i=1}^{n-3} \beta_i\bigg)\frac{\alpha_1}{\beta_1}\bigg(\prod_{i=1}^{n-4}\frac{\alpha_i\gamma_{i+1}-\alpha_{i+1}\gamma_i}{\alpha_i\gamma_{i+1}}
\bigg)\\
&=-f\sqrt{(fg)^{n-4}} \bigg(\frac{\alpha_1\gamma_{n-3}}{\beta_1\alpha_{n-3}}\bigg) \bigg(\prod_{i=1}^{n-3} \beta_i\bigg)\bigg(\prod_{i=1}^{n-4}\frac{\alpha_i\gamma_{i+1}-\alpha_{i+1}\gamma_i}{\alpha_i\gamma_{i+1}}
\bigg)
\end{align*}

Now, we compute the pushforward of the terms in the numerator.
We have \begin{itemize}
    \item $da\mapsto\frac{1}{2\sqrt{f}}df$
    \item $db\mapsto\frac{1}{2\sqrt{g}}dg$ 
    \item $dx_i\mapsto\frac{1}{2\sqrt{-\frac{\alpha_i\beta_i}{\gamma_i}g}}d\left(-\frac{\alpha_i\beta_i}{\gamma_i}g\right)=\frac{-gd\left(\frac{\alpha_i\beta_i}{\gamma_i}\right)-\frac{\alpha_i\beta_i}{\gamma_i}dg}{2\sqrt{-\frac{\alpha_i\beta_i}{\gamma_i}g}}$ for $i\in [m]$
    \item $dy_i\mapsto\frac{1}{2\sqrt{-\frac{\beta_i\gamma_i}{\alpha_i}f}}d\left(-\frac{\beta_i\gamma_i}{\alpha_i}f\right)=\frac{-fd\left(\frac{\beta_i\gamma_i}{\alpha_i}\right)-\frac{\beta_i\gamma_i}{\alpha_i}df}{2\sqrt{-\frac{\beta_i\gamma_i}{\alpha_i}f}}$ for $i\in [m]$
\end{itemize}

%$\frac{1}{2\sqrt{f}}df=\frac{1}{2\sqrt{f}}\frac{AdB-BdA}{A^2}$
%$db=\frac{1}{2\sqrt{g}}dg=\frac{1}{2\sqrt{g}}\frac{AdC-CdA}{A^2}$

Note that, as we are taking the wedge product of all the above $1$-forms,  the term $\frac{\alpha_i\beta_i}{\gamma_i}dg$ in the image of $dx_i$ does not contribute to the product after wedging with $da=\frac{1}{2\sqrt{f}}df$. Similarly, the term $\frac{\beta_i\gamma_i}{\alpha_i}df$ in the image of $dy_i$ does not contribute. Thus, we can rewrite the numerator as 
\begin{align*}
  & \left( \frac{1}{2\sqrt{f}}df \right) \wedge \left(\frac{1}{2\sqrt{g}}dg\right) \wedge \bigg(\bigwedge_{i=1}^{n-3} \frac{-gd\left(\frac{\alpha_i\beta_i}{\gamma_i}\right)}{2\sqrt{-\frac{\alpha_i\beta_i}{\gamma_i}g}}\wedge \frac{-fd\left(\frac{\beta_i\gamma_i}{\alpha_i}\right)}{2\sqrt{-\frac{\beta_i\gamma_i}{\alpha_i}f}}\bigg)\\
  &=\frac{\sqrt{(fg)^{n-4}}}{2^{2n-4} 
\prod_{i=1}^{n-3}\beta_i} df\wedge dg \wedge \bigg(\bigwedge_{i=1}^{n-3} d\left(\frac{\alpha_i\beta_i}{\gamma_i}\right)\wedge d\left(\frac{\beta_i\gamma_i}{\alpha_i}\right)\bigg)
\end{align*}

By product rule, inverse rule, and quotient rule, we can further simplify the above as follows
\[
\begin{aligned}
&d\left(\frac{\alpha_i\beta_i}{\gamma_i}\right)
\wedge d\left(\frac{\beta_i\gamma_i}{\alpha_i}\right)
=\left(\!\!\left(\frac{\alpha_i}{\gamma_i}\right)d\beta_i+ \beta_id\left(\frac{\alpha_i}{\gamma_i}\right)\!\!\right) \wedge \left(\!\!(\left(\frac{\gamma_i}{\alpha_i}\right)d\beta_i+\beta_id\left(\frac{\gamma_i}{\alpha_i}\right)\!\!\right)\\
&=\left(\!\!\left(\frac{\alpha_i}{\gamma_i}\right)d\beta_i+ \beta_id\left(\frac{\alpha_i}{\gamma_i}\right)\!\!\right) \wedge \left(\!\!\left(\frac{\gamma_i}{\alpha_i}\right)d\beta_i-\beta_i \frac{\gamma_i^2}{\alpha_i^2}d\left(\frac{\alpha_i}{\gamma_i}\right)\!\!\right)\\
&=2\frac{\beta_i\gamma_i}{\alpha_i}d\left(\frac{\alpha_i}{\gamma_i}\right)\wedge d\beta_i\\
&=\frac{2\beta_i}{\alpha_i\gamma_i} (\gamma_i d\alpha_i-\alpha_id\gamma_i)\wedge d\beta_i
\end{aligned}
\]

Combining the expressions for the numerator and the denominator, we obtain the pushforward form $\Omega_{V(3,n)}$ as follows
\begin{align*}
&\frac{-1}{f
(\prod_{i=1}^{n-3}\beta_i)^2}  \left(\frac{\beta_1\alpha_{n-3}}{\alpha_1\gamma_{n-3}}\right) \bigg(\prod_{i=1}^{n-4}\frac{\alpha_i\gamma_{i+1}}{\alpha_i\gamma_{i+1}-\alpha_{i+1}\gamma_i}
\bigg) df\wedge dg \wedge \bigg(\bigwedge_{i=1}^{n-3} \frac{\beta_i\gamma_i}{\alpha_i}d\bigg(\frac{\alpha_i}{\gamma_i}\bigg)\wedge d\beta_i\bigg) \\
&=\frac{-1}{f
(\prod_{i=1}^{n-3}\beta_i)^2}  \left(\frac{\beta_1\alpha_{n-3}}{\alpha_1\gamma_{n-3}}\right) \bigg(\prod_{i=1}^{n-4}\frac{\alpha_i\gamma_{i+1}}{\alpha_i\gamma_{i+1}-\alpha_{i+1}\gamma_i}
\bigg) \bigg(\prod_{i=1}^{n-3}\frac{\beta_i\gamma_i}{\alpha_i}\bigg)df\wedge dg \wedge \bigg(\bigwedge_{i=1}^{n-3} d\!\left(\frac{\alpha_i}{\gamma_i}\right)\wedge d\beta_i\bigg) \\
&=\frac{-1}{f
(\prod_{i=1}^{n-3}\beta_i)}  (\frac{\beta_1\gamma_1}{\alpha_1\gamma_{n-3}}) \bigg(\prod_{i=1}^{n-4}\frac{\gamma_{i+1}^2}{\alpha_i\gamma_{i+1}-\alpha_{i+1}\gamma_i}
\bigg)  df\wedge dg \wedge \bigg(\bigwedge_{i=1}^{n-3} d\!\left(\frac{\alpha_i}{\gamma_i}\right)\wedge d\beta_i\bigg) 
\end{align*}

Note that, from the above computation, we get a coefficient $2^{-(2n-4)}\cdot 2^{n-3}=2^{-(n-1)}$. Since this corresponds to the pushforward from a single open neighborhood of a preimage, and there are $2^{n-1}$ such preimages in total, the full form $\Omega_{V(3,n)}$ is obtained by multiplying by $2^{n-1}$, therefore the coefficients cancel.
\end{proof}

Since $f$ and $g$ are complicated rational functions in the variables $\alpha,\beta,\gamma$, it is convenient to express the form $\Omega_{V(3,n)}$ in a simpler way.

\begin{proposition}\label{Canonical form on V(3,n)}
    The rational form $\Omega_{V(3,n)}$ can be expressed as follows.
    
\begin{equation}\label{eq:omega23_simple}
\begin{aligned}
& \frac{1}{f
(\prod_{i=1}^{n-3}\beta_i)}  \frac{\alpha_{n-3}}{(\alpha_1\gamma_{n-3}-\alpha_{n-3}\gamma_1)\beta_{n-3}\gamma_{n-3}^2} \bigg(\prod_{i=1}^{n-4}\frac{\gamma_{i+1}^2}{\alpha_i\gamma_{i+1}-\alpha_{i+1}\gamma_i}
\bigg) \cdot \\
&\qquad\qquad d\alpha_1 \wedge d\beta_1 \wedge d\gamma_1\wedge \bigg(\bigwedge_{i=2}^{n-4} d\left(\frac{\alpha_i}{\gamma_i}\right)\wedge d\beta_i\bigg)\wedge d\alpha_{n-3}\wedge d\beta_{n-3}\wedge  d\gamma_{n-3}
    \end{aligned}
    \end{equation}
\end{proposition}

\begin{proof}
We first rewrite $f$ as the quotient $\big(\gamma_j\frac{1}{\beta_j}-\gamma_i\frac{1}{\beta_i}\big)/ \big(\frac{\gamma_i}{\alpha_i}-\frac{\gamma_j}{\alpha_j}\big)$ and take its total differential.

\[
\begin{aligned}
df&=\frac{\big(\frac{\gamma_i}{\alpha_i}-\frac{\gamma_j}{\alpha_j}\big)d\big(\gamma_j\frac{1}{\beta_j}-\gamma_i\frac{1}{\beta_i}\big) - (\gamma_j\frac{1}{\beta_j}-\gamma_i\frac{1}{\beta_i})d(\frac{\gamma_i}{\alpha_i}-\frac{\gamma_j}{\alpha_j})}{(\frac{\gamma_i}{\alpha_i}-\frac{\gamma_j}{\alpha_j})^2}\\
&=\frac{\big(\frac{\gamma_i}{\alpha_i}-\frac{\gamma_j}{\alpha_j}\big)\big(\gamma_jd\big(\frac{1}{\beta_j}\big)+\frac{1}{\beta_j}d\gamma_j-\gamma_id\big(\frac{1}{\beta_i}\big)-\frac{1}{\beta_i}d\gamma_i\big) - \big(\gamma_j\frac{1}{\beta_j}-\gamma_i\frac{1}{\beta_i}\big)d\big(\frac{\gamma_i}{\alpha_i}-\frac{\gamma_j}{\alpha_j}\big)}{\big(\frac{\gamma_i}{\alpha_i}-\frac{\gamma_j}{\alpha_j}\big)^2}
\end{aligned}
\]

Now, notice that in the expression of $\Omega_{V(3,n)}$, we are taking $df$ wedged with $ d\big(\frac{\alpha_i
}{\gamma_i}\big) \wedge d\beta_i$. So we can ignore any multiple of $d\big(\frac{\alpha_i}{\gamma_i}\big),d(\beta_i)$ in $df$ and obtain the same wedge product. By the inverse rule, we can also any appearance of $d\big(\frac{\gamma_i}{\alpha_i}\big),d\big(\frac{1}{\beta_i}\big)$ in $df$.
So we can replace $df$ by 
\[\frac{\left(\frac{\gamma_i}{\alpha_i}-\frac{\gamma_j}{\alpha_j}\right)\left(\frac{1}{\beta_j}d\gamma_j-\frac{1}{\beta_i}d\gamma_i\right) }{\left(\frac{\gamma_i}{\alpha_i}-\frac{\gamma_j}{\alpha_j}\right)^2}
=\frac{\frac{1}{\beta_j}d\gamma_j-\frac{1}{\beta_i}d\gamma_i
}{\left(\frac{\gamma_i}{\alpha_i}-\frac{\gamma_j}{\alpha_j}\right)}\]

Now, notice that $d\gamma_i=d\big(\alpha_i\cdot \frac{\gamma_i}{\alpha_i}\big)=\alpha_id\big(\frac{\gamma_i}{\alpha_i}\big)+\frac{\gamma_i}{\alpha_i}d\alpha_i$, where we can forget of the term involving $d\big(\frac{\gamma_i}{\alpha_i}\big)$, as it  will not contribute to the final wedge product. % we can further replace $d\gamma_i$ with $\frac{\gamma_i}{\alpha_i}d\alpha_i$. 
An analogous computation (or by symmetrically exchanging $\alpha_i$ with $\gamma_i$) holds for $dg$, where instead we use that $d\alpha_j=d\big(\gamma_j\cdot \frac{\alpha_j}{\gamma_j}\big)=\gamma_jd\big(\frac{\alpha_j}{\gamma_j}\big)+\frac{\alpha_j}{\gamma_j}d\gamma_j$, where $d\alpha_j$ can be further substituted by $\frac{\alpha_j}{\gamma_j}d\gamma_j$.
Eventually, we can replace $df$ and $dg$ respectively by 
\[\frac{\frac{1}{\beta_j}d\gamma_j-\frac{1}{\beta_i}\big(\frac{\gamma_i}{\alpha_i}d\alpha_i\big)
}{\big(\frac{\gamma_i}{\alpha_i}-\frac{\gamma_j}{\alpha_j}\big)} \qquad \text{and} \qquad \frac{\frac{1}{\beta_j}\left(\frac{\alpha_j}{\gamma_j}d\gamma_j\right)-\frac{1}{\beta_i}d\alpha_i
}{\left(\frac{\alpha_i}{\gamma_i}-\frac{\alpha_j}{\gamma_j}\right)}.\]
\begin{comment}
A similar computation (or by the symmetry exchange $\alpha_i$ with $\gamma_i$) shows that we can replace $dg$ by 
\[\frac{\frac{1}{\beta_j}d\alpha_j-\frac{1}{\beta_i}d\alpha_i
}{(\frac{\alpha_i}{\gamma_i}-\frac{\alpha_j}{\gamma_j})}\]
And using $d\alpha_j=d(\gamma_j\cdot \frac{\alpha_j}{\gamma_j})=\gamma_jd(\frac{\alpha_j}{\gamma_j})+\frac{\alpha_j}{\gamma_j}d\gamma_j$, we can further replace $d\alpha_j$ by $\frac{\alpha_j}{\gamma_j}d\gamma_j$. %Thus, we can replace $dg$ by \[\frac{\frac{1}{\beta_j}(\frac{\alpha_j}{\gamma_j}d\gamma_j)-\frac{1}{\beta_i}d\alpha_i}{(\frac{\alpha_i}{\gamma_i}-\frac{\alpha_j}{\gamma_j})}\]
\end{comment}
Thus, taking the wedge product of the terms above, we can replace $df\wedge dg$ in $\Omega_{V(3,n)}$ by 
\[
%\frac{\frac{1}{\beta_j}d\gamma_j-\frac{1}{\beta_i}(\frac{\gamma_i}{\alpha_i}d\alpha_i)}{(\frac{\gamma_i}{\alpha_i}-\frac{\gamma_j}{\alpha_j})}\wedge \frac{\frac{1}{\beta_j}(\frac{\alpha_j}{\gamma_j}d\gamma_j)-\frac{1}{\beta_i}d\alpha_i}{(\frac{\alpha_i}{\gamma_i}-\frac{\alpha_j}{\gamma_j})}= 
-\frac{\alpha_j\gamma_i}{(\alpha_i\gamma_j-\alpha_j\gamma_i)}\frac{1}{\beta_i\beta_j} d\alpha_i\wedge  d\gamma_j.\]
%\[=\frac{-1}{\beta_1\beta_2}\frac{1}{\frac{\alpha_1}{\gamma_1}-\frac{\alpha_2}{\gamma_2}} d\alpha_1\wedge d\alpha_2.\]
For simplicity, we fix $i=1,$ and $j=n-3$. Thus, $\Omega_{V(3,n)}$ equals
\[\frac{-1}{f
(\prod_{i=1}^{n-3}\beta_i)}  \left(\frac{\beta_1\gamma_1}{\alpha_1\gamma_{n-3}}\right) \!\!\bigg(\prod_{i=1}^{n-4}\frac{\gamma_{i+1}^2}{\alpha_i\gamma_{i+1}-\alpha_{i+1}\gamma_i}
\bigg)\!\!  \bigg(-\frac{\alpha_{n-3}\gamma_1}{(\alpha_1\gamma_{n-3}-\alpha_{n-3}\gamma_1)} \frac{1}{\beta_1\beta_{n-3}} d\alpha_1\wedge  d\gamma_{n-3}\bigg) \wedge
\]
\[
\wedge\bigg(\bigwedge_{i=1}^{n-3} d\bigg(\frac{\alpha_i}{\gamma_i}\bigg)\wedge d\beta_i\bigg)
\]

By quotient rule, \[d\alpha_i \wedge d\!\left(\frac{\alpha_i}{\gamma_i}\right)=-\frac{\alpha}{\gamma^2}d\alpha\wedge d\gamma, \qquad  d\gamma_i\wedge d\!\left(\frac{\alpha_i}{\gamma_i}\right)= \frac{1}{\gamma_i}  d\gamma_i\wedge d\alpha_i.\]
So $\Omega_{V(3,n)}$ equals
\begin{align*}
&-\frac{1}{f
(\prod_{i=1}^{n-3}\beta_i)}  (\frac{\beta_1\gamma_1}{\alpha_1\gamma_{n-3}}) (\prod_{i=1}^{n-4}\frac{\gamma_{i+1}^2}{\alpha_i\gamma_{i+1}-\alpha_{i+1}\gamma_i}
)  (\frac{\alpha_{n-3}\gamma_1}{(\alpha_1\gamma_{n-3}-\alpha_{n-3}\gamma_1)}\frac{1}{\beta_1\beta_{n-3}})(-\frac{\alpha_1}{\gamma_1^2\gamma_{n-3}} )\cdot \\ &\qquad\qquad d\alpha_1\wedge  d\gamma_{n-3} \wedge d\gamma_1\wedge d\beta_1 \wedge (\bigwedge_{i=2}^{n-4} d(\frac{\alpha_i}{\gamma_i})\wedge d\beta_i)\wedge d\alpha_{n-3}\wedge d\beta_{n-3},
\end{align*}
\begin{comment}    
=-& \frac{1}{f
(\prod_{i=1}^{n-3}\beta_i)}  \frac{\alpha_{n-3}}{(\alpha_1\gamma_{n-3}-\alpha_{n-3}\gamma_1)\beta_{n-3}\gamma_{n-3}^2} (\prod_{i=1}^{n-4}\frac{\gamma_{i+1}^2}{\alpha_i\gamma_{i+1}-\alpha_{i+1}\gamma_i}
) \cdot \\
&\qquad\qquad d\alpha_1\wedge  d\gamma_{n-3} \wedge d\gamma_1\wedge d\beta_1 \wedge (\bigwedge_{i=2}^{n-4} d(\frac{\alpha_i}{\gamma_i})\wedge d\beta_i)\wedge d\alpha_{n-3}\wedge d\beta_{n-3}
\end{comment}

which then equals the expression \eqref{eq:omega23_simple} after the due simplification of terms. 
\end{proof}

Notice that there is a fraction $d(\frac{\alpha_i}{\gamma_i})$ appearing in the above form. One has the freedom to replace it with either $d\alpha_i$ or $d\gamma_i$ after compensating a rational function. The following proposition rewrites $\Omega_{V(3,n)}$ in a particular way useful for later purposes.

\begin{proposition}\label{the better form of the form}
    Fix $3\leq m \leq n-2$. The rational form $\Omega_{V(3,n)}$ can be expressed as
    $$\begin{aligned}
        \frac{ \gamma_{m-1}}{\alpha_1 \alpha_{m-1} \beta_{n-3}\gamma_1\gamma_{n-3}} &\frac{1}{\prod_{i=1}^{m-2}(\beta_i\gamma_{i+1}-\gamma_i\beta_{i+1}) \prod_{i=m-1}^{n-4} (\alpha_i\beta_{i+1}-\beta_i\alpha_{i+1})}\cdot \\
        &d\alpha_1 \wedge  \left( \bigwedge_{i=1}^{m-2} d\beta_i \wedge d\gamma_i \right) \wedge \left( \bigwedge_{i=m-1}^{n-3} d\alpha_i \wedge d\beta_i \right) \wedge d\gamma_{n-3}.
    \end{aligned}$$
    
\end{proposition}
\begin{proof}
We use Lemma \ref{pick 8 out of 9} to rewrite the expression of $\Omega_{V(3,n)}$ in Proposition \ref{Canonical form on V(3,n)}. For $2\leq i\leq m-2$, we replace $d(\frac{\alpha_i}{\gamma_i})$ by $\frac{\alpha_i^2}{\gamma_i^2\beta_if} d\gamma_i$. For $m-1\leq i\leq n-4$, we replace $d(\frac{\alpha_i}{\gamma_i})$ by $-\frac{1}{\beta_ig} d\alpha_i$. Thus, $(-1)^{n-5}\Omega_{V(3,n)}$ equals
$$\frac{1}{f
(\prod_{i=1}^{n-3}\beta_i)}  \frac{\alpha_{n-3}}{(\alpha_1\gamma_{n-3}-\alpha_{n-3}\gamma_1)\beta_{n-3}\gamma_{n-3}^2} \bigg(\prod_{i=1}^{n-4}\frac{\gamma_{i+1}^2}{\alpha_i\gamma_{i+1}-\alpha_{i+1}\gamma_i}
\bigg) \prod_{i=2}^{m-2}\frac{\alpha_i^2}{\gamma_i^2\beta_if} \prod_{i=m-1}^{n-4}\frac{1}{\beta_ig}   \cdot \mathcal{D}, 
$$
where $\mathcal{D}$ is the top-degree $(2n-4)$-form on $V(3,n)$ given by 
$$d\alpha_1 \wedge \left( \bigwedge_{i=1}^{m-2} d\beta_i \wedge d\gamma_i \right) \wedge \left( \bigwedge_{i=m-1}^{n-3} d\alpha_i \wedge d\beta_i \right) d\gamma_{n-3}.$$

By computation, we simplify the rational function using identities such as
$$\alpha_i\gamma_{i+1}-\alpha_{i+1}\gamma_i=\frac{-A_{i,i+1}}{\beta_i\beta_j}, \qquad f=B_{ij}/A_{ij},\qquad g=C_{ij}/A_{ij}$$ and write
$$(-1)^{n-5}\Omega_{V(3,n)}=\frac{1}{f \left( \prod_{i=1}^{n-3} \beta_i \right)} \frac{\alpha_{n-3}}{-\prod_{i=1}^{n-3} - \frac{A_{i,i+1}}{\beta_i \beta_{i+1}}} \frac{\prod_{i=2}^{n-3} \gamma_i^2}{\beta_{n-3} \gamma_{n-3}^2} \prod_{i=2}^{m-2} \frac{\alpha_i^2}{\gamma_i^2} \prod_{i=2}^{n-4} \frac{1}{\beta_i} \cdot \frac{1}{f^{m-3} g^{n-m-2}}\cdot \mathcal{D}.$$
After simplification, we write
$$\begin{aligned}
\Omega_{V(3,n)} &=   \frac{\beta_1 \alpha_{n-3} \prod_{i=2}^{m-2} \alpha_i^2\prod_{i=m-1}^{n-4} \gamma_i^2}{-f^{m-2} g^{n-m-2}\prod_{i=1}^{n-3} A_{i,i+1}}\cdot \mathcal{D}= \frac{\beta_1 \alpha_{n-3} \prod_{i=2}^{m-2} \alpha_i^2\prod_{i=m-1}^{n-4} \gamma_i^2}{A_{1,n-3} \prod_{i=1}^{m-2}B_{i,i+1} \prod_{i=m-1}^{n-4} C_{i,i+1}}\cdot \mathcal{D}\\
&= \frac{\beta_1 \alpha_{n-3} \prod_{i=2}^{m-2} \alpha_i^2\prod_{i=m-1}^{n-4} \gamma_i^2}{\beta_1\beta_{n-3}(\gamma_1\alpha_{n-3}-\alpha_1\gamma_{n-3}) \prod_{i=1}^{m-2} \alpha_i \alpha_{i+1}(\beta_i\gamma_{i+1}-\gamma_i\beta_{i+1}) \prod_{i=m-1}^{n-4} \gamma_i\gamma_{i+1} (\alpha_i\beta_{i+1}-\beta_i\alpha_{i+1})}\cdot \mathcal{D}\\
&= \frac{ \gamma_{m-1}}{\alpha_1 \alpha_{m-1} \beta_{n-3}\gamma_1\gamma_{n-3}  \prod_{i=1}^{m-2}(\beta_i\gamma_{i+1}-\gamma_i\beta_{i+1}) \prod_{i=m-1}^{n-4} (\alpha_i\beta_{i+1}-\beta_i\alpha_{i+1}) } \cdot \mathcal{D}.\end{aligned}$$
\end{proof}

Now, we show a relationship between $\Omega_{V(3,n)}$ and $\Omega_{V(3,n-1)}$, which is crucial to our inductive argument of showing that $V(3,n)$ is a positive geometry.

%{\color{red} proof of the next proposition can be greatly simplified in light of the previous newly added proposition. Will do later. }

\begin{proposition}\label{inductive formula of canonical form on V(3,n)} We have the following equality
    $$\Omega_{V(3,n)}=-\frac{\beta_1\gamma_2}{\beta_1\gamma_2-\beta_2\gamma_1}\varphi_4^* \Omega_{V(3,n-1)}\wedge d\log \beta_1\wedge d\log \gamma_1=\frac{-\alpha_2\beta_1}{\alpha_1\beta_2-\alpha_2\beta_1}\varphi_4^* \Omega_{V(3,n-1)}\wedge d\log \beta_1\wedge d\log \alpha_1$$ where $\varphi_4:V(3,n)\dashrightarrow V(3,n-1)$ is the map defined by forgetting the fourth column.
\end{proposition}
\begin{proof}
    We use Lemma \ref{pick 8 out of 9} to rewrite the expression of $\Omega_{V(3,n)}$ in Proposition \ref{Canonical form on V(3,n)} as follows.     
\begin{align*}
-\frac{1}{f
(\prod_{i=1}^{n-3}\beta_i)}  \frac{\alpha_{n-3}}{(\alpha_1\gamma_{n-3}-\alpha_{n-3}\gamma_1)\beta_{n-3}\gamma_{n-3}^2}   (\prod_{i=1}^{n-4}\frac{\gamma_{i+1}^2}{\alpha_i\gamma_{i+1}-\alpha_{i+1}\gamma_i}
) \cdot (\frac{\alpha_2A_{1,n-3}}{ \gamma_2})\Omega^{1,2,n-3}\wedge (\bigwedge_{i=3}^{n-4} d(\frac{\alpha_i}{\gamma_i})\wedge d\beta_i) 
\end{align*}

Since $\varphi_4$ sends $$\begin{pmatrix}
    1&0&0&\alpha_1&\alpha_2&\cdots &\alpha_{n-3}\\
    0&0&1&\beta_1&\beta_2&\cdots &\beta_{n-3}\\
    0&1&0&\gamma_1&\gamma_2&\cdots  &\gamma_{n-3}
\end{pmatrix} \mapsto \begin{pmatrix}
    1&0&0&\alpha_2&\cdots &\alpha_{n-3}\\
    0&0&1&\beta_2&\cdots &\beta_{n-3}\\
    0&1&0&\gamma_2&\cdots  &\gamma_{n-3}
\end{pmatrix},$$
the pullback $\varphi_4^*\Omega_{V(3,n-1)}$ is obtained by taking the expression for $\Omega_{V(3,n-1)}$ and formally add $1$ to all sub-indices.
Thus, using the formula above for $\Omega_{V(3,n)}$, we can write $\varphi_4^*\Omega_{V(3,n-1)}$ as

\begin{align*}
-&\frac{1}{f
(\prod_{i=2}^{n-3}\beta_i)}  \frac{\alpha_{n-3}}{(\alpha_2\gamma_{n-3}-\alpha_{n-3}\gamma_2)\beta_{n-3}\gamma_{n-3}^2} (\prod_{i=2}^{n-4}\frac{\gamma_{i+1}^2}{\alpha_i\gamma_{i+1}-\alpha_{i+1}\gamma_i}
) (\frac{\alpha_3 A_{2,n-3}}{ \gamma_3})\Omega^{2,3,n-3}\wedge (\bigwedge_{i=4}^{n-4} d(\frac{\alpha_i}{\gamma_i})\wedge d\beta_i) 
\end{align*}
We see that our expressions for $\Omega_{V(3,n)}$ and $\varphi_4^*\Omega_{V(3,n-1)}$ share a lot of common terms. After cancellation, to show the desired equality, it suffices to show that $M=N$, where
\begin{align*}
& M=  -\frac{1}{
\beta_1}  \frac{1}{(\alpha_1\gamma_{n-3}-\alpha_{n-3}\gamma_1)}  \frac{\gamma_{2}^2}{\alpha_1\gamma_{2}-\alpha_{2}\gamma_1}
(\frac{\alpha_2 A_{1,n-3}}{ \gamma_2})\Omega^{1,2,n-3}\wedge (d(\frac{\alpha_3}{\gamma_3})\wedge d\beta_3) 
\\
 &N= \frac{\beta_1\gamma_2}{\beta_1\gamma_2-\beta_2\gamma_1}\frac{1}{(\alpha_2\gamma_{n-3}-\alpha_{n-3}\gamma_2)}  (\frac{\alpha_3 A_{2,n-3}}{ \gamma_3})\Omega^{2,3,n-3}
\wedge d\log \beta_1\wedge d\log \gamma_1
\end{align*}
Notice that 
\begin{align*}
& \Omega^{1,2,n-3}\wedge d(\frac{\alpha_3}{\gamma_3})\wedge d\beta_3 \\
&= \frac{\alpha_1}{\beta_1\gamma_1 B_{2,n-3}}\Omega^{1,2,n-3}_{\alpha_1}\wedge d(\frac{\alpha_3}{\gamma_3})\wedge d\beta_3\\
&= \frac{\alpha_1}{\beta_1\gamma_1 B_{2,n-3}}
d\beta_1\wedge  d\gamma_1\wedge d\alpha_2\wedge d\beta_2\wedge d\gamma_2 \wedge d\alpha_{n-3}\wedge d\beta_{n-3}\wedge d\gamma_{n-3}\wedge d(\frac{\alpha_3}{\gamma_3})\wedge d\beta_3\\
&=  \frac{\alpha_1}{B_{2,n-3}}
 d\alpha_2\wedge d\beta_2\wedge d\gamma_2 \wedge d\alpha_{n-3}\wedge d\beta_{n-3}\wedge d\gamma_{n-3}\wedge (\frac{\gamma_3d\alpha_3-\alpha_3d\gamma_3}{\gamma_3^2})\wedge d\beta_3 \wedge (d\log \beta_1\wedge d\log\gamma_1)\\
&= \frac{\alpha_1}{B_{2,n-3}}
(\frac{1}{\gamma_3} \Omega^{2,3,n-3}_{\gamma_3}+\frac{\alpha_3}{\gamma_3^2}\Omega^{2,3,n-3}_{\alpha_3}) \wedge (d\log \alpha_1\wedge d\log\gamma_1)
\\
&=  -\frac{\alpha_1}{B_{2,n-3}}
\frac{\alpha_3 A_{2,n-3}}{\gamma_3} \Omega^{2,3,n-3} \wedge (d\log \beta_1\wedge d\log\gamma_1)
\end{align*}

\begin{comment}
    
\begin{align*}
 & \Omega^{1,2,n-3}\wedge d(\frac{\alpha_3}{\gamma_3})\wedge d\beta_3 \\
&= \frac{\beta_1}{\alpha_1\gamma_1 A_{23}}\Omega^{1,2,n-3}_{\beta_1}\wedge d(\frac{\alpha_3}{\gamma_3})\wedge d\beta_3\\
&= \frac{\beta_1}{\alpha_1\gamma_1 A_{23}}
d\alpha_1\wedge  d\gamma_1\wedge d\alpha_2\wedge d\beta_2\wedge d\gamma_2 \wedge d\alpha_{n-3}\wedge d\beta_{n-3}\wedge d\gamma_{n-3}\wedge d(\frac{\alpha_3}{\gamma_3})\wedge d\beta_3\\
&=  \frac{\beta_1}{A_{23}}
 d\alpha_2\wedge d\beta_2\wedge d\gamma_2 \wedge d\alpha_{n-3}\wedge d\beta_{n-3}\wedge d\gamma_{n-3}\wedge (\frac{\gamma_3d\alpha_3-\alpha_3d\gamma_3}{\gamma_3^2})\wedge d\beta_3 \wedge (d\log \alpha_1\wedge d\log\gamma_1)\\
&=  \frac{\beta_1}{A_{23}}
(\frac{1}{\gamma_3} \Omega^{2,3,n-3}_{\gamma_3}+\frac{\alpha_3}{\gamma_3^2}\Omega^{2,3,n-3}_{\alpha_3}) \wedge (d\log \alpha_1\wedge d\log\gamma_1)
\\
&=  -\frac{\beta_1}{A_{23}}
\frac{\alpha_3\gamma_3A_{2,n-3}}{\gamma_3^2} \Omega^{2,3,n-3} \wedge (d\log \alpha_1\wedge d\log\gamma_1)
\end{align*}
\end{comment}

where the last equality holds by the same computation before except the indices are shifted by $1$. We plug it into the expression of $M$ to obtain that 
\begin{align*}
M&= \frac{1}{
\beta_1}  \frac{1}{(\alpha_1\gamma_{n-3}-\alpha_{n-3}\gamma_1)}  \frac{\gamma_{2}^2}{\alpha_1\gamma_{2}-\alpha_{2}\gamma_1}
(\frac{\alpha_2 A_{1,n-3}}{ \gamma_2})\frac{\alpha_1}{B_{2,n-3}}
\frac{\alpha_3 A_{2,n-3}}{\gamma_3} \Omega^{2,3,n-3} \wedge (d\log \beta_1\wedge d\log\gamma_1)\\
&=\frac{1}{
\beta_1}  \frac{1}{(\alpha_1\gamma_{n-3}-\alpha_{n-3}\gamma_1)}  \frac{\gamma_{2}^2}{\alpha_1\gamma_{2}-\alpha_{2}\gamma_1}
(\frac{\alpha_2 A_{1,n-3}}{ \gamma_2})\frac{\alpha_1}{B_{2,n-3}}\frac{\beta_1\gamma_2-\beta_2\gamma_1}{\beta_1\gamma_2}(\alpha_2\gamma_{n-3}-\alpha_{n-3}\gamma_2) N 
\end{align*}
Using our $A,B,C$ notation, we can rewrite the above formula as
\begin{align*}
M=&\frac{1}{\beta_1} \frac{\beta_1\beta_{n-3}}{A_{1,n-3}}\frac{\gamma_2^2 \beta_1\beta_2}{A_{12}} \frac{\alpha_2 A_{1,n-3}}{\gamma_2} \frac{\alpha_1}{B_{2,n-3}}\frac{B_{12}}{\beta_1\gamma_2\alpha_1\alpha_2}\frac{A_{2,n-3}}{\beta_2\beta_{n-3}} N\\
    =&\frac{B_{12} A_{2,n-3}}{A_{12}B_{2,n-3}}N=N
\end{align*}
This computation concludes the proof.
\end{proof}

\begin{remark}\label{symmtry of pushforward}
    Recall from Remark \ref{D_n symmetry} that the pushforward form $\theta_*(\Omega_{\Gr(2,n)})$ is $D_n$-symmetric. We expect the canonical form to be $D_n$-symmetric: the positive part $V(3,n)\bigcap \Gr_{\geq0}(3,n)$ is $D_n$-symmetric, and thus so is its iterated boundaries. The canonical form should thus be $D_n$-symmetric since it determines the iterated boundaries by recursively taking poles and residues. However, our above expressions of the pushforward using local matrix coordinates do not make the $D_n$-symmetry manifest.
\end{remark}

\begin{example}\label{Canonical form n=5}
    When $n=5$, the above form gives the canonical form on $G(3,5)=V(3,5)$. This shows that $V(3,5)=G(3,5)$ is indeed a positive geometry with the pushforward form $\theta_*\Omega_{G(2,5)}$.
\end{example}

\begin{example}\label{Canonical form n=6}
    When $n=6$, the above computation gives $\theta_*\Omega_{G(2,6)}=\frac{p_{135}}{p_{156}p_{345}} \Omega^{123}$. The associated rational function is not cyclic symmetric although the form $\theta_*\Omega_{G(2,6)}$ must be. This is because $\Omega^{123}$ is defined after fixing $p_{123}$ to be $1$ and is not cyclically symmetric. We will show later in Theorem \ref{thm:V(3,6) is a positive geometry} that this indeed makes $V(3,6)$ a positive geometry.
\end{example}

\begin{proposition}\label{form invariance under column scaling}
    The rational form $\Omega_{V(3,n)}$ is invariant under scaling any column by $t$, up to a term of the form $\eta\wedge dt$ for some degree $2n-5$ rational form $\eta$ on $V(3,n)$.
\end{proposition}

\begin{proof}
By cyclic symmetry, it suffices to verify it for a specific column. Let's say we scale the fourth column by $t$. Take the formula in Proposition \ref{the better form of the form}. Observe that $\alpha_1,\beta_1,\gamma_1$ each appear once in the denominator of the rational function and $d\alpha_1,d\beta_1,d\gamma_1$ all appear once in the form. The proposition follows from the standard algebra $\frac{d(t\alpha)}{t\alpha}=\frac{d\alpha}{\alpha}+\frac{dt}{t}$. 
\end{proof}

\section{Boundaries of $V(3,n)$}\label{sec:boundary components}

By definition, the algebraic boundary of $V(3,n)$ is the Zariski closure of the analytic boundary of $V(3,n)_{\geq0}$. In this section, we work out the analytic boundary of $V(3,n)_{\geq0}$ and decompose it as a union of components, each of whose Zariski closure is irreducible. 

We denote $\Gr(3,n)$ and $V(3,n)$ by $G$ and $V$ for ease of notations. The totally nonnegative Grassmannian $G_{\geq0}$, defined in Definition \ref{Total Nonnegativity on Grassmannians}, is a analytic closed subset in $G$ with interior $\itr(G_{\geq0})=G_{>0}$. We define the positive part of $V$ by $V_{\geq0}=V\cap G_{\geq0}$, following Definition \ref{Total Nonnegativity on Grassmannians}. Notice that this is a priori different from Lam's definition stated in Conjecture \ref{conj lam}, but we show in Proposition \ref{Equality of boundaries} that these notions are equivalent.

The semialgebraic subset $(G_{\geq0}\setminus G_{>0})$ of $G$ is defined by the condition that all Pl\"ucker coordinates are nonnegative and one of the Pl\"ucker coordinates is zero. A priori, we can have the vanishing of any Pl\"ucker coordinate. However, as a direct consequence of the supermodular property (see Lemma  \ref{supermodularity}), $(G_{\geq0}\setminus G_{>0})$ is in fact the union of totally nonnegative parts of closed positroid varieties given by the vanishing of a cyclic interval Pl\"ucker coordinate $p_{i,i+1,i+2}$, where the indices are taken modulo $n$. They are exactly the codimension $1$ positroid varieties, and we will refer to them as positroid divisors.

\subsection{Decomposition of $(G_{\geq0}\setminus G_{>0})\cap V$}\label{Decomposition of G setminus G>0 in V}

This subsection is organized as follows. Lemma \ref{two positroid surfaces} through \ref{intersecting in C} studies intersection theory in the algebraic variety $\Gr(3,n)$. Lemma \ref{supermodularity} and \ref{must be cyclic} considers the restriction to the totally nonnegative part. The remaining considers further restriction inside $V(3,n)$, leading to a decomposition of $(G_{\geq0}\setminus G_{>0})\cap V$.

For a sequence of subsets $(A_1,\cdots, A_k)$ of $ [n]$, denote by $Z(A_1,\cdots, A_k)$ the subscheme of $\Gr(3,n)$ given by the scheme theoretic vanishing of $\{p_I:I\subset A_i \text{ for some } i\}$. Denote the scheme theoretic intersection $Z(A_1,\cdots, A_k)\cap V(3,n) $ by $ZV(A_1,\cdots, A_k)$. When the subsets $(A_1,\ldots, A_k)$ form a partition of $[n]$, we also write $Z(A_1,\ldots, A_k)$ as $Z(A_1|\cdots|A_k)$. Finally, let $C$ denote the open subscheme of $\Gr(3,n)$ given by the complement of $Z(12x:x\in [n])$ and cyclic rotations of that.

\begin{lemma}\label{two positroid surfaces}
Let $n\geq 4$. Let $i,j,k,l$ be distinct indices in $[n]$. Then $Z(ijk,jkl)=Z(jkx :  x\in [n])\cup Z(ijkl)$ as schemes.
\end{lemma}

\begin{proof}
By $S_n$-symmetry on $\Gr(3,n)$ by permuting columns, it suffices to show the statement for $i=1,j=2,k=3,l=4$. On the left hand side, $Z(123,234)=Z(123)\bigcap Z(234)$ is the intersection of two positroid divisors. On the right hand side, $Z(23x: x\in [n])$ and $Z(1234)$ are both codimension $2$ positroid varieties.

By \cite[Theorem 5.14]{KLS13}, all positroid varieties are compatibly Frobenius split, which implies that the scheme-theoretic intersection of any number of positroid varieties is still Frobenius split, which implies normal and thus reduced. Thus, both sides of the desired equality are reduced.

To show that the equality holds set-theoretically, recall that the positroid varieties form a stratification. Namely, any closed positroid variety $\Pi_f$ decomposes as a disjoint union of open positroid varieties $\Pi_f=\bigsqcup_{g\geq f} \Pi_{g}^\circ$. This implies that $\Pi_f\bigcap \Pi_g=\bigsqcup_{h\geq f,g} \Pi_h^\circ$. Thus, intersecting two positroid varieties set-theoretically is essentially a straighforward exercise of studying the poset of positroids, or affine bounded permutations, etc. 
\end{proof}

\begin{remark}
The above lemma can be understood set-theoretically via the projective geometry interpretation. Let $M$ be a $3$ by $n$ matrix representing a point $[{M}]\in \Gr(3,n)$. Each nonzero column among columns $c_1(M),\ldots, c_n(M)$ of $M$ give points in $\PP^2$. This point configuration depends only on $[M]$ up to $\operatorname{Aut}(\PP^2)$. 

On the left hand side, the vanishing of $p_{123}$ indicates that $c_1,c_2,c_3$ give points lying on the same line. Similarly, the vanishing of $p_{234}$ indicates that $c_2,c_3,c_4$ give points lying on the same line. These two conditions combined is equivalent to insisting that one of the two following holds.
\begin{itemize}
    \item $c_1,c_2,c_3,c_4$ give points lying on the same line.
    \item $c_2,c_3$ give the same point.
\end{itemize}
The first possibility corresponds to the vanishing of $p_I$ for any $I\subset \{1,2,3,4\}$ and the second possibility corresponds to the vanishing of $p_{2,3,x}$ for all $x\in [n]$.
\end{remark}

\begin{lemma}\label{intersecting in C}
Let $C$ be the open subscheme of $\Gr(3,n)$ given by the complement of $Z(12x:x\in [n])$ and cyclic rotations of that. For any two subsets $A,B\subset [n]$ such that $|A\cap B|\geq 2$, we have $Z(A,B)\cap C=Z(A\cup B)\cap C$.
\end{lemma}
\begin{proof}
By definition, we have $Z(A,B)\cap C\supset Z(A\cup B)\cap C$ because any subset of $A$ or $B$ is a subset of $A\cup B$. We now show the other direction by double induction on $|A|$ and $|B|$. 

Since $|A\cap B|\geq 2$, if $|A|=2$, then $A\subset B$ and the conclusion is clear. Now suppose $|A|=3$. Again by $S_n$-symmetry, we may assume without loss of generality that $A=\{123\}$ and $B=\{12b_1b_2\cdots b_m\}$. Apply Lemma \ref{two positroid surfaces} to $Z(123,12b_i)$ for each $i$ to conclude that $p_{23b_i},p_{13b_i}$ must all vanish. Apply Lemma \ref{two positroid surfaces} to $Z(23b_i,23b_j)$ to conclude that $p_{2b_i b_j},p_{3b_i b_j}$ must all vanish. Apply Lemma \ref{two positroid surfaces} to $Z(2b_ib_j,2b_jb_k)$ to conclude that $p_{b_ib_jb_k}$ must all vanish. Thus, all plucker coordinates relative to three element subsets of $\{1,2,3,b_1,\cdots,b_m\}=A\cup B$ must vanish. We verified the $|A|=3$ case.

Fix integers $M\geq N\geq 4$. Assume that the statement is true for $|A|\leq N,|B|\leq M$ but the equalities do not simultaneously hold. We want to show that the statement is true for $|A|=N, |B|=M$. 

If $A\subset B$, $Z(A,B)=Z(A\cup B)$, so the statement is vacuously true. Now we assume that both set differences $A\setminus B$ and $B\setminus A$ are nonempty. Fix any $a\in A \setminus B$. Apply the inductive hypothesis to $A'=A\setminus\{a\}$ and $B$ to obtain all vanishing of three element subsets of $A'\cup B=A\cup B\setminus\{a\}$.
Similarly, we can fix any element $b\in B\setminus A$ and apply the inductive hypothesis to $A $ and $ B'=B\setminus \{b\}$.

The only possibility of having a three element subset $I$ of $A\cup B$ not realized in the above way is if $I\supset (A\setminus B)\cup (B\setminus A)$. Since $|I|=3$, we have $|A\setminus B|\leq1$ or $|B\setminus A|\leq 1$. We assumed that both $A\setminus B$ and $B\setminus A$ are nonempty. So $|A\setminus B|=1$ or $|B\setminus A|= 1$, whereas $|B|\geq |A|\geq N\geq 4$. This implies that $|A\cap B|\geq 3$. 

Now, notice that $I\cap (A\cap B)=I\setminus((A\setminus B)\cup (B\setminus A))$ can have at most $1$ element. So there exists $c\in A\cap B$ such that $c\notin I$. And since $|A\cap B|\geq 3$, we can apply inductive hypothesis to $A\setminus \{c\},B\setminus \{c\}$ to conclude that $p_I=0$.

The inductive step allows us to reduce $|A|$ or $|B|$ by one. And eventually one of them has cardinality $3$, which is the base case.
\end{proof}

Now, we restrict to the totally nonnegative world. The following lemma is a well-known result on the {\it supermodularity} of Plücker coordinates on the totally nonnegative Grassmannians.

\begin{lemma}[{\cite[Proposition 8.7]{lam2015}}] \label{supermodularity}
Let $X\in \Gr(k,n)_{\geq0}$. Fix a cyclic ordering $<$. For any two subsets $I=\{i_1<\cdots <i_k\}$ and $J=\{j_1<\cdots<j_k\}$. Suppose the multiset $I\cup J$ equals $\{a_1\leq b_1\leq a_2\leq b_2\leq \cdots\leq a_k\leq b_k \}$. We define \begin{align*}
    \min(I,J)&=\{\min(i_1,j_1),\cdots,\min(i_k,j_k)\},&\sort_1(I,J)&=\{a_1,\cdots,a_k \}\\
    \max(I,J)&=\{\max(i_1,j_1),\cdots,\max(i_k,j_k)\},&\sort_2(I,J)&=\{b_1,\cdots,b_k \}
\end{align*} Then $$p_{I}(X)p_J(X)\leq p_{\min(I,J)}(X)p_{\max(I,J)}(X)\leq p_{\sort_1(I,J)}(X)p_{\sort_2(I,J)}(X).$$

\end{lemma}

This supermodularity on the totally nonnegative Grassmannian implies the following cyclic structure, which is especially elegant restricted to the complement of $C$. Recall from Definition \ref{Total Nonnegativity on Grassmannians} that we defined totally nonnegative part of any subvariety in $\Gr(3,n)$ by intersecting with $\Gr_{\geq0}(3,n)$.

\begin{lemma}\label{must be cyclic}
    Let $a<b<c$ be in cyclic order. Then $$Z([a,b]\cup \{c\})_{\geq 0}\cap C=(Z([c,b])_{\geq 0}\cap C)\cup (Z([a,c])_{\geq 0}\cap C) $$
\end{lemma}
\begin{proof}
It suffices to show that the left hand side is contained in the right hand side. The other inclusion is clear, as we are requiring manifestly more vanishing conditions. 

For any $d,e$ such that $e<a<b<d<c$ in cyclic ordering, apply Lemma \ref{supermodularity} to $I=\{e<a<b\},J=\{b<d<c\}$. Then $\sort_1=\{e<b<d\},\sort_2=\{a<b<c\}$ and we conclude that $p_{eab}p_{bdc}\leq p_{ebd}p_{abc}$ on $\Gr(3,n)_{\geq0}$. Since $p_{abc}$ vanishes on $Z([a,b]\cup \{c\})_{\geq 0}$, we must have $p_{eab}=0$ or $p_{bdc}=0$. That is, we showed that $Z([a,b]\cup \{c\})_{\geq 0}\cap C$ is contained in the union $$(Z([a,b]\cup \{c\}, eab)_{\geq 0}\cap C)\cup (Z([a,b]\cup \{c\}, bdc)_{\geq 0}\cap C)$$
By Lemma \ref{intersecting in C}, this union is further equal to  $$(Z([a,b]\cup \{c,e\})_{\geq 0}\cap C)\cup (Z([a,b]\cup \{c,d\})_{\geq 0}\cap C)$$

Thus, whenever we pick $d,e$ such that $d<a<b<e<c$ in cyclic order, we may absorb $d$ or $e$ in the set of  vanishing Pl\"{u}cker coordinates. In particular, we may take $d=a-1,e=b+1$, so $d$ or $e$ will be absorbed into the cyclic interval and thus we can perform this procedure until one we reach $c$. This shows that we end up with the vanishing of $[c,b]$ or $[a,c]$, which shows the containment.
\end{proof}

From now on, we further restrict to the ABCT variety $V(3,n)$. 

\begin{lemma}\label{V(3,6) case}
    Let $i<i+1<i+2<j<k<l$ be in cyclic order. Then for any sequence of subsets $\mathcal{A}$, we have a decomposition of $ZV(\mathcal{A},\{i,i+1,i+2\})_{\geq0}\cap C$ as $$
        (ZV(\mathcal{A},\{l,i,i+1,i+2\})_{\geq0}\cap C)\cup(ZV(\mathcal{A},\{i,i+1,i+2,j\})_{\geq0}\cap C)\cup (ZV(\mathcal{A},\{i,i+1,i+2\},\{j,k,l\})_{\geq0}\cap C)
  $$
\end{lemma}
\begin{proof}
    Since $ZV(\mathcal{A},\mathcal{B})=ZV(\mathcal{A})\cap ZV(\mathcal{B})$ as schemes, it suffices to show the statement for $\mathcal{A}=\emptyset$. Furthermore, by cyclic symmetry, it suffices to show the statement for $i=2$. Then we have $2<3<4<j<k<l$ in cyclic ordering. %For ease of mind, the reader can think of the special case $j=5,k=6,l=1$, and the argument is exactly the same.
    We want to show that $$ZV(234)_{\geq0}\cap C=(ZV(l234)_{\geq0}\cap C)\cup(ZV(234j)_{\geq0}\cap C)\cup (ZV(234,jkl)_{\geq0}\cap C).$$ The left hand side containing the right hand side is clear. We now show the other containment.

    Denote binomials $f_{(i_1,i_2,\cdots, i_6)}:=p_{i_1,i_2,i_3}p_{i_1,i_5,i_6}p_{i_2,i_4,i_6}p_{i_3,i_4,i_5}-p_{i_2,i_3,i_4}p_{i_1,i_2,i_6}p_{i_1,i_3,i_5}p_{i_4,i_5,i_6}$. By Theorem \ref{Ideal of ABCT}, the ABCT variety $V(3,n)$ is defined as a closed subscheme of $\Gr(3,n)$ by all such $f_{(i_1,i_2,\cdots, i_6)}$'s. We use this description to compute $ZV(234)_{\geq0}\cap C$.

In particular, the vanishing of $p_{234}$ and $f_{l234jk}$ gives the vanishing of at least one Plücker coordinate in $\{{p_{l23}},p_{ljk},p_{24k},p_{34j}\}$. So we have \begin{align*}
    ZV(234)_{\geq0}\cap C&\subset(ZV(234,l23)_{\geq 0}\cap C)\cup (ZV(234,ljk)_{\geq 0}\cap C)\cup\\
    & 
(ZV(234,24k)_{\geq 0}\cap C)\cup
(ZV(234,34j)_{\geq 0}\cap C).
\end{align*}

By Lemma \ref{intersecting in C}, we can further rewrite the union into 
\begin{align*}
    ZV(234)_{\geq0}\cap C&\subset(ZV(l234)_{\geq 0}\cap C)\cup (ZV(234,ljk)_{\geq 0}\cap C)\cup\\
    & 
(ZV(234k)_{\geq 0}\cap C)\cup
(ZV(234j)_{\geq 0}\cap C).
\end{align*}
Observe that the right hand side is exactly what we want, except that the third term in the union is extra. It suffices to show that the third term is contained in the union of the other three.

Apply Lemma \ref{supermodularity} applied to $I=\{l23\}, J=\{4jk\}, \sort_1=\{l3j\}, \sort_2=\{24k\}$, we have that $p_{l23}p_{4jk}\leq p_{l3j}p_{24k}$ for points in $G_{\geq0}$. Thus, the vanishing of $p_{24k}$ implies the vanishing of $p_{l23}$ or $p_{4jk}$, which implies that $$ZV(234k)_{\geq 0}\cap C\subset (ZV(234k,l23)_{\geq 0}\cap C) \cup (ZV(234k,4jk)_{\geq 0}\cap C) .$$ Now, notice by Lemma \ref{intersecting in C} and set-inclusion that $$ZV(234k,l23)_{\geq 0}\cap C\subset ZV(234,l23)_{\geq 0}\cap C=ZV(l234)_{\geq 0}\cap C$$ and $$ZV(234k,4jk)_{\geq 0}\cap C=ZV(234jk)_{\geq 0}\cap C\subset ZV(234j)_{\geq 0}\cap C.$$ Combining the above three, we have $ZV(234k)_{\geq 0}\cap C\subset (ZV(l234)_{\geq 0}\cap C)\cup ZV(234j)_{\geq 0}\cap C$. 
\end{proof}

\begin{theorem}\label{decomposition of boundary}
        The semialgebraic set $(G_{\geq0}\setminus G_{>0})\cap V$ is the union of the following components.
        \begin{itemize}
            \item The {totally nonnegative} part of $Z(A|B)$, where $(A|B)$ is a partition of $[n]$ into two cyclic intervals such that $|A|,|B|\geq 2$. In this case, $Z(A|B)\subset V$ and thus $Z(A|B)=ZV(A|B)$.
            \item The {totally nonnegative} part of $ZV(12x:x\in [n])$ and cyclic rotations of that.
        \end{itemize}
\end{theorem}

\begin{proof}
It is easy to verify that the above two families of semialgebraic subsets are all contained in $(G_{\geq0}\setminus G_{>0})\cap V(3,n)$, since at least one Pl\"ucker coordinate vanishes. Thus, we only need to show containment in the reverse direction.

Recall that $(G_{\geq0}\setminus G_{>0})\cap V(3,n)= \cup (Z({i,i+1,i+2})_{\geq0}\cap V(3,n))=\cup ZV(i,i+1,i+2)_{\geq0}$ by the decomposition of the boundary of the totally nonnegative Grassmannian.  By cyclic symmetry, it suffices to show that the vanishing of one Pl\"{u}cker coordinate $ZV(i,i+1,i+2)_{\geq0}$ in $V$ is contained in the union of the two families.

Recall that $C$ be the open subscheme of $\Gr(3,n)$ given by the complement of $Z(12x:x\in [n])$ and cyclic rotations of that. So the second family is just the {totally nonnegative part of the complement of} $C\cap V(3,n)$ in $V(3,n)$. Thus, it suffices to show that $ZV(i,i+1,i+2)_{\geq0}\cap C$ is contained in the union of $Z(A,B)_{\geq0}$ for $A\bigsqcup B$ a partition of $[n]$ into $2$ cyclic intervals of size at least $2$. We prove it combinatorially using the previous lemmas.

Take any point $X\in ZV(i,i+1,i+2)_{\geq0}\cap C$. Consider the maximal cyclic interval $A$ containing $[i,i+2]$ such that $X\in Z(A)$. By cyclic symmetry, we may assume that $A=[1,a]$ with $a\geq3$. This implies that $X\in Z([1,a])_{\geq0}\cap C$. If $a\geq n-2$, $Z([1,a])\subset Z([1,n-2])$ and we are done. Now, assume that $3\leq a\leq n-3$. In particular, $a+1\leq n-2$. 

Fix any $b\in [a+1,n-2]$. By the assumption that $3\leq a\leq n-3$, such $b$ exists and we have that $1<2<3$ are contained in $[1,a]$ and $b<n-1<n$ are contained in $[a+1,n]$. Apply Lemma \ref{V(3,6) case} to the cyclicly ordered subset $\{1<2<3<b<n-1<n\}$ and $\mathcal{A}=[1,a]$ to obtain that $ZV([1,a])_{\geq0}\cap C$ is the union
$$(ZV(n123,[1,a])_{\geq0}\cap C)\cup(ZV(\{1,2,3,b\},[1,a])_{\geq0}\cap C)\cup (ZV([1,a],\{b,n-1,n\})_{\geq0}\cap C).$$
By Lemma \ref{intersecting in C}, we can further rewrite it as $$(ZV([n,a])_{\geq0}\cap C)\cup(ZV([1,a]\cup\{b\})_{\geq0}\cap C)\cup (ZV([1,a],\{b,n-1,n\})_{\geq0}\cap C).$$

\noindent {\bf Case 1.} $X\in ZV([n,a])_{\geq0}\cap C$. This contradicts maximality of the interval $[1,a]$.\\

\noindent {\bf Case 2.} $X\in ZV([1,a]\cup\{b\})_{\geq0}\cap C$. Lemma \ref{must be cyclic} implies that $[1,a]$ is not maximal, contradiction.\\ 

We conclude that $X$ must be in $ZV([1,a],\{b,n-1,n\})_{\geq0}\cap C$. In particular, $X\in ZV(\{b,n-1,n\})$ for any $b\in [a+1,n-2]$. By Lemma \ref{intersecting in C}, $X\in ZV([a+1,n])$. So $X\in ZV([1,a],[a+1,n])$.

Finally, it is easy to use Theorem \ref{Ideal of ABCT} to verify that $Z(A|B)\subset V$ and thus $Z(A|B)=ZV(A|B)$ when $(A|B)$ is a partition of $[n]$ into two cyclic intervals of size at least $2$.
\end{proof}
%\begin{comment}
    
\begin{remark}
The first family of boundaries in Theorem \ref{decomposition of boundary} can be derived by a blowup construction. Let $Z_{k-1}(2,n)$ be the indeterminacy locus of the rational map $\theta_{k-1}:\Gr(2,n)\dashrightarrow \Gr(k,n)$. The scheme $Z_{k-1}(2,n)$ is a degeneracy locus of a vector bundle morphism of expected codimension $n-k+1$, and it is  reducible. It is reduced and each irreducible component thereof is a locally complete intersection in $\Gr(2,n)$. For $k=3$, $Z_2(2,n)$ is locally a subspace arrangement. To look at the boundary structure of $V(3,n)$, one can compute the {\it De Concini-Procesi sequential blowup} \cite{DP95}, which is obtained by blowing-up all the varieties in the {\it intersection poset} of the irreducible components of $Z_2(2,n)$.  
\end{remark}

%\end{comment}

\subsection{The analytic Boundary of $V(3,n)_{\geq0}$}

In this subsection, we show that the analytic boundary of $V_{\geq0}$ equals the semialgebraic set $(G_{\geq0}\setminus G_{>0})\cap V$ studied in the previous subsection.

\begin{definition}
The boundary of $V_{\geq 0}$ is defined to be
\[
\partial V_{\geq0} = V_{\geq 0}\setminus \itr_V(V_{\geq 0}).
\]
\end{definition}

The following is an immediate consequence of basic point-set topology.

\begin{proposition}\label{prop:oneinclusion}
We have $\itr_V (V_{\geq0})\supset G_{>0}\cap V$. Equivalently, $\partial V_{\geq0}\subset (G_{\geq0}\setminus G_{>0})\cap V$.
\end{proposition}
\begin{proof}
By definition of interior and subspace topology, 
\[
\itr_V(V_{\geq0})=\cup_{U\subset V_{\geq0}, U \text{ open in } V} U=\cup_{W\cap V\subset V_{\geq0}, W \text{ open in } G} (W\cap V).
\]
This is equal to $(\cup_{W\cap V\subset V_{\geq0}, W \text{ open in } G} W)\cap V$. Moreover, $W\cap V\subset V_{\geq0}$ if and only if $W\subset  V_{\geq0}\cup (G\setminus V)$. Therefore we have 

\[
\itr_V(V_{\geq0})=(\cup_{W\subset  V_{\geq0}\cup (G-V), W \text{ open in } G} W)\cap V=\itr (V_{\geq0}\cup (G-V))\cap V.
\]
Since $V_{\geq0}=G_{\geq0}\cap V$, the interior of $V_{\geq 0}$ is $\itr_G((G_{\geq0}\cap V)\cup (G\setminus V))\cap V$, which equals
$\itr_G(G_{\geq0}\cup (G\setminus V))\cap V$. It is clear that this set contains $\itr (G_{\geq0})\cap V=G_{>0}\cap V$. So we see that $\partial V_{\geq0}=V_{\geq 0}\setminus \itr_V(V_{\geq 0})\subset V_{\geq 0}\setminus (G_{>0}\cap V) = (G_{\geq 0}\setminus G_{>0})\cap V$. 
\end{proof}

\begin{theorem}\label{conj:equality of boundaries}
We have the equality $\partial V_{\geq0} =(G_{\geq0}\setminus G_{>0})\cap V$. 
\end{theorem}
\begin{proof}
From the proof of Proposition \ref{prop:oneinclusion}, $\itr_V (V_{\geq0})=\itr_G (G_{\geq0}\cup (G\setminus V))\cap V$. We know that it contains $G_{>0}\cap V$. Showing the other containment would imply the statement.

 We prove it by contradiction. Suppose there exists a point $p\in V$ such that $p\in \itr_G (G_{\geq0}\cup (G\setminus V))$ and $ p\not\in  G_{>0}$. Now, $p\in \itr (G_{\geq0}\cup (G\setminus V))$ implies that $p\in G_{\geq0}\cup (G\setminus V)= G\setminus ((G\setminus G_{\geq 0}) \cap V)$. So $p\not\in (G\setminus G_{\geq 0}) \cap V$. But $p$ is already in $V$, so $p\not\in G\setminus G_{\geq 0}$. Combining this with the assumption that $p\not\in G_{>0}$, we conclude that such a point $p$ must be contained in $(G_{\geq 0}\setminus G_{>0})\cap V$. Notice that we have the decomposition of $(G_{\geq 0}\setminus G_{>0})\cap V$ as the union of two families of semialgebraic sets in $G$ from Theorem \ref{decomposition of boundary}.

    To reach a contradiction with the assumption that $p\in \itr_G (G_{\geq0}\cup (G\setminus V))$, we show that any open set $U$ in $G$ containing $p$ has nonempty intersection with $G\setminus (G_{\geq0}\cup (G\setminus V)))=(G\setminus G_{\geq 0}) \cap V$. This is proven by constructing a sequence of points $q_\epsilon \in (G\setminus G_{\geq 0}) \cap V$ limiting to $p$ as $\epsilon \mapsto 0$. We divide the analysis according to which component, among the ones classified in Theorem \ref{decomposition of boundary}, $p$ belongs to.

    \textbf{Case 1.} The point $p$ lies in $Z(A,B)_{\geq0}\bigcap C$, where $A,B$ are cyclic intervals, $A\bigsqcup B=[n]$, and $|A|,|B|\geq2$. By cyclic symmetric, we may assume that $A=\{1,2,\cdots,m\},\ B=\{m+1,m+2,\cdots ,n\}$ for $2\leq m\leq n-1$. Then, we have $n$ distinct points $c_1,\cdots,c_n$ in $\mathbb{P}^2$ given by the columns such that $c_1,\cdots,c_m$ lie on a common line $l_1$ and $c_{m+1}\cdots c_n$ lies on a common line $l_2$ in $ \mathbb{P}^2$. After applying an automorphism of $\mathbb{P}^2$, we may assume that $l_1=\{[0:y:z]\},\ l_2=\{[x:0:z]\}$. This corresponds to left multiplication by $GL_2$ on the matrix and does not change the point it represents in $\Gr(3,n)$. By total positivity of $p$, we may assume that $p$ is given by 
    $$\begin{pmatrix}
        \cdots0\cdots &\cdots 0\cdots & \cdots x_k \cdots\\
        \cdots y_i\cdots &\cdots 0\cdots &\cdots 0\cdots\\
        \cdots z_i\cdots &\cdots z_j\cdots &\cdots z_k\cdots 
    \end{pmatrix}$$
such that $x_j,y_i,z_i,z_j$ satisfy the total positivity condition. Index the columns by $I\bigsqcup J\bigsqcup K$ according to which line the corresponding projective point lies on. Notice that $|I|,|K|\geq2$ by the full rank condition, since the matrix represents a point in $\Gr(3,n)$.

Fix any $\epsilon>0$.
Let $q_{\epsilon}$ be the point given by 
$$\begin{pmatrix}
        \cdots\frac{-\epsilon z_i^2}{y_i}\cdots &\cdots -j\sqrt{\epsilon}z_j\cdots & \cdots x_k \cdots\\
        \cdots y_i\cdots &\cdots \frac{1}{j}\sqrt{\epsilon}z_j\cdots &\cdots \frac{-\epsilon z_k^2}{x_k}\cdots\\
        \cdots z_i\cdots&\cdots z_j\cdots & \cdots z_k\cdots 
    \end{pmatrix}$$
We verify that each column of the above matrix gives a projective point lying on the same degree $2$ curve $\{[x:y:z]:xy=-\epsilon z^2\}$. This implies that $q_\epsilon\in V$. Also, we have $\lim_{\epsilon\to 0} q_\epsilon =p$.
    
    We can furthermore check that $q_\epsilon\in G-G_{\geq0}$.
    \begin{enumerate}
        \item If $|I|\geq2$, $\Delta_{i_1,i_2,k}>0,\Delta_{i_1,i_2,i_3},\Delta_{i_1,i_2,j}<0$ for small $\epsilon$. 
        \item If $|K|\geq2$, $\Delta_{i_1,k_1,k_2}>0,\Delta_{k_1,k_2,k_3},\Delta_{j,k_1,k_2}<0$ for small $\epsilon$.
        \item If $|I|=|K|=1$, $|J|=n-2\geq 3$. Then $\Delta_{j_1,j_2,j_3}<0,\Delta_{ijk}>0$ for small $\epsilon$. 
    \end{enumerate}
    
    \textbf{Case 2.} The point $p$ lies in some cyclic rotation of $ZV(12x:x\in [n])_{\geq0}$. By cyclic symmetry, we can assume that $p$ gives six points $c_1\cdots c_n$ in $\mathbb{P}^2$ such that $2$ points $c_1$ and $c_2$ coincide. Take the degree $2$ curve $\mu$ passing through the given five points. We may move $c_2$ along $\mu$ slightly to produce another point in $V$. If our degree $2$ curve is smooth at the projective point represented by $c_1$ and $c_2$, there are two directions of moving this point, where one will result in a correct orientation whereas the other one result in an incorrect orientation. Otherwise, our conic is a union of two lines and the projective point represented by $c_1$ and $c_2$ lie on the intersection. In this case, we have four directions to go, two of which will give the correct orientation. We can take $q_\epsilon$ be a sequence of points on $\mu$ converging to $p$ from the direction giving the incorrect orientation. 
    
    The above process separates two colliding points. We can iterate until we reach a sequence of points not in $C$ that approach $p$. For each point in this sequence, if it is not in $G_{\geq0}\setminus G_{>0}$, we win. Otherwise, since it is not in $C$, we are back in Case 1. So we can find another sequence in $(G\setminus G_{\geq 0})\bigcap V$ approaching this point. Then, any Hausdorff neighborhood of $p$ must contain a Hausdorff neighborhood of a point in the original sequence, which must contain a point in $(G\setminus G_{\geq 0})\bigcap V$. 
\end{proof}
The following example illustrates what happens in the case $|I|\geq2$ or $|K|\geq 2$.
\begin{example} 
Define the following  matrices
\[
p=\begin{pmatrix}
        0&0&0&1&2&3\\
        3&2&1&0&0&0\\
        1&1&1&1&1&1
    \end{pmatrix}, 
q_\epsilon=\begin{pmatrix}
        -\epsilon/3&-\epsilon/2&-\epsilon/1&1&2&3\\
        3&2&1&-\epsilon/1&-\epsilon/2&-\epsilon/3\\
        1&1&1&1&1&1
    \end{pmatrix}.
 \]
Then $\Delta_{123}=\Delta_{456}=-\frac{\epsilon}{3}<0,\ \Delta_{234}=1+2\epsilon+\frac{3}{4}\epsilon^2>0$. The following pictures gives $6$ points in $\mathbb{P}^2$ for $p$ and $q_{1/2}$. One can see that $c_1,\cdots,c_6$ are oriented clockwise. If we perturb the six points to $xy=-0.5 z^2$, we obtain $d_1,\cdots, d_6$. The fact that minors of $q_{\epsilon}$ have different signs can be visualized by the fact that the orientation breaks going between two strands (Figure \ref{fig1}). In particular, $d_1,d_2,d_3$ is oriented clockwise and $d_2,d_3,d_4$ is oriented counter-clockwise, resulting the signed area of the two triangles to have opposite signs.
Notice that if we perturb the original configuration to $xy=0.5 z^2$ instead, we would get the consistent clockwise orientation of $6$ points (Figure \ref{fig2}), resulting in $q_{-\epsilon}\in G_{>0}$.

\begin{figure}[h]
\centering
\includegraphics[scale=0.1]{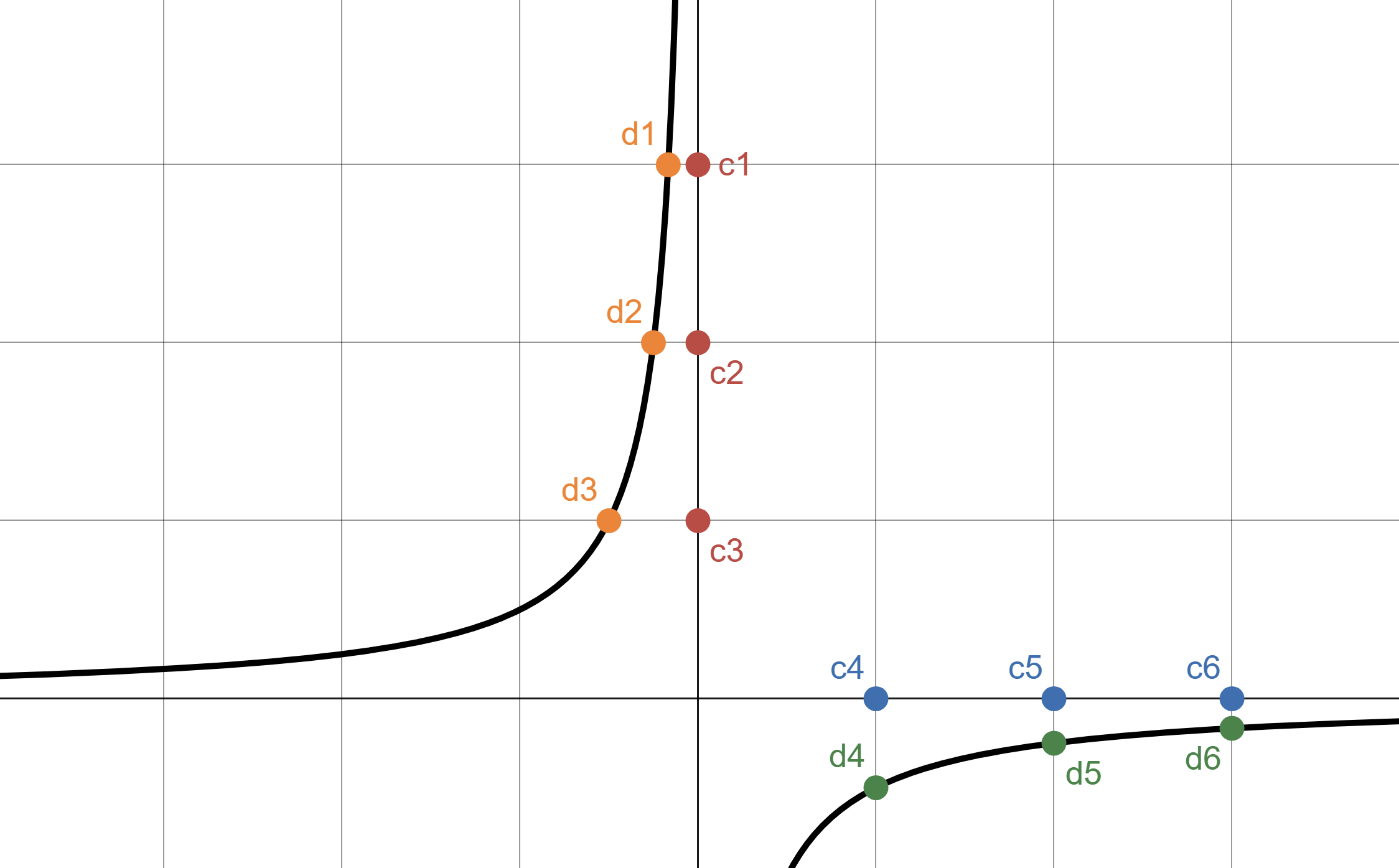}
\caption{}
\label{fig1}
\end{figure}
\begin{figure}[h]
\includegraphics[scale=0.1]{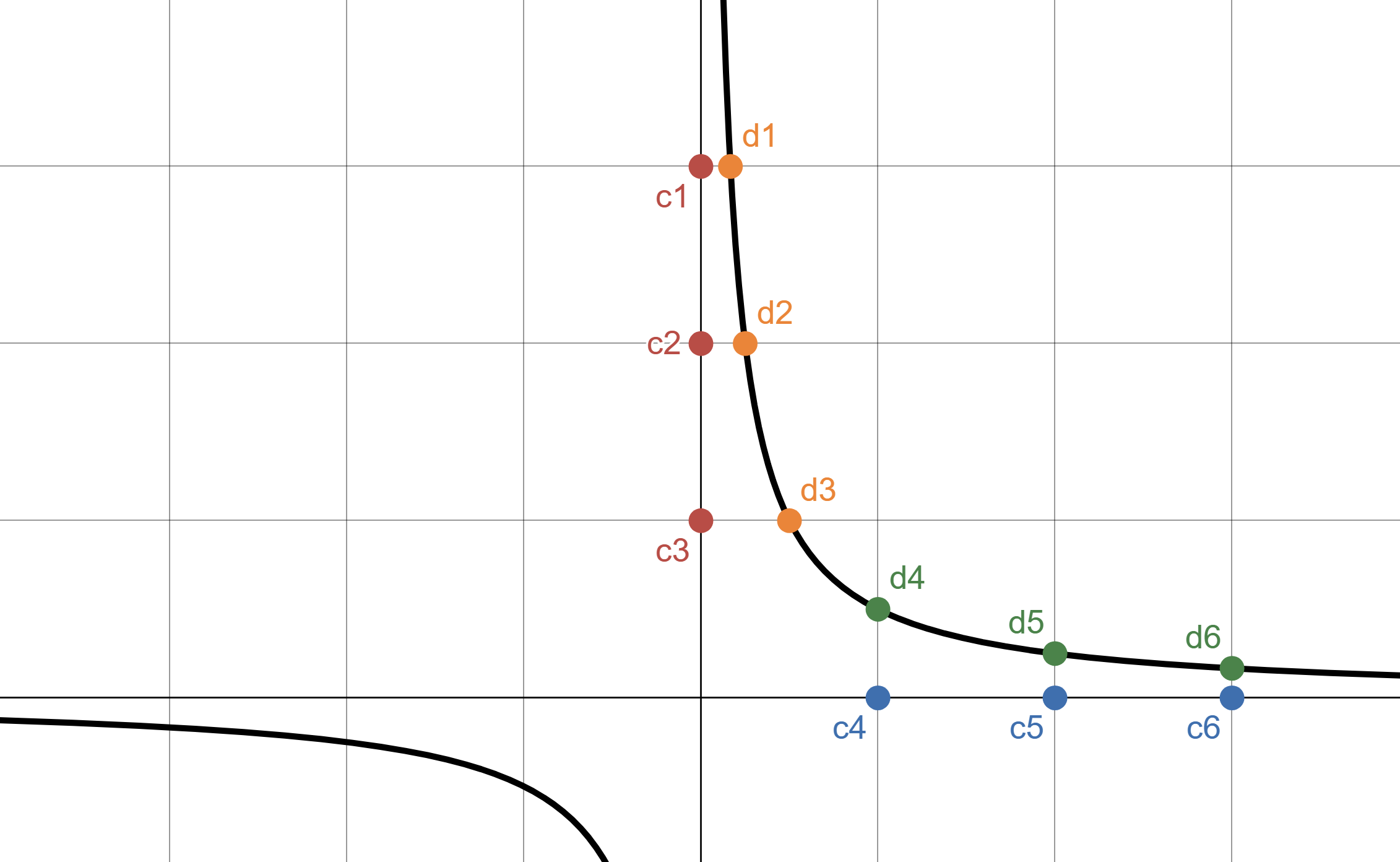}
\caption{}
\label{fig2}
\end{figure}
\end{example}

\begin{proposition}\label{analytic closure}
The semi-alegbraic set $V_{\geq0}$ is the analytic closure of its interior $V_{>0}=V\bigcap G_{>0}$.
\end{proposition}
\begin{proof}
    From Theorem \ref{conj:equality of boundaries}, we know that the analytic interior of $V_{\geq0}$ equals $V_{\geq0}\setminus ((G_{\geq0}\setminus G_{>0})\cap V) = V\bigcap G_{>0}$. The semi-algebraic set $V_{\geq0}$ is already analytically closed, so it suffices to show that  $V_{\geq0}$ is contained in the analytic closure of $V_{>0}$. We show that for any $p\in \partial V_{\geq0}=(G_{\geq0}\setminus G_{>0})\cap V$, there exists a sequence of points $q_{\epsilon}\in V_{>0}$ limiting to $p$ as $\epsilon\mapsto 0$. Notice that this is analogous to what we proved in Theorem \ref{conj:equality of boundaries}, except that the sequence were required to be ``strictly negative" points in $V$, whereas now they are required to be ``strictly positive" points in $V$.

    Exactly the same construction as in the proof of Theorem \ref{conj:equality of boundaries} follows. When $p$ lies in $Z(A,B)_{\geq0}\bigcap C$, take the $q_{\epsilon}$ in the proof of Theorem \ref{conj:equality of boundaries} for $\epsilon <0$ works. And when $p$ lies in a cyclic rotation of $ZV(12x:x\in [n])_{\geq0}$, the same procedure follows, except that now we move the new point in the direction of the correct orientation.
\end{proof}

\subsection{Two notions of positivity agree}
In the article \cite{Lam24}, Lam defined the totally nonnegative part of $V$ as $ \overline{\theta_2(\Gr(2,n)_{>0})}$,
where the closure is taken in the analytic topology of $\Gr(3,n)(\RR)$. This definition first takes totally nonnegative part in $\Gr(2,n)$ and then takes analytic closure in the image. Meanwhile, our definition of $V_{\geq0}$ first takes Zariski closure in the image and then takes totally nonnegative part in $\Gr(3,n)$. These two notions both contain $\theta_2(\Gr(2,n)_{>0})$, because the map $\theta_2$ clearly preserves positivity as a consequence of Proposition \ref{prop:map in terms of Pluckers}. In this subsection, we show that these notions agree.  

\begin{proposition} \label{Equality of boundaries}
We have the following equality of boundaries:
\[
\overline{\theta_2(\Gr(2,n)_{>0})} = V(3,n)\cap \Gr(3,n)_{\geq 0}.
\]
\end{proposition}

\begin{proof} 
In the proof of Theorem \ref{decomposition of boundary}, we showed that 
$V_{\geq 0}$ is the analytic closure 
of $G_{>0}\cap V$. Notice that Proposition \ref{prop:map in terms of Pluckers} implies that 
$\theta_2$ restricts to a map on positive parts 
\[
\theta^{>0}_2: \Gr(2,n)_{>0}\rightarrow G_{>0}\cap V. 
\]
To prove the statement, it is enough to prove that the positive map $\theta^{>0}_2$
is surjective. To check this, 
it is enough to do it on charts. On a suitable 
chart of the image, we have the following rational identity that can easily be verified by Proposition \ref{prop:map in terms of Pluckers} 
\[
\frac{p_{ij}^2}{p_{ik}^2}=\frac{p_{ijs}p_{ijl}p_{ksl}}{p_{iks}p_{ikl}p_{jsl}}.
\]
Fixing a point in $G_{>0}\cap V$, with all positive values of the coordinates $p_{ijk}$'s, we can use these formulas to produce all the ratios $p_{ij}/p_{kl}$ by taking the positive square root of the right hand side. 
\end{proof}
\section{Geometry of $V(3,n)$ and Iterated Algebraic Boundaries}\label{rational normal}
This section establishes the following geometric result, which is crucial for determining iterated boundaries of $V(3,n)$ and for the uniqueness of canonical forms on $V(3,n)$ and on its iterated boundaries, as in Remark \ref{uniqueness of canonical form}.

\begin{theorem}\label{geometry of V and its boundaries}
The ABCT variety $V(3,n)$ and its boundaries of collinear or colliding configurations as in Definition \ref{dfn:iterated boundaries two types} are reduced, irreducible, rational, normal, Cohen-Macaulay, regular in codimension $1$, and have expected dimensions from the perspective of point configurations.

A boundary of collinear configurations $\Pi(A_0; I_a \text{ for } a\in A| I_b \text{ for } b\in B)$ has codimension $n-r+|A_0|+1$ in $V(3,n)$. A boundary of colliding configurations $V(A_0;I_1,\cdots, I_r)_{\red}$ has codimension $n-r+|A_0|$ in $V(3,n)$. 
\end{theorem}
In the above Theorem, by `` iterated algebraic boundaries" we mean the family of subvarieties in $V(3,n)$ given in Definition \ref{dfn:iterated boundaries two types}. Later, we will show in Corollary \ref{iterated boundaries of V(3,n)} that they are indeed exactly all the iterated algebraic boundaries of $V(3,n)$ in the sense of Definition \ref{def:iterated boundary}, justifying the naming.

In Subsection \ref{subsection:geometry of V}, we prove the above theorem for $V(3,n)$. In Subsection \ref{subsection:Iterated Boundary components}, we give the formal definition of a family of subvarieties in $V(3,n)$, which we called iterated boundaries in the above theorem. We prove the above theorem in full generality for all iterated boundaries at the end of Subsection \ref{subsection:geometry of iterated}.

\subsection{Geometry of $V(3,n)$}\label{subsection:geometry of V}
\begin{lemma}\label{lem:V(3,n) is rational}
The variety $V(k,n)$ is rational. 
\begin{proof}
Consider the chart
$U=\lbrace p_{12}\neq 0\rbrace\subset \Gr(2,n)$.
Inside $U$, consider the locus $X$ of points of the form
\[
\begin{pmatrix}
1 & 0 & x_1 & \ldots x_{k-1} & \ldots \\
0 & 1 & y_1 & \ldots y_{k-1} & \ldots \\
\end{pmatrix} 
\]
such that the $k\times k$ matrix
\[
M = \begin{pmatrix}
1 & 0 & x_1^{k-1} &\ldots & x^{k-1}_{k-2} \\
0 & 0 & x_1^{k-2}y_1 & \ldots & x^{k-2}_{k-2}y_{k-2} \\
\vdots & \vdots & \vdots & \vdots & \vdots \\
0 & 1 & y_1^{k-1} & \ldots & y^{k-1}_{k-2} 
\end{pmatrix}
\]
is invertible. The matrix $M$ is a Vandermonde 
matrix and so it is invertible over 
an open dense subset $V$ of the torus $T = (\mathbb C^{*})^{2(k-2)}$ whose coordinates are $x_i,y_i$. The map $\theta_2$ is defined over $X\subset \Gr(2,n)$.  Consider the image $Y=\theta_2(X)\subset \Gr(k,n)$, which is a family of subspaces over $V$. Acting on the left by $M^{-1}$ on $Y$ preserves the vector spaces parametrized by $Y$. Thus we may regard $Y'=M^{-1}Y$ as a subset of $\CC^{k(n-k)}\cong U'=\lbrace p_{12\cdots k}\neq 0\rbrace \subset \Gr(k,n)$. 

Let $p\in \CC^k$ denote a column of a point in $R\in Y$. From the parametrization, we see that $p$ is a point in the cone over the rational normal curve $\nu_{k-1}(\PP^1)$. The ideal of this cone in $\CC[z_1,\ldots,z_k]$ is generated by quadrics $Q_i$ whose corresponding $k\times k$ matrices are the $A_i$. For a matrix in $Y'$, the corresponding column is $M^{-1}p$. Thus, given a quadric $Q_i$ defining the ideal of the cone over $\nu_{k-1}(\PP^1)$, the quadric $Q'_i$ over $V$ whose corresponding matrix $A'_i = M^t A_i M$ is a quadric vanishing on $M^{-1}p$. Thus the quadrics $Q'_i$ define the affine cone $C$ of the rational normal curve whose points are of the form $M^{-1}p$, defined over the function field $\CC(V)$ of $V$. The cone $C$ is a rational variety over $\CC(V)$. Doing this construction for each of the $(n-k)$ columns of a point $R\in Y$, we see that $Y'$ can be regarded as product of rational cones over the field $\CC(V)$. Since $V$ is rational, $Y'$ is rational over $\CC$. 
\end{proof}
\end{lemma}

Let $\itr V$ denote the open subscheme of $V=V(3,n)$ given by the non-vanishing of all Pl\"ucker coordinates. Fix one specific choice of a partition $(A|B)$ of $[n]$ into two subsets of cardinality at least $2$. Let $Z(A|B)^\circ$ be the open subscheme of $Z(A|B)$ given by the non-vanishing of all Pl\"ucker coordinates indexed by non-subsets of $A$ or $B$. That is, $Z(A|B)^\circ$ is the open positroid variety corresponding to the closed positroid variety $Z(A|B)$. Thus, $Z(A|B)^\circ \bigcup\itr V$ is an open subscheme of $V$ given by the non-vanishing of all Pl\"ucker coordinates indexed by non-subsets of $A$ or $B$. The open subscheme of $V$ defined by the non-vanishing of $p_{123}$ is a closed affine subvariety in the affine chart in the Grassmannian $$\begin{pmatrix}
    1&0&0&\alpha_4&\cdots&\alpha_{n}\\
    0&0&1&\beta_4&\cdots&\beta_{n}\\
    0&1&0&\gamma_4&\cdots &\gamma_{n}
\end{pmatrix}$$
cut out by the prime ideal generated by polynomials $P_{ijk}$ for $\{i,j,k\}\subset \{4,5,\cdots ,n\} $. Here, $P_{ijk}$ denotes the determinant of $\begin{pmatrix}
        \alpha_i\beta_i&\alpha_j\beta_j&\alpha_k\beta_k\\
        \alpha_i\gamma_i&\alpha_j\gamma_j&\alpha_k\gamma_k\\
        \beta_i\gamma_i&\beta_j\gamma_j&\beta_k\gamma_k\\
    \end{pmatrix}$ as in Proposition \ref{ring map injective}.

By the Jacobian criterion of smoothness, since $V$ has codimension $n-5$ in $\Gr(3,n)$, the singular locus of $V\bigcap \{p_{123}\not=0\}$ is given by the vanishing of all $(n-5)\times (n-5)$ minors of the Jacobian matrix, where the rows are indexed by all $P_{ijk}$ and the columns are indexed by all $\frac{\partial}{\partial \alpha_{l}},\frac{\partial}{\partial \beta_{l}},\frac{\partial}{\partial \gamma_{l}}$ for $l=4,5,\cdots ,n$.

\begin{remark}\label{rmk: some rels of P_ijk}
One has the following:
\begin{enumerate}
    \item $P_{ijk}$ is a polynomial only involving $\alpha_i,\alpha_j,\alpha_k,\beta_i,\beta_j,\beta_k,\gamma_i,\gamma_j,\gamma_k$;
    \item Note that $\frac{\partial P_{ijk}}{\partial \alpha_i} = \beta_i C_{jk}+\gamma_i A_{jk}$. Moreover \[\frac{\partial P_{ijk}}{\partial \alpha_i} \alpha_i =-\beta_i\gamma_iB_{jk}=-\beta_i\gamma_i\alpha_j\alpha_k \det\begin{pmatrix}
    
        \beta_j&\beta_k\\
        \gamma_j&\gamma_k\\
    \end{pmatrix},
    \]
and  
\[
\frac{\partial P_{ijk}}{\partial \beta_i} \beta_i =\alpha_i\gamma_i\beta_j\beta_k \det\begin{pmatrix}
        \alpha_j&\alpha_k\\
        \gamma_j&\gamma_k\\
    \end{pmatrix},
\]
on the locus $\lbrace P_{ijk}=0\rbrace$; 
    \item When $\alpha_j=0$, $\frac{\partial P_{ijk}}{\partial \alpha_i} = \alpha_k\beta_j\gamma_j \det \begin{pmatrix}
        \gamma_i & \gamma_k\\
        \beta_i &\beta_k
    \end{pmatrix}$;
    \item When $\alpha_k=0$, $\frac{\partial P_{ijk}}{\partial \alpha_i} = \alpha_j\beta_k\gamma_k \det \begin{pmatrix}
        \gamma_i & \gamma_j\\
        \beta_i &\beta_j
    \end{pmatrix}.$
\end{enumerate}

\end{remark}

\begin{lemma}\label{int V is smooth}
    The variety $\itr V$ is smooth.
\end{lemma}
\begin{proof}
Since $p_{123}$ does not vanish on $\itr V$, $\itr V$ is an open subscheme in a closed subscheme of the affine chart $\{p_{123}\not=0\}\subset \Gr(3,n)$, given by the non-vanishing of all maximal minors and the vanishing of the determinants $P_{ijk}$. As noticed before, the singular locus is contained in the vanishing of any $(n-5)\times (n-5)$ minor of the Jacobian matrix. Thus, it suffices to show that there exists one such a minor whose zero locus has empty intersection with $\itr V$.

Consider the $(n-5)\times (n-5)$ minor of the Jacobian matrix whose rows are indexed by $P_{45k}$ for $k=6,7,\cdots,n$ and columns indexed by $\alpha_{i}$ for $i=6,7,\cdots,n$. By the relation (1) in Remark \ref{rmk: some rels of P_ijk}, this minor is a diagonal matrix. So its determinant equals the product $\prod_{k=6}^n \frac{\partial P_{45k}}{\partial \alpha_k}$.

By the relation (2) of Remark \ref{rmk: some rels of P_ijk}, the vanishing of $\prod_{k=6}^n \frac{\partial P_{45k}}{\partial \alpha_k}$ is contained in the vanishing of the product $\prod_{k=6}^n \beta_k\gamma_k \alpha_4\alpha_5 \det \begin{pmatrix}
        \beta_4 &\beta_5\\
        \gamma_4 &\gamma_5\\
    \end{pmatrix}$, which is a product of maximal minors of the matrix 
    
\[
\begin{pmatrix}
    1&0&0&\alpha_4&\cdots&\alpha_{n}\\
    0&0&1&\beta_4&\cdots&\beta_{n}\\
    0&1&0&\gamma_4&\cdots &\gamma_{n}
\end{pmatrix}
\]
This shows that $\itr V$ has empty intersection with the singular locus of $V$.
\end{proof}

\begin{proposition}\label{prop: smooth along Z(A|B) generic}
    The variety $V$ is smooth along the generic point of $Z(A|B)$, for any choice $(A|B)$ of a partition of $[n]$ into two subsets of size at least $2$.
\end{proposition}
\begin{proof}
By the $S_n$-symmetry of $V$, we may assume that $A$ and $B$ are cyclic intervals. Furthermore, by possibly cyclically shifting $A$ and $B$, we may assume that $2,3,4\in A$ and $n,1\in B$. Thus, we have $A=\{2,3,\cdots , q\}$ and $B=\{q+1,q+2,\cdots, n, 1\}$ for some $4\leq q\leq n-1$.

Now, $Z(A|B)$ is a closed positroid variety, so it is reduced and irreducible. Therefore, the generic point of $Z(A|B)$ is the generic point of any nonempty open subscheme of $Z(A|B)$. So, it suffices to show that the singular locus of $V\cap \{p_{123}\not=0\}$ does not contain $Z(A|B)\cap \{p_{123}\not=0\}$.

The singular locus of $V\cap \{p_{123}\not=0\}$ is contained in the vanishing of any $(n-5)\times (n-5)$ minor of the Jacobian matrix described above. So to show the statement, it is enough to prove that there exists a minor zero locus does not contain $Z(A|B)\cap \{p_{123}\not=0\}$.

Consider the following $(n-5)\times (n-5)$ minor $M$ of the Jacobian matrix. The rows of $M$ are indexed by $P_{4,k,n}$ for $5\leq k\leq n-1$. The columns of $M$ are indexed by $\partial \alpha_i$, for $5\leq i\leq q$, and by $\partial \beta_i$, for $q+1 \leq i \leq n-1$. 

By observation (1) of Remark \ref{rmk: some rels of P_ijk}, the matrix $M$ is diagonal, so the determinant of $M$ is the product of its diagonal entries. Since $Z(A|B)\cap \{p_{123}\not=0\}$ is reduced and irreducible, its defining ideal is prime. Now, it is enough to check that each diagonal entry does not vanish identically on $Z(A|B)\cap \{p_{123}\not=0\}$. We have two cases.\\ 
{\bf Case 1:} Consider the indices $i$ such that $5\leq i\leq q$ (which can only happen when $q\geq 5$). The corresponding diagonal entry is $\frac{\partial P_{4,i,n}}{\partial \alpha_i}=-\frac{\partial P_{i,4,n}}{\partial \alpha_i}$. 

Notice that $-\alpha_4=p_{234}$ vanishes on $Z(A|B)$ because $2,3,4\in A$. By observation (3) of Remark \ref{rmk: some rels of P_ijk}, $\frac{\partial P_{i,4,n}}{\partial \alpha_i}=\alpha_{n}\beta_4\gamma_4\det \begin{pmatrix}
    \gamma_i & \gamma_{n}\\
    \beta_i & \beta_{n}
\end{pmatrix} = p_{23n}p_{124}p_{134}p_{1 i n}$ on $Z(A|B)$. 
Since $2,3,4,i\in A$ and $1,n\in B$, each factor of this product does not vanish on $Z(A|B)$. Again, since $Z(A|B)$ is defined by a prime ideal, the product does not vanish on $Z(A|B)$. Thus, the diagonal entry $\frac{\partial P_{i,4,q+1}}{\partial \alpha_i}$ does not vanish on $Z(A|B)$ either.\\
{\bf Case 2:} Consider the indices $i$ such that $q+1\leq i \leq n-1$. In particular, $i\in B$. By observation (2) of Remark \ref{rmk: some rels of P_ijk}, the vanishing of $\frac{\partial P_{4,i,n}}{\partial \beta_i}=-\frac{\partial P_{i,4,n}}{\partial \beta_i}$ is contained in the vanishing of the product $\alpha_i\gamma_i \beta_4 \beta_n \det \begin{pmatrix}
    \alpha_4&\alpha_n\\
    \gamma_4&\gamma_n
\end{pmatrix}$. Again, note that $-\alpha_4=p_{234}$ vanishes on $Z(A|B)$ because $2,3,4\in A$. So we can rewrite $\alpha_i\gamma_i \beta_4 \beta_n \det \begin{pmatrix}
    \alpha_4&\alpha_n\\
    \gamma_4&\gamma_n
\end{pmatrix}=-\alpha_i\gamma_i \beta_4 \beta_n \alpha_n\gamma_4=-p_{23i} p_{13i} p_{124}p_{12n}p_{23n}p_{134}$. Since $2,3,4\in A$ and $1,i,n\in B$, each factor on the latter product does not vanish on $Z(A|B)$. Again, since $Z(A|B)$ is defined by a prime ideal, the product does not vanish on $Z(A|B)$. Hence the diagonal entry $\frac{\partial P_{i,4,q+1}}{\partial \alpha_i}$ does not vanish on $Z(A|B)$ either.
\end{proof}

\begin{comment}
    Following previous notations, set 
    \begin{itemize}
        \item $A_{ij}:=\beta_i\beta_j(\gamma_i\alpha_j-\alpha_i\gamma_j)$
        \item $B_{ij}:=\alpha_i\alpha_j(\beta_i\gamma_j-\beta_i\gamma_j)$
        \item $C_{ij}:=\gamma_i\gamma_j(\alpha_i\beta_j-\alpha_i\beta_j)$
    \end{itemize}
    They satisfy the following properties
    \begin{itemize}
        \item $f=f_{ij}=\frac{B_{ij}}{A_{ij}}$, $g=g_{ij}=\frac{C_{ij}}{A_{ij}}$, $h=h_{ij}=\frac{C_{ij}}{B_{ij}}$ is independent of the choice of sub-indices.
        \item $\alpha_i\beta_i C_{jk}+\alpha_i\gamma_i A_{jk}+\beta_i\gamma_i B_{jk}=0$ for any choice of sub-indices.
    \end{itemize}
Notice that since $A_{ij},B_{ij},C_{ij}$ are product of Pl\"ucker coordinates, they are non-vanishing regular functions on $\itr V$. Thus, $f,g,h$ are all non-vanishing regular functions on $\itr V$.
\end{comment}

\begin{proposition}\label{prop:V(3,n) is normal}
The variety $V(3,n)$ is normal.
\end{proposition}
\begin{proof}
Let $V=V(3,n)$. The statement is equivalent to saying that the coordinate ring of $V$ satisfies the conditions $S_2$ and $R_1$. 
By Proposition \ref{prop:interpretation of V}, the variety $V$ is Cohen-Macaulay and so its coordinate ring is $S_2$. Therefore, to prove normality, we need to prove the coordinate ring of $V$ satisfies the condition $R_1$, i.e., $V$ is regular in codimension $1$.

By Proposition \ref{prop: smooth along Z(A|B) generic}, it is enough to show that $V\setminus \bigcup Z(A|B)$ is $R_1$, where the union runs through all the partitions of $[n]$ into two subsets of cardinality at least $2$. Since the vanishing of all Pl\"ucker coordinates is empty, it suffices to show that $V\setminus \bigcup Z(A|B)$ restricted to every single $\{p_{ijk}\not=0\}$ is $R_1$. By $S_n$-symmetry of $V$, this is equivalent to proving that $(V\setminus \bigcup Z(A|B))\bigcap \{p_{123}\not=0\}$ is $R_1$.

Consider the rational projection $\pi: (V\setminus \bigcup Z(A|B))\bigcap \{p_{123}\not=0\}\dashrightarrow (\mathbb{P}^2)^{n-3}$ given by
\[
\begin{pmatrix}
    1&0&0&\alpha_4&\cdots&\alpha_{n}\\
    0&0&1&\beta_4&\cdots&\beta_{n}\\
    0&1&0&\gamma_4&\cdots &\gamma_{n}
\end{pmatrix}
\mapsto ([\alpha_i:\beta_i:\gamma_i] \mbox{ with } i=4,\ldots, n).
\]

The source of the morphism $\pi$ is a constructible subset of $(\mathbb{A}^{3})^{n-3}$ and the projection $\pi$ is obtained from restricting the $(n-3)$-fold product of the usual projection $\mathbb{A}^3\dashrightarrow \mathbb{P}^2$ to the source. Every irreducible component of the indeterminancy locus of $\pi$  is the defined by the common zeros of $\alpha_i,\beta_i,\gamma_i$ for some $i$, which is codimension at least $2$ in $(V-\bigcup Z(A|B))\bigcap \{p_{123}\not=0\}$. Therefore, proving that that the domain of $\pi$ is $R_1$ would show the statement. 

The condition for a point in $\Gr(3,n)$ without zero columns to lie in $V$ and to lie in $Z(A|B)$ is invariant under the $\mathbb{C}^*$-action rescaling each column, so the domain of $\pi$ is a union of fibers of $\mathbb{A}^3\dashrightarrow \mathbb{P}^2$. We conclude that the fibers of $\pi$ are all equal to $(\mathbb{C}^*)^{n-3}$. Note that the dimension of the image of $\pi$ is $n-1$. By \cite[Ex. 24.8.F.(b)]{Vakil24}, the preimage of the smooth locus of the image of $\pi$ is smooth. Thus it is sufficient to verify that the image of $\pi$ is $R_1$. 

The image of $\mathrm{Im}(\pi)$ of $\pi$ is an $\mathrm{GL}(3)$-invariant subscheme of $(\mathbb{P}^2)^{n-3}$ given by the following two conditions on $n-3$ points in $\mathbb{P}^2$:
\begin{enumerate}
    \item[(i)] The $n-3$ points, along with $[1:0:0],[0:1:0],[0:0:1]$, lie on a conic in $\mathbb{P}^2$;
    \item[(ii)] The $n-3$ points, along with $[1:0:0],[0:1:0],[0:0:1]$, cannot be partitioned into two subsets $A, B$ of cardinalities at least $2$ such that points in the same subset are collinear.
\end{enumerate}
Notice that both conditions can be expressed in terms of the vanishing (or non-vanishing) of maximal minors.

Consider the diagonal action of $\mathrm{GL}(3)$ on $(\mathbb{P}^2)^n$. \cite[Theorem 1, Ch. 2]{DO88} shows that the stable points of the action contains the open subscheme defined by the following two conditions:
\begin{enumerate}
    \item[(i)] No $3$ points collide;
    \item[(ii)] No $5$ points collinear.
\end{enumerate}

Let $V_{2,n}$ be the closed subscheme of $(\mathbb{P}^2)^n$ given by the condition that $n$ points lie on a common conic. The stable points of the $\mathrm{GL}(3)$-action on $V_{2,n}$ contains configurations given by:
\begin{enumerate}
    \item[(i)] $n$ points lie on a conic in $\mathbb{P}^2$;
    \item[(ii)] No $3$ points collide;
    \item[(iii)] No $5$ points are collinear.
\end{enumerate}

In \cite{CGMS18}, $V_{2,n}$ is defined differently, but their Theorem 1.1 (1) and Remark 3.3 show that their original definition and our definition here are equivalent. They show that $V_{2,n}$ is reduced, irreducible, and normal.

Therefore, its GIT quotient $V_{2,n} /\!\!/ \mathrm{GL}(3)$ is normal too. The latter contains the geometric quotient of the stable points of $V_{2,n}$. Therefore, the GIT quotient contains the configuration space of $n$ points on $\mathbb{P}^2$ up to an automorphism of $\mathbb{P}^2$ such that:
\begin{enumerate}
    \item[(i)] $n$ points lie on a conic;
    \item[(ii)] No $3$ points collide;
    \item[(iii)] No $5$ points are collinear.
\end{enumerate}
We can further restrict to the open set where the first three points are in general position, and thus we may fix them to be $[1:0:0],[0:1:0],[0:0:1]$. We then obtain the normality of the subscheme $U$ of $(\mathbb{P}^2)^{n-3}$ given by: 
\begin{enumerate}
   \item[(a)] The $n-3$ points, along with $[1:0:0],[0:1:0],[0:0:1]$, lie on a conic;
    \item[(b)] The $n-3$ points, along with $[1:0:0],[0:1:0],[0:0:1]$, satisfy that no $3$ points collide;
    \item[(c)] The $n-3$ points, along with $[1:0:0],[0:1:0],[0:0:1]$, satisfy that no $5$ points are collinear.
\end{enumerate}
Hence, in order to show that the image $\mathrm{Im}(\pi)$ is $R_1$, it is sufficient to prove that the complement of $\mathrm{Im}(\pi) \cap U\subset \mathrm{Im}(\pi)$ has codimension at least $2$ in $\mathrm{Im}(\pi)$.

Condition $(a)$ of $U$ is implied by condition (i) for $\mathrm{Im}(\pi)$. Imposing that $3$ points collide, besides assuming the conditions for $\mathrm{Im}(\pi)$, results in the same conditions for $\mathrm{Im}(\pi)$ except that we now have $2$ fewer points to deal with, and this space has $2$ dimensions fewer. So assuming condition $(b)$ of $U$, along with those for $\mathrm{Im}(\pi)$, yields a codimension at least $2$ locus inside $\mathrm{Im}(\pi)$. 

Finally, we show assuming condition $(b)$ of $U$, besides the conditions for $\mathrm{Im}(\pi)$, guarantees condition $(c)$ of $U$. Suppose this is not the case. Then, there exists $5$ collinear points collinear but no $3$ points collide. Hence there are at least $3$ distinct points among the $5$ points that are collinear. Therefore, any conic containing all the $n$ points contains a line, and so is either a union of two lines or a non-reduced line. In either case, condition $(ii)$ of $\mathrm{Im}(\pi)$ is violated, reaching a contradiction. This implies that condition (c) is always satisfied in $\mathrm{Im}(\pi)$ outside the codimension at least $2$ locus where condition (b) of $U$ is false. In conclusion, the normality of $U$ implies that $\mathrm{Im}(\pi)$ is $R_1$, thus establishing that $V$ is $R_1$. 
\end{proof}

\subsection{Iterated Boundaries}\label{subsection:Iterated Boundary components}
In this subsection, we define a family of subvarieties in $V(3,n)$ by intersecting with certain positroid varieties and possibly take the reduced subscheme structure on them. Their importance in our work stems from the fact that we shall prove that they are exactly the iterated boundaries of $V(3,n)$; see Definition \ref{def:iterated boundary}. 

\begin{definition} \label{two classes of algebraic boundaries}
Fix a partition $A_0\bigsqcup (\bigsqcup_{i=1}^r I_i)$ of $[n]$, where each $I_i$ is a cyclic interval. Fix a partition $[r]=A\bigsqcup B$ into cyclic intervals. Define two families of positroid varieties in $\Gr(3,n)$ as follows.
\begin{align*}
&\Pi(A_0;I_1,\cdots, I_r):=\{[M]\in \Gr(3,n): \rk M_{A_0}=0, \  \rk M_{I_i}\leq 1 \text{ for all } i\},  \text{ and} \\
&\Pi(A_0; I_a \text{ for } a\in A| I_b \text{ for } b\in B):=\\
&\{[M]\in \Gr(3,n): \rk M_{A_0}=0, \  \rk M_{I_i}\leq 1 \text{ for all } i, \ \rk M_{\bigsqcup_{a\in A} I_a}\leq 2, \ \rk M_{\bigsqcup_{b\in B} I_b }\leq 2\}.    
\end{align*}
\end{definition}

We now intersect these positroid varieties with the $\mathrm{ABCT}$ variety $V(3,n)$. It is easy to see from Theorem \ref{Ideal of ABCT} that a positroid variety of the form $\Pi(A_0; I_a \text{ for } a\in A| I_b \text{ for } b\in B)$ is contained in $V(3,n)$, so the intersection is the positroid variety itself. As for a positroid variety of type $\Pi(A_0;I_1,\cdots, I_r)$, it is contained in $V(3,n)$ if and only if $\#(\bigsqcup_{i=1}^r I_i)\leq 5$. 

\begin{definition}\label{dfn:iterated boundaries two types}
Consider the following two classes of subvarieties in $V(3,n)$.
\begin{enumerate}
    \item  We call positroid varieties of the form $\Pi(A_0; I_a \text{ for } a\in A| I_b \text{ for } b\in B)$ {\it boundaries of collinear configurations}. 
    \item  Let $V(A_0;I_1,\cdots, I_r)\coloneqq \Pi(A_0;I_1,\cdots, I_r)\cap V(3,n)$ be the scheme-theoretic intersection between the positroid variety and the $\mathrm{ABCT}$ variety $V(3,n)$. We call its {\it reduced subscheme} $V(A_0;I_1,\cdots, I_r)_{\red}$ the {\it boundary of colliding configurations}.

\end{enumerate}
    \end{definition}
We will show later that the above are exactly all the iterated boundaries of $V(3,n)$ in the sense of Definition \ref{def:iterated boundary}.

\begin{remark}\label{rmk:point confi for colliding boundary}
From the perspective of the Gelfand-MacPherson correspondence recalled in Remark \ref{rmk:positroid as point config}, as a set, $\Pi(A_0;I_1,\cdots, I_r)$ is the subset of $\Gr(3,n)$ such that columns in $A_0$ are zero and nonzero columns in each $I_i$ give the same point on $\mathbb{P}^2$. Intersecting with $V(3,n)$ means requiring that these points also lie on a conic. As a set, $\Pi(A_0; I_a \text{ for } a\in A| I_b \text{ for } b\in B)$ is the subset such that the above colliding conditions hold and, moreover, nonzero columns from $\bigsqcup_{a\in A} I_a$ and from $\bigsqcup_{b\in B} I_b$ respectively give collinear points on $\mathbb{P}^2$. Thus, all points lie on the union of two lines, which is a conic. Indeed, such a variety sits inside $V(3,n)$.
\end{remark}

\begin{remark} Iterated algebraic boundaries given in the above definition can be equivalently written in notations developed in Subsection \ref{Decomposition of G setminus G>0 in V}. For example, $$\Pi(\emptyset;\{1\} \{2\}\{3\}| \{4\}\{5\}\{6\}\{7\})=Z(\{123\}|\{4567\})\text{, 
 \ \ \ \ and}$$ $$ V(\emptyset;\{12\},\{3\},\{4\},\{5\},\{6\})=ZV(\{123\},\{124\},\{125\},\{126\})$$ The new notation draws explicit connections to point configurations, making it easier to keep track of strata for later inductive arguments. The old notation makes it clear what Pl\"ucker coordinates define its ideal of vanishing. We will use both notations. 
\end{remark}

\subsection{An Algebraic Lemma}
\begin{lemma}\label{lem for iso of sheaves}
Let $R$ be a ring. Fix $r_{i,j}\in R$, for $i=1, 2, 3$ and $1\leq j\leq k$. Let $a_1,b_1,c_1,a_2,b_2,c_2$, and $v$ be indeterminates. Let $I\subset R[a_1,b_1,c_1,v]$ be the ideal defined by 
\[
I = (r_{1,j}a_1b_1 + r_{2,j}a_1c_1 + r_{3,j}b_1c_1 \ | \ 1\leq j\leq k).
\]
Let 
\[
\phi: R[a_1,a_2,b_1,b_2,c_1,c_2]\longrightarrow R[a_1,b_1,c_1,v]/I
\]
be the ring map defined by 
\[
a_1\mapsto a_1,b_1\mapsto b_1,c_1\mapsto c_1
\]
\[
a_2\mapsto a_1v, b_2\mapsto b_1v, c_2\mapsto c_1v.
\]
Then 
\[
\ker(\phi) = (a_1b_2-a_2b_1, a_1c_2-a_2c_1, b_1c_2-b_2c_1)+\] \[ (r_{1,j}a_1b_1 + r_{2,j}a_1c_1 + r_{3,j}b_1c_1, r_{1,j}a_1b_2 + r_{2,j}a_1c_2 + r_{3,j}b_1c_2, r_{1,j}a_2b_2 + r_{2,j}a_2c_2 + r_{3,j}b_2c_2, 1\leq j\leq k).
\]
and the image of ${\phi}$ is the $R$-submodule of the target generated by the images of monomials
\[
a_1^x b_1^y c_1^z v^w
\quad\text{with}\quad x+y+z\ge w,
\]
\begin{proof}
Let $J$ be the ideal on the right-hand side of the last display. It is easy to see that $J\subset \ker(\phi)$. Let
\[
\widetilde{\phi}: R[a_1,b_1,c_1,a_2,b_2,c_2]/J
\longrightarrow
R[a_1,b_1,c_1,v]/I
\]
be the induced map. To establish the desired equality of ideals, it suffices to show that $\widetilde{\phi}$ is injective. The image of $\widetilde{\phi}$ is the $R$-submodule of the target generated by the images of monomials
\[
a_1^x b_1^y c_1^z v^w
\quad\text{with}\quad x+y+z\ge w,
\]
as these are the images of the standard $R$-module generators of the source.

We define an $R$-linear map
\[
\psi \colon R[a_1,b_1,c_1,v] \longrightarrow R[a_1,b_1,c_1,a_2,b_2,c_2]
\]
on monomials as follows.
Given $a_1^x b_1^y c_1^z v^w$, we convert powers of $v$ into second-index variables using the rule:
\begin{itemize}
\item first replace $c_1v$ by $c_2$ as many times as possible;
\item then replace $b_1v$ by $b_2$;
\item finally replace $a_1v$ by $a_2$.
\end{itemize}
For instance $\psi(a_1b_1v)=a_1b_2,
\psi(a_1b_1c_1v^2)=a_1b_2c_2$. 
We extend $\psi$ to all polynomials by $R$-linearity.

By construction, $\psi$ is an $R$-module inverse to $\widetilde{\phi}$ on the level of monomials, hence on the image of $\widetilde{\phi}$.
It remains to check that $\psi$ descends to the quotient $R[a_1,b_1,c_1,v]/I$. 
Thus it suffices to show that the $R$-submodule generated by
$r_{1,j}a_1b_1+r_{2,j}a_1c_1+r_{3,j}b_1c_1$, $1\leq j\leq k$ is sent by $\psi$ into the submodule generated by
\[
r_{1,j}a_1b_1+r_{2,j}a_1c_1+r_{3,j}b_1c_1,\quad
r_{1,j}a_1b_2+r_{2,j}a_1c_2+r_{3,j}b_1c_2,\quad
r_{1,j}a_2b_2+r_{2,j}a_2c_2+r_{3,j}b_2c_2, 1\leq j\leq k
\]

For each $j$, since $\psi$ is defined on monomials and extended by $R$-linearity, it suffices to consider
\[
a_1^x b_1^y c_1^z v^w \cdot (r_{1,j}a_1b_1+r_{2,j}a_1c_1+r_{3,j}b_1c_1).
\]

If $w\le z$, then $c_1^z$ absorbs $v^w$, and the image under $\psi$ is divisible by
$r_{1,j}a_1b_1+r_{2,j}a_1c_1+r_{3,j}b_1c_1$. Otherwise, factor out $(c_1v)^z$ and reduce to the case $z=0$.
We are then reduced to
\[
v^w\left(r_{1,j} a_1^{x+1} b_1^{y+1}
+ r_{2,j} a_1^{x+1} b_1^{y} c_1
+ r_{3,j} a_1^{x} b_1^{y+1} c_1\right), \ \ w>0.
\]

If $w\le y+1$, the image under $\psi$ is a multiple of
$r_{1,j}a_1b_2+r_{2,j}a_1c_2+r_{3,j}b_1c_2$.
Otherwise, it is a multiple of
$r_{1,j}a_2b_2+r_{2,j}a_2c_2+r_{3,j}b_2c_2$.
In all cases, the image lies in the desired submodule. Hence $\psi$ is well-defined, showing that $\widetilde{\phi}$ is injective. Thus $J = \ker(\phi)$. 
\end{proof}
\end{lemma}

\subsection{Geometry of Iterated Boundaries}\label{subsection:geometry of iterated}
The boundaries of collinear configurations are positroid varieties. Thus, the geometric properties of these varieties are well-studied and are known to be positive geometries. Here we focus on the boundaries of colliding configurations $V(A_0;I_1,\cdots, I_r)_{\red}$.  Recall from Example \ref{ideal of positroid} that the positroid varieties are cut out as reduced schemes by Pl\"ucker coordinates from cyclic rank conditions. However, the scheme theoretic intersections $V(A_0;I_1,\cdots, I_r)$ could be non-reduced, i.e., $V(A_0;I_1,\cdots, I_r)\not=V(A_0;I_1,\cdots, I_r)_{\red}$. We hereby give the ideal description of $V(A_0;I_1,\cdots, I_r)_{\red}$ as a subscheme of $V(3,n)$ in terms of both matrix coordinates and Pl\"ucker coordinates, generalizing Theorem \ref{prop:interpretation of V} and Theorem \ref{Ideal of ABCT}. We proceed with the definition of a closed subscheme of $V(A_0;I_1,\cdots, I_r)$. 

\begin{definition} \label{Candidate for reduced colliding boundary}
Let $\theta:\operatorname{Mat}_{3\times n}\to  \operatorname{Mat}_{6\times n}$ be 
\[
\begin{pmatrix} 
x_{11} & x_{12} & \dots & x_{1n} \\
x_{21} & x_{22} & \dots & x_{2n} \\
x_{31} & x_{32} & \dots & x_{3n}
\end{pmatrix}\mapsto
\begin{pmatrix} 
x_{11}^2 & x_{12}^2 & \dots & x_{1n}^2 \\
x_{21}^2 & x_{22}^2 & \dots & x_{2n}^2 \\
x_{31}^2 & x_{32}^2 & \dots & x_{3n}^2 \\
x_{11}x_{21} & x_{12}x_{22}&\dots & x_{1n}x_{2n}\\
x_{11}x_{31} & x_{12}x_{32}&\dots & x_{1n}x_{3n}\\
x_{21}x_{31} & x_{22}x_{32}&\dots & x_{2n}x_{3n}\\
\end{pmatrix}.
\]
Let $\tilde{\theta}:\operatorname{Mat}_{3\times n}\to  \operatorname{Mat}_{6\times (n+n')}$ be defined by as follows. For each $a\not=b\in I_i$ and every $i\in [r]$, append the following column to the matrix image of $\theta$ 
\[
(x_{1a}x_{1b}, x_{2a}x_{2b},x_{3a}x_{3b},x_{1a}x_{2b},x_{1a}x_{3b},x_{2a}x_{3b})^T.
\]
The closed subscheme $V(A_0;I_1,\cdots, I_r)'$ of $V(A_0;I_1,\cdots, I_r)$ is then defined by the vanishing of the $6\times 6$ minors of the image of $\tilde{\theta}$. 
\end{definition}
We will eventually show that $V(A_0;I_1,\cdots, I_r)'$ defined above equals the boundary of colliding configurations $V(A_0;I_1,\cdots, I_r)_{\red}$. 

We first give the key construction of a family of {\it incidence scheme correspondences} for studying the boundaries of colliding configurations. If $I_i$ has size $1$ for all $i$, $V(A_0;I_1,\cdots, I_r)_{\red}=V(3,n-\#A_0)$ by forgetting the zero columns in $A_0$, and we looked at the geometry of $V(3,n)$ before. So we hereby assume that $\#I_1\geq 2$ by symmetry as in Remark \ref{S_n symmetry}. Fix two indices $s,t\in I_1$. Our approach is to construct a kind of resolution of $V(A_0;I_1,\cdots, I_r)'$ inductively, at each step shrinking $I_1$ into some smaller set $I_1'$ to get a boundary in $V(3,n-1)$.

\begin{definition}
Let $s, t\in I_1$ be distinct. Let
\[
\mathcal{I}_{st}(A_0;I_1,\cdots, I_r)=\{([M], [x:y]): y\cdot M_{\{s\}}=x\cdot M_{\{t\}}\}\subset V(A_0;I_1,\cdots, I_r)'\times \mathbb{P}^1,
\]
where $M_{\{s\}}$ denotes the $s$-th column of the matrix $M$. The incidence condition can be rephrased in terms of Pl\"ucker coordinates as $p_{s,j,k}(M)y=p_{t,j,k}(M)x \text{ for all } j,k$.
\end{definition}

\begin{proposition}\label{prop:bundle over P^1}
The morphism obtained by projecting to the second factor 
\[
\pi_2 :\mathcal{I}_{st}(A_0;I_1,\cdots, I_r)\subset V(A_0;I_1\cdots, I_r)'\times \mathbb{P}^1 \to \mathbb{P}^1
\]
is trivialized over standard charts of $\mathbb{P}^1=\left\{\left[x/y:1\right]\right\}\cup \left\{\left[1:y/x\right]\right\}$.
More precisely, let $U_1$ and $U_2$ be the preimages of the charts $\left\{\left[x/y:1\right]\right\}$ and $\left\{\left[1:y/x\right]\right\}$, respectively. Restrictions of $\pi_2$ to each $U_i$ factors through an isomorphism to some product space and the projection map on the product space to the second factor
\[
U_1\simeq  V(A_0;I_1',\cdots, I_r)'\times \left\{\left[x/y:1\right]\right\}\to \left\{\left[x/y:1\right]\right\}
\] 
and 
\[
U_2\simeq  V(A_0;I_1',\cdots, I_r)'\times \left\{\left[1:y/x\right]\right\}\to \left\{\left[1:y/x\right]\right\},
\]  
where $V(A_0;I_1',\cdots, I_r)'\subset V(3,n-1)$ obtained by identifying $s,t\in I_1$ as a single index to form $I_1'$ of size $\#I_1-1$.

\end{proposition}
\begin{proof}
The preimage $U_1$ of $ \left\{\left[x/y:1\right]\right\}$ under $\pi_2$ is the open subscheme 
\[
\mathcal{I}_{st}(A_0;I_1,\cdots, I_r)\cap (V(A_0;I_1\cdots, I_r)'\times \left\{\left[x/y:1\right]\right\}).
\] 
This is closed in $V(A_0;I_1\cdots, I_r)'\times \left\{\left[x/y:1\right]\right\}$, being cut out by the condition $M_{\{s\}}=x/y \cdot M_{\{t\}}$.
Define the morphism 
\[
U_1\to   V(A_0;I_1',\cdots, I_r)'\times \left\{\left[x/y:1\right]\right\},
\] 
where the first component is given by projecting $U_1\to V(A_0;I_1',\cdots, I_r)'$ by forgetting column $s$, whereas the second component is projecting to the second component. This morphism has an inverse because one can uniquely recover column $s$ from column $t$ and $x/y$. The chart isomorphism for $U_2$ follows from the same argument. 
\end{proof}

\begin{proposition} \label{prop:iso of sheaves}
One has a surjective morphism obtained by projecting to the first factor
\[
\pi_1:\mathcal{I}_{st}(A_0;I_1,\cdots, I_r)\subset V(A_0;I_1\cdots, I_r)'\times \mathbb{P}^1 \to V(A_0;I_1\cdots, I_r)'.
\]
The morphism $\pi_1$ is birational and induces an isomorphism of structure sheaves $\mathcal{O}_{V(A_0;I_1,\cdots, I_r)'}\cong (\pi_1)_*\mathcal{O}_{\mathcal{I}_{st}(A_0;I_1,\cdots, I_r)}$. 
\end{proposition} 
\begin{proof}
By using an inductive argument, we may assume that $A_0=\emptyset$ (see the proof of Theorem \ref{geometry of V and its boundaries} where the inductive argument is explicit). We have the induced morphism of sheaves $\mathcal{O}_{V(I_1,\cdots, I_r)'}\to (\pi_1)_*\mathcal{O}_{\mathcal{I}_{st}(I_1,\cdots, I_r)}$. A sheaf morphism being an isomorphism can be checked affine-locally. We have affine charts on $V(A_0;I_1\cdots, I_r)'$ obtained by intersecting $V(A_0;I_1\cdots, I_r)'$ with standard affine charts $\{p_{J}\not=0\}$ of the Grassmannian. Notice that if $\# (J\cap I_1)\geq2$, $\{p_{J}\not=0\}$ has empty intersection with $V(I_1,\cdots, I_r)$. We divide the case into $\# (J\cap \{s,t\})=0 \text{ or } 1$. Denote the open intersection $\{p_{J}\not=0\}\cap V(I_1,\cdots, I_r)'$ by $V(I_1,\cdots, I_r)'_J$.

When $\# (J\bigcap \{s,t\})=1$, by $S_n$-symmetry, because the $I_i$ are cyclic intervals, we can assume without loss of generality that $J=\{1,2,3\}$, $s=3$, and that $I_1$ is an interval $[3,k]$ for $k\geq4$. Fix a chart for $V(I_1,\cdots, I_r)'_J$ where the first $3\times 3$ minor is $\begin{pmatrix}
    1&0&0\\
    0&0&1\\
    0&1&0\\
\end{pmatrix}$. The relations coming from colliding points from $I_1$ forces $a_i=c_i=0$ for $i\in I_1$. Thus, $\pi_1$ is an isomorphism restricted to its preimage, where the inverse can be obtained by recovering the point in $\mathbb{P}^1$ from $[1:b_t]$.

When $s,t\not\in J$, again we can assume that $J=\{1,2,3\}$, $s=4, t=5$, and that $I_1$ is an interval containing $4,5$. Fix a chart for $V(I_1,\cdots, I_r)'_J$ where the first $3\times 3$ submatrix of our matrix representative $M$ is $\begin{pmatrix}
    1&0&0\\
    0&0&1\\
    0&1&0\\
\end{pmatrix}.$
Write $M=\begin{pmatrix}
    1&0&0&a_4&a_5&\cdots &a_n\\
    0&0&1&b_4&b_5&\cdots &b_n\\
    0&1&0&c_4&c_5&\cdots &c_n\\
    \end{pmatrix}$ for coordinates on $V(I_1,\cdots, I_r)'_J$. The ring of regular functions on $V(I_1,\cdots, I_r)'_J$ is 
    
\[
\mathcal{O}(V(I_1,\cdots, I_r)'_J)=k[a_i,b_i,c_i: i=4,5,\cdots,n  ]/L_J,
\] 
where $L_J$ is generated by the maximal minors of the image of $\tilde\theta$ and $2\times 2$ minors of $M_{I_i}$ for each $i$. We now explicitly compute the ring of regular functions on the preimage $(\pi_1)^{-1}(V(I_1,\cdots, I_r)'_J)$.

The preimage $\pi_1^{-1}(V(I_1,\cdots, I_r)'_J)$ is the open subscheme of $\mathcal{I}_{st}\subset V(I_1,\cdots, I_r)'\times \mathbb{P}^1$ defined by the nonvanishing of $p_{J}$ on the first component. Thus, in view of the covering $\mathcal{I}_{st}=U_1\cup U_2$ with 
\[
U_1\simeq  V(A_0;I_1',\cdots, I_r)'\times \left\{\left[x/y:1\right]\right\}, \qquad U_2\simeq  V(A_0;I_1',\cdots, I_r)'\times \left\{\left[1:y/x\right]\right\}\]
derived from Proposition \ref{prop:bundle over P^1}, 
the preimage $\pi_1^{-1}(V(I_1,\cdots, I_r)'_J)$ is the union \[(U_1\cap \{p_J\not=0\})\cup (U_2\cap \{p_J\not=0\})\]. So we have 
\[
U_1\cap \{p_J\not=0\}\simeq  \left(V(A_0;I_1',\cdots, I_r)'\cap \{p_J\not=0\}\right )\times \left\{\left[x/y:1\right]\right\},
\]
and
\[
U_2\cap \{p_J\not=0\}\simeq  \left(V(A_0;I_1',\cdots, I_r)'\cap \{p_J\not=0\}\right)\times \left\{\left[1:y/x\right]\right\},
\]
where the $p_J$ on the right-hand side is understood to be a coordinate on $V(3,n-1)$ where columns are indexed by $[n]$ with the indices $s$ and $t$ identified. This is well-defined because $s,t\not\in J$ and thus $J$ has still size $3$, upon identifying $s$ with $t$. 

On $U_2\cap \{p_J\not=0\}\simeq  \left(V(A_0;I_1',\cdots, I_r)'\cap \{p_J\not=0\}\right)\times \left\{\left[1:y/x\right]\right\}$, fix coordinates 
\[
M'=\begin{pmatrix}
    1&0&0&a_4&a_6&\cdots &a_n\\
    0&0&1&b_4&b_6&\cdots &b_n\\
    0&1&0&c_4&c_6&\cdots &c_n\\
   \end{pmatrix}
\] 
on the right-hand side by specifying the first $3\times 3$ submatrix. The isomorphism is given by 
\[
\begin{pmatrix}
    1&0&0&a_4&a_4\cdot y/x&a_6&\cdots &a_n\\
    0&0&1&b_4&b_4\cdot y/x&b_6&\cdots &b_n\\
    0&1&0&c_4&c_4\cdot y/x&c_6&\cdots &c_n\\
    \end{pmatrix}
\leftrightarrow
\left(\begin{pmatrix}
    1&0&0&a_4&a_6&\cdots &a_n\\
    0&0&1&b_4&b_6&\cdots &b_n\\
    0&1&0&c_4&c_6&\cdots &c_n\\
    \end{pmatrix}, [1:y/x]\right)
\]

For ease of notation, write $v$ for the ratio $y/x$. Thus, we have the coordinate ring
$\mathcal{O}(U_2\cap \{p_J\not=0\})= k[v, a_i,b_i,c_i: i=4,6,7,\cdots,n ]/ L_{1} ,$ where $L_1$ is the ideal defined by the vanishing of the maximal minors of $\tilde\theta(M')$ and the  $2\times 2$ minors of $M'_{I_i}$ for all $i$.

Similarly, for $U_1\bigcap \{p_J\not=0\}\simeq  \left(V(A_0;I_1',\cdots, I_r)'\bigcap \{p_J\not=0\}\right)\times \left\{\left[x/y:1\right]\right\}$, fix coordinates $$M''=\begin{pmatrix}
    1&0&0&a_5&a_6&\cdots &a_n\\
    0&0&1&b_5&b_6&\cdots &b_n\\
    0&1&0&c_5&c_6&\cdots &c_n\\
    \end{pmatrix}$$ on the right-hand side and the isomorphism is given by 
$$\begin{pmatrix}
    1&0&0&a_5\cdot x/y&a_5&a_6&\cdots &a_n\\
    0&0&1&b_5\cdot x/y&b_5&b_6&\cdots &b_n\\
    0&1&0&c_5\cdot x/y&c_5&c_6&\cdots &c_n\\
    \end{pmatrix}
\leftrightarrow
\left(\begin{pmatrix}
    1&0&0&a_5&a_6&\cdots &a_n\\
    0&0&1&b_5&b_6&\cdots &b_n\\
    0&1&0&c_5&c_6&\cdots &c_n\\
    \end{pmatrix}, [x/y:1]\right)$$
Recall that we write $v$ for $y/x$. We have the coordinate ring $\mathcal{O}(U_2\cap \{p_J\not=0\})= k[v^{-1}, a_i,b_i,c_i: i=5,6,\cdots,n ]/ L_{2} ,$ where $L_2$ is the ideal defined by the vanishing of the maximal minors of $\tilde\theta(M'')$ and the $2\times 2$ minors of $M''_{I_i}$ for all $i$.

The gluing identifies the open subscheme where $t$ is invertible, on which we have the connecting isomorphism 
$$
\left(\begin{pmatrix}
    1&0&0&a_4&a_6&\cdots &a_n\\
    0&0&1&b_4&b_6&\cdots &b_n\\
    0&1&0&c_4&c_6&\cdots &c_n\\
    \end{pmatrix}, [1:y/x]\right) \leftrightarrow
\left(\begin{pmatrix}
    1&0&0&a_5&a_6&\cdots &a_n\\
    0&0&1&b_5&b_6&\cdots &b_n\\
    0&1&0&c_5&c_6&\cdots &c_n\\
    \end{pmatrix}, [x/y:1]\right)
$$
where $a_5=a_4\cdot y/x,b_5=b_4\cdot y/x,c_5=c_4\cdot y/x$.

Thus, the coordinate ring of $(\pi_1)^{-1}(V(I_1,\cdots, I_r)'_J)$ is the subset of
$$\{(f,g)\in k[v, a_i,b_i,c_i: i=4,6,7,\cdots,n ]/ L_{1} \times k[t^{-1}, a_i,b_i,c_i: i=5,6,\cdots,n ]/ L_{2}\}$$ given by the condition that 
$f\leftrightarrow g$ under the correspondence $a_5=va_4,b_5=vb_4,c_5=vc_4$. Notice that $f$ uniquely determines $g$ under this correspondence, and so $(\pi_1)_*\mathcal{O}(V(I_1,\cdots, I_r)'_J)$ is the subring of $$k[v, a_i,b_i,c_i: i=4,6,7,\cdots,n ]/ L_{1} $$ satisfying the condition that after replacing $a_4\mapsto v^{-1}a_5,b_4\mapsto v^{-1}b_5,c_4\mapsto v^{-1}c_5$ in $f$, we obtain some polynomial $g$ in $k[v^{-1}, a_i,b_i,c_i: i=5,6,\cdots,n ]/ L_{2}\}$. The ring homomorphism $\mathcal{O}(V(I_1,\cdots, I_r)'_J)\to (\pi_1)_*\mathcal{O}(V(I_1,\cdots, I_r)'_J)$ maps $a_5\mapsto va_4,b_5\mapsto vb_4,c_5\mapsto vc_4$, and all other variables are sent to themselves.

Thus, to show the statement, it suffices to show that the ring homomorphism 
$$k[a_i,b_i,c_i: i=4,5,\cdots,n  ]/L_J \to k[v, a_i,b_i,c_i: i=4,6,7,\cdots,n ]/ L_{1}$$ sending 
\[
a_5\mapsto va_4,b_5\mapsto vb_4,c_5\mapsto vc_4
\]
 and fixing all other variables is injective, with image being the subring of the target consisting of polynomials $f$ such that its corresponding $g$ is in $k[v^{-1}, a_i,b_i,c_i: i=5,6,\cdots,n ]/ L_{2}$.

First, notice that we may organize all relations in $L_J$ and in $L_{1}$ not involving $a_4,b_4,c_4, a_5,b_5,c_5$ and it turns out that those are the same by definition of $L_J$ and $L_1$. Let $R$ be $k[a_i,b_i,c_i: i=6,\cdots,n ]$ modulo these relations. 
The remaining relations involving variables indexed by $4$ and $5$ are exactly of the form described in Lemma \ref{lem for iso of sheaves}. Finally, notice that $k[v, a_i,b_i,c_i: i=4,6,7,\cdots,n ]/ L_{1}$ is spanned as $R$-module by monomials $a_{4}^{x}b_{4}^{y}c_{4}^{z}v^w$, which get transformed to $a_{5}^{x}b_{5}^{y}c_{5}^{z}v^{w-x-y-z}$. Also, $L_2$ does not involve $v$, and so whether a polynomial in $k[t^{\pm 1}, a_i,b_i,c_i: i=5,6,\cdots,n ]/ L_{2}$ is in $k[v^{- 1}, a_i,b_i,c_i: i=5,6,\cdots,n ]/ L_{2}$ is independent of the choice of lift in $k[v^{\pm 1}, a_i,b_i,c_i: i=5,6,\cdots,n ]$. We conclude that the $R$-submodule of $ k[v, a_i,b_i,c_i: i=4,6,7,\cdots,n ]/ L_{1} $ consisting of polynomials $f$ whose corresponding $g$ is in $k[v^{-1}, a_i,b_i,c_i: i=5,6,\cdots,n ]/ L_{2}$ is spanned by $a_{4}^{x}b_{4}^{y}c_{4}^{z}v^w$ such that $x+y+z\geq w$. So, by Lemma \ref{lem for iso of sheaves}, the image of the ring homomorphism is $(\pi_1)_*\mathcal{O}(V(I_1,\cdots, I_r)'_J)$.
\end{proof}

\begin{theorem}\label{thm:matrix description of colliding boundary}
As schemes, one has $V(A_0; I_1,\cdots , I_r)_{\red}=V(A_0; I_1,\cdots , I_r)'$.
In other words, the boundary of colliding configurations $V(A_0;I_1,\cdots, I_r)_{\red}$ is cut out as a subscheme of the Grassmannian by the vanishing of Pl\"ucker coordinates defining the positroid variety $\Pi(A_0;I_1,\cdots, I_r)$ as in Example \ref{ideal of positroid} and the vanishing of the maximal minors of the image of $\tilde{\theta}$.
\end{theorem}

\begin{proof}
Proposition \ref{prop:bundle over P^1} can be repeated with every $V(A_0; I_1,\cdots , I_r)'$ replaced by $V(A_0; I_1,\cdots , I_r)$ and fibers are isomorphic to $V(A_0; I_1',\cdots , I_r)$. The analogous map to Proposition \ref{prop:iso of sheaves} is still birational onto $V(A_0; I_1,\cdots , I_r)$. Thus, $V(A_0; I_1,\cdots , I_r)$ is irreducible if $V(A_0; I_1',\cdots , I_r)$ is irreducible and $\dim V(A_0; I_1,\cdots , I_r)= \dim V(A_0; I_1',\cdots , I_r)+1$. By induction, $V(A_0; I_1,\cdots , I_r)$ is irreducible. The base case is $V(3,n)$ being irreducible as in Theorem \ref{prop:interpretation of V}. Also, the same induction for $V(A_0; I_1,\cdots , I_r)'$ shows that $V(A_0; I_1,\cdots , I_r)$ and $V(A_0; I_1,\cdots , I_r)'$ have the same dimension.

Thus, to show that $V(A_0; I_1,\cdots , I_r)'=V(A_0; I_1,\cdots , I_r)_{\red}$, it suffices to show that the scheme $V(A_0; I_1,\cdots , I_r)'$ is reduced. Notice that reducedness is a condition that can be checked locally. Now, we see that $V(A_0;I_1,\cdots, I_r)'\subset V(3,n)$ is reduced if so is $\mathcal{I}_{st}(A_0;I_1,\cdots, I_r)$ by Proposition \ref{prop:iso of sheaves}. In turn, the latter is reduced if so is $V(A_0;I_1',\cdots, I_r)'\subset V(3,n-1)$ by Proposition \ref{prop:bundle over P^1}. In this way, we decrease both $n$ and the size of $I_1$ as long as some $\# I_i\geq 2$. The base case of the induction is when each $I_i$ has size $1$ so that the boundary component is isomorphic to $V(3,m)$ by forgetting zero columns. The fact that $V(3,m)$ is reduced follows from Theorem \ref{prop:interpretation of V}.
\end{proof}

As a direct corollary of Theorem \ref{thm:matrix description of colliding boundary}, we have the ideal description in Pl\"ucker coordinates.  
\begin{theorem} \label{Ideal of colliding boundary}
 The boundary of colliding configurations $V(A_0;I_1,\cdots, I_r)_{\red}$ is cut out as a subscheme of $V(A_0;I_1,\cdots, I_r)$ by the vanishing of quartics in Pl\"ucker coordinates given by 
 \[p_{j_1',j_2,j_3}p_{j_1,j_5,j_6}p_{j_2,j_4,j_6}p_{j_3,j_4,j_5}-p_{j_2,j_3,j_4}p_{j_1',j_2,j_6}p_{j_1,j_3,j_5}p_{j_4,j_5,j_6} = 0,\] where the indices varies through all $7$-element subsets of $[n]$ such that $j_1,j_1'\in I_i$ for some $i$.
 \end{theorem}
\begin{proof}
    A standard computation shows that the extra minors appearing in $\tilde{\theta}(M)$ that are required to vanish on $V(A_0;I_1,\cdots, I_r)_{\red}$ can be expressed as these Pl\"ucker polynomials.
\end{proof}

Finally, we prove the main theorem of this section.

\begin{proof}[Proof of Theorem \ref{geometry of V and its boundaries}]
Theorem \ref{thm:matrix description of colliding boundary} implies that $V(A_0;I_1,\cdots, I_r)'=V(A_0;I_1,\cdots, I_r)_{\red}$ can be interchanged in the statement of Proposition \ref{prop:bundle over P^1} and \ref{prop:iso of sheaves}. Similar to the reducedness argument as in the proof of Theorem \ref{thm:matrix description of colliding boundary}, note that irreducibility, rationality, normality, and Cohen-Macaulayness can be passed from fiber of locally trivial family to the total space and from source to target if the morphism is surjective, birational, and induces isomorphism on structure sheaves. Thus, we apply the same inductive argument and the base case is when each $I_i$ has size $1$ so that the boundary component is isomorphic to $V(3,m)$ by forgetting zero columns. Results proved in subsection \ref{subsection:geometry of V}, along with the result that $V(3,m)$ is reduced, irreducible, and Cohen-Macaulay as in Theorem \ref{prop:interpretation of V}, gives the base case. The dimension count is again an easy consequence of Proposition \ref{prop:bundle over P^1} and \ref{prop:iso of sheaves}. So far, we verified the theorem for $V(3,n)$ and boundaries of colliding configurations. The theorem holds for boundaries of collinear configurations because they are positroid varieties.
\end{proof}

\begin{remark}
    The above argument shows that a boundary of colliding configurations $V(A_0;I_1,\cdots, I_r)_{\red}$ has codimension $n-r+|A_0|$ in $V(3,n)$. This is indeed the expected codimension from the perspective of  the Gelfand-MacPherson correspondence as in Remark \ref{rmk:positroid as point config}.

    After a smooth conic in $\mathbb{P}^2$ is fixed by five general points, having a new column amounts to choosing a projective point on the conic and choosing a scaling that determines the column. Thus, adding a column creates $2$ extra dimensions. This corresponds to the fact that $\dim V(3,n+1)=\dim V(3,n)+2$. Now, if we instead add a new column with the condition that it is parallel to an existing column, this only gives $1$ extra dimension. And if we add a zero column, it contributes no extra dimensions. Thus, the codimension of $V(A_0;I_1,\cdots, I_r)_{\red}$ in $V(3,n)$ is given by $2|A_0|+\sum_{i=1}^r (|I_i|-1)=n-r+|A_0|$. A boundary of collinear configurations $\Pi(A_0; I_a \text{ for } a\in A| I_b \text{ for } b\in B)$ has $1$ more codimension in $V(3,n)$ than the corresponding boundary of colliding configurations. This can be interpreted as the fact that $4$ general points determine a union of two lines whereas $5$ general points determine a smooth conic.

\end{remark}

\begin{remark}
    The extra relations in Theorem \ref{thm:matrix description of colliding boundary} coming from the extra columns in $\tilde{\theta}(M)$ have a natural interpretation in terms of point configurations as in Remark \ref{rmk:point configuration for V}.

   Consider appending to the matrix $$M=\begin{pmatrix} 
x_{11} & x_{12} & \dots & x_{1n} \\
x_{21} & x_{22} & \dots & x_{2n} \\
x_{31} & x_{32} & \dots & x_{3n}
\end{pmatrix}$$ columns given by $(\sqrt{x_{1a}x_{1b}},\sqrt{x_{2a}x_{2b}},\sqrt{x_{3a}x_{3b}})^{T}$ for each $a\not=b\in I_i$ for all $i\in [r]$. Notice that $\theta$ and $\tilde{\theta}$ are $\mathrm{GL}(3)$-equivariant. Fixing a chart on $\Gr(3,n)$ given by some $3\times 3$ submatrix being the identity, then Pl\"ucker positivity implies that entries in each row must be nonnegative or nonpositive, depending on the chart you choose. Thus, on positive points, the square roots makes sense by choosing the squareroot according to the sign requirement.

Now, for each $[M]\in V(A_0;I_1,\cdots, I_r)_{\geq0}$, nonzero columns in $M$ give points configurations on a common conic where $[x_{1a}:x_{2a}:x_{3a}]=[x_{1b}:x_{2b}:x_{3b}]=p\in \mathbb{P}^2$ as in Remark \ref{rmk:point confi for colliding boundary}. Then the column $(\sqrt{x_{1a}x_{1b}},\sqrt{x_{2a}x_{2b}},\sqrt{x_{3a}x_{3b}})^{T}$ gives the same point $[\sqrt{x_{1a}x_{1b}}:\sqrt{x_{2a}x_{2b}}:\sqrt{x_{3a}x_{3b}}]=p\in \mathbb{P}^2$. Therefore, this new column gives a point that lies on the same conic in $\mathbb{P}^2$. By Theorem \ref{prop:interpretation of V}, if we append the image of $(\sqrt{x_{1a}x_{1b}},\sqrt{x_{2a}x_{2b}},\sqrt{x_{3a}x_{3b}})^{T}$ under the Veronese map as a new column of $\theta(M)$, the new matrix is still not full-rank.

The Veronese image of $(\sqrt{x_{1a}x_{1b}},\sqrt{x_{2a}x_{2b}},\sqrt{x_{3a}x_{3b}})^{T}$ is 
\[
(x_{1a}x_{1b}, x_{2a}x_{2b},x_{3a}x_{3b},\sqrt{x_{1a}x_{1b}}\sqrt{x_{2a}x_{2b}},\sqrt{x_{1a}x_{1b}}\sqrt{x_{3a}x_{3b}},\sqrt{x_{2a}x_{2b}}\sqrt{x_{3a}x_{3b}})^T.
\]
On the positroid cell $\Pi(A_0;I_1,\cdots, I_r)_{\geq0}$, the submatrix $\begin{pmatrix}
    x_{1a}&x_{1b}\\
    x_{2a}&x_{2b}\\    x_{3a}&x_{3b}
    \end{pmatrix}$
has rank at most $1$. Therefore, the vanishing of $2\times 2$ minors give $x_{1a}x_{2b}=x_{2a}x_{1b}$ and therefore $\sqrt{x_{1a}x_{1b}}\sqrt{x_{2a}x_{2b}}=x_{1a}x_{2b}=x_{2a}x_{1b}$.

We thus recover the column $$(x_{1a}x_{1b}, x_{2a}x_{2b},x_{3a}x_{3b},x_{1a}x_{2b},x_{1a}x_{3b},x_{2a}x_{3b})^T$$ which we append to form $\tilde{\theta}$ as in the above theorem. The above argument shows that the matrix $\tilde{\theta}(M)$ is not full-rank on $V(A_0;I_1,\cdots, I_r)_{\geq0}$, and thus the condition still holds on its Zariski closure $V(A_0;I_1,\cdots, I_r)_{\red}$. This is a much easier way to see that these extra relations must hold on the reduced subscheme. The content of the above theorem is that these extra conditions suffice to cut out the reduced scheme structure. In fact, we can prove the following theorem whose proof is not difficult but tedious and so we omitted it.  
\end{remark}

\begin{theorem}\label{thm: ideal of reduced part of V(A_0,I_1...I_r)}
The generators of the ideal of $V(A_0;I_1,\cdots, I_r)' = V(A_0;I_1,\cdots, I_r)_{\red}$ described in Theorem \ref{thm:matrix description of colliding boundary} form a Gr\"obner basis. 
\end{theorem}

\section{$V(3,5)$ and $V(3,6)$ are Positive Geometries}\label{n=5,6 case}
In this section, we consider the case $n=5$ or $6$. This will serve as the base case for our inductive argument on the main theorem.

When $n=5$, $V(3,5)=G(3,5)$. As stated in Example \ref{Canonical form n=5}, the pushforward form equals the canonical form on $G(3,5)$. The conjecture is thus verifies for $n=5$.

Now, let $n=6$. For simplicity of notation, we shall also denote with $D(F)\subset X$ the principal open $\lbrace F\neq 0\rbrace \subset X$ inside a locally closed set $X$. Note that in the case $n=6,$ the positroid variety $Z(12x)\subset \Gr(3,6)$ is contained inside $V(3,6)$. So $ZV(12x)=Z(12x)$ is a positroid variety. Thus, by Theorem \ref{decomposition of boundary}, the boundary divisors of $V(3,6)$ are given by $Z(A|B)$, $Z(12x)$, and cyclic rotations thereof. In this case, all boundary components of $V(3,6)$ turn out to be positroid varieties, which are already known to be positive geometries.

Fix the affine chart  $$\begin{pmatrix}
    1&0&0&\alpha_1&\alpha_2&\alpha_{3}\\
    0&0&1&\beta_1&\beta_2&\beta_{3}\\
    0&1&0&\gamma_1&\gamma_2 &\gamma_{3}
\end{pmatrix}$$
As in Example \ref{Canonical form n=6}, we compute pushforward form $\Omega_{V(3,6)}$ on the given chart in $V(3,6)$ and it is given by 
\[
\Omega_{V(3,6)} = \frac{\Omega^{123}_{\gamma_3}}{p_{234}p_{345}p_{126}p_{145}p_{256}}
\] 
%= \frac{\Omega_{\gamma_1}}{p_{234}p_{126}p_{156}p_{356}p_{245}}=
%=\frac{\Omega_{\alpha_2}}{p_{234}p_{156}p_{126}p_{346}p_{145}}
%=\frac{[246]\Omega_{\alpha_3}}{p_{234}p_{345}p_{126}p_{156}p_{245}p_{146}}.

Here one-forms are pulled-back to $V(3,6)$ via the inclusion map $V(3,6)\to G(3,6)$. Whereas the nine forms $d \alpha_i,d \beta_i,d \gamma_i$ form a frame on $G(3,6)$, any $8$ of them do not globally form a frame on the entire $V(3,6)$. The variety $V(3,6)$ is a hypersurface in $G(3,6)$ cut out, in the given chart, by the polynomial 
\[P = \det \begin{pmatrix}
    \alpha_1\beta_1&\alpha_2\beta_2&\alpha_3\beta_3\\
        \alpha_1\gamma_1&\alpha_2\gamma_2&\alpha_3\gamma_3\\
        \beta_1\gamma_1&\beta_2\gamma_2&\beta_3\gamma_3\\
\end{pmatrix}.
\] 
The tangent bundle of $V(3,6)$ is the quotient of the tangent bundle on $G(3,6)$ by the relation $dP=0$ in local coordinates. Namely, we have the relation $\frac{\partial P}{\partial \alpha_1}d\alpha_1+\cdots+ \frac{\partial P}{\partial \gamma_3}d\gamma_3=0$. This means that on the distinguished open set $D(\frac{\partial P}{\partial \alpha_1})$, every one-form except $d\alpha_1$ forms a basis of the tangent bundle, i.e., $\Omega^{123}_{\alpha_1}$ is a local frame on $D(\frac{\partial P}{\partial \alpha_1})$. Thus, $\Omega^{123}=\frac{\Omega_{\alpha_1}}{\frac{\partial P}{\partial \alpha_1}}$ does not have poles along any component of $D(\frac{\partial P}{\partial \alpha_1})$. The same argument holds if we replace $\alpha_1$ with any other eight variables. 

Notice that by the Jacobian criterion, the union of the nine open sets $D(\frac{\partial P}{\partial \alpha_i}),D(\frac{\partial P}{\partial \beta_i}),D(\frac{\partial P}{\partial \gamma_i})$ equals the smooth locus of $V(3,6)$, which has codimension at least $2$ because $V(3,6)$ is normal. Since $\Omega^{123}$ does not have poles along the union, $\Omega^{123}$ does not have poles along the entire chart $\{p_{123}\not=0\}$ that we are working on.

\begin{remark}\label{rmk:simple poles}
The function $p_{234}$ vanishes at order one on the subscheme $ZV(234)$ because one can easily check with Macaulay 2 that $ZV(234)$ in $V(3,6)$ is reduced and is equal to the union of positroid varieties $Z(234|156)\cup Z(34x:x\in [6])\cup Z(2345)$. Thus, if a rational form $\Omega$ has a pole on a closed subvariety $Z$ locally given by the vanishing of $p_{234}$ (i.e., a closed subvariety given by an irreducible component of $ZV(234)$) and if $\Omega = \Omega'\wedge  d\log p_{234}$ restricted on $Z$, then $\Res_Z \Omega =\Omega'$.
\end{remark}

We are now ready to prove that 
$\Omega$ has the required recursive properties.

\begin{theorem}\label{thm:canonical form on V(3,6)}
The $8$-form $\Omega$ has simple poles exactly along the union of the following components of $\Gr(3,6)$.
    \begin{itemize}
        \item[(a)] $Z(A|B)$, where $A,B$ are cyclic intervals, $A\sqcup B=[6]$, and $|A|,|B|\geq2$, in which case $Z(A|B)=ZV(A|B)$.
        \item[(b)] $Z(12x)= V(12x:x\in [6])$ and cyclic rotations thereof.
    \end{itemize}
Furthermore, the residue along the above two families of poles are equal to the canonical forms on the corresponding pole component as a positroid variety.
   \end{theorem}
\begin{proof}
    Notice that $\Omega$ is obtained from pushing forward the canonical form on $\Gr(2,6)$. So $\Omega$ is $D_6$-symmetric, where $D_6$ is the dihedral group acting on the six columns. Thus, each pole is also $D_6$-symmetric. We divide the proof in two parts.

\noindent {\bf Part 1: Poles of $\Omega$ along the divisors of type (a) and (b).}

We first compute the residue of $\Omega$ along the irreducible components of $ZV(234)$, i.e., $Z(234|156)$, $Z(34x)$, $Z(2345)$, following Remark \ref{rmk:simple poles}. Note that, in our local chart, $p_{234}=-\alpha_1$. \\
{\bf Residues on $Z(234|156)$.} From the first expression of $\Omega$, we have $\Res_{Z(234|156)} \Omega = \frac{\Omega_{\gamma_3,\alpha_1}}{p_{345}p_{126}p_{145}p_{256}} $. Note that $\Res_{Z(234|156)} \Omega$ is a form on an affine chart of $Z(234|156)$. Now, notice that the relation $p_{456}p_{125}=p_{145}p_{256}$ holds on $Z(234|156)$. So we may write $\Res_{Z(234|156)} \Omega = \frac{\Omega_{\gamma_3,\alpha_1}}{p_{345}p_{126}p_{145}p_{256}} =\frac{\Omega_{\gamma_3,\alpha_1}}{p_{345}p_{126}p_{456}p_{125}}$. We point out that the last equality holds when we restrict to $Z(234|156)$: it does not hold on the affine chart of $\Gr(3,6)$ or on $V(3,6)$.  
Finally, we verify that
$\frac{\Omega_{\gamma_3,\alpha_1}}{p_{345}p_{126}p_{456}p_{125}}$ is the canonical form on $Z(234|156)$.
This is because $\frac{\Omega_{\gamma_3,\alpha_1}}{p_{345}p_{126}p_{456}p_{125}}\wedge d \log p_{156}\wedge d\log p_{234}$ equals the canonical form on $\Gr(3,n)$.

In the computations we are using $\Omega_{\gamma_3}$, which is a local frame on $D(\frac{\partial P}{\partial \gamma_3})$. The above arguments show that if we do have a simple pole on $Z(234|156)$, its residue on $Z(234|156)$ restricted to $Z(234|156)\cap D(\frac{\partial P}{\partial \gamma_3})$ equals the canonical form. What we need is that the residue on $Z(234|156)$ agrees with the canonical form before restricting.

The set $Z(234|156)\cap D(\frac{\partial P}{\partial \gamma_3})$ is dense in $Z(234|156)$, as the smaller set $Z(234|156)\cap D(\alpha_3\beta_3C_{12})$ is easily seen to be dense in $Z(234|156)$.

Now, density implies that the pole of $\Omega$ intersected with $Z(234|156)$ is the Zariski closure of the pole of $\Omega_{|_{D(\frac{\partial P}{\partial \gamma_3})}}$ intersected with $Z(234|156)\cap D(\frac{\partial P}{\partial \gamma_3})$. The form  $\frac{\Omega_{\gamma_3}}{p_{234}p_{345}p_{126}p_{145}p_{256}}$ has a simple pole on $Z(234|156)\cap D(\frac{\partial P}{\partial \gamma_3})$ because the residue computation gives us the canonical form on $Z(234|156)$, which is nonzero. So we conclude that $\Omega$ has a simple pole on the closure $Z(234|156)$. 

Finally, since we already have $\Res \Omega$ well defined on $Z(234|156)$, agreeing with the canonical form on a dense open set, it must be equal to the canonical form. This is because two rational forms agreeing on a dense open set must be equal.\\
{\bf Residues on $Z(34x)$ and $Z(2345)$.}
On $Z(34x)$ and $Z(2345)$, a similar computation shows the same result. We need to use the relations $p_{456}p_{235}=p_{356}p_{245}$ on $Z(34x)$ and $p_{456}p_{134}=p_{346}p_{145}$ on $Z(2345)$. Finally, we conclude the following. \\
This discussion shows that, given the three local expressions for $\Omega$ 

\[
\Omega=\frac{\Omega_{\gamma_3}}{p_{234}p_{345}p_{126}p_{145}p_{256}} = \frac{\Omega_{\gamma_1}}{p_{234}p_{126}p_{156}p_{356}p_{245}}=\frac{\Omega_{\alpha_2}}{p_{234}p_{156}p_{126}p_{346}p_{145}},
\]

\begin{enumerate}

\item[(i)] the residue of $\frac{\Omega_{\gamma_3}}{p_{234}p_{345}p_{126}p_{145}p_{256}}$ along $p_{234}=0$ equals the canonical form on $Z(234|156)$.

\item[(ii)] the residue of $\frac{\Omega_{\gamma_1}}{p_{234}p_{126}p_{156}p_{356}p_{245}}$ along $p_{234}=0$ equals the canonical form on $Z(34x)$.

\item[(iii)] the residue of $\frac{\Omega_{\alpha_2}}{p_{234}p_{156}p_{126}p_{346}p_{145}}$ along $p_{234}=0$ equals the canonical form on $Z(2345)$. 

\end{enumerate}

\noindent {\bf Part 2: There are no other poles.}
We verified that $\Omega_{V(3,6)}$ has poles along the desired divisors (a) and (b). We now show that these are the only poles.

From the expression $\Omega_{V(3,6)}=\frac{p_{135}}{p_{156}p_{345}}\Omega^{123}$ on the chart $\{p_{123}\not=0\}$, we know that the poles of $\Omega_{V(3,6)}$ are contained in the union $ZV(123)\cup ZV(345)\cup ZV(561)$.

By cyclic symmetry, the poles are also contained in $ZV(234)\cup ZV(456)\cup ZV(612)$.

One easily use Macaulay 2 to compute the scheme theoretic intersection $$(Z(123)\cup Z(345)\cup Z(561))\bigcap(Z(234)\cup Z(456)\cup Z(612)) $$ inside the variety $G(3,6)$ and see that it is exactly equal to the union of the $2$ families of divisors above. In fact, in order to see that there are no other poles, we only need the statement for set theoretic intersection. This can easily be verified by looking at what Plucker coordinates are required to vanish.
\end{proof}

\begin{remark}\label{open problem motivation}
    The cyclic symmetry plays a crucial role in the proof of the above theorem than simply shortening the argument. The way the above proof was written hides it under the rug, which we now make explicit. 

    In {\bf Part 2} of the above proof, we invoked cyclic symmetry of our candidate form, which is a smart argument. Instead, one can try to work on local charts in coordinates and verify with brute force that indeed our candidate form do not have poles along ``non-positroid" subvarieties of $ZV(123)$ such as $V(1345)$ or its cyclic rotations. It turns out that it is easy to show that $\Omega$ does not have poles along $V(1345)$ if we work on the chart $D(p_{123})$ in local coordinates. However, it is very difficult to show that $\Omega$ does not have poles along $V(2456)$ on the same chart $D(p_{123})$. The previous implies the latter by cyclic symmetry.

    Now, cyclic symmetry is an easy consequence of writing our candidate form as a push-forward of the canonical form on $\Gr(2,n)$. However, cyclic symmetry is not at all obvious for the expression in Example \ref{Canonical form n=6}, which is what we mostly have been using in this section. Since the push-forward formulation is our only source of cyclic symmetry, we cannot avoid the push-forward machinery and instead simply give the correct canonical form in local coordinates and extend. From another perspective, this means that we do not have a good way to write down the canonical form on $V(3,n)$. This motivates open problems in the last section.
\end{remark}

\begin{theorem}\label{thm:V(3,6) is a positive geometry}
The variety $V(3,6)$ is a normal positive geometry. 
\begin{proof}
Theorem \ref{thm:canonical form on V(3,6)} shows that the poles of $\Omega$ are along the Zariski closures of $\partial V_{\geq0}$ and the residues coincide with the canonical forms of its components, each of which is a positive geometry itself. This verifies the recursive definition of a positive geometry for the pair $(V(3,6),V(3,6)_{\geq 0})$ and the given $\Omega$; see Definition \ref{def: positive geom}. By Proposition \ref{prop:V(3,n) is normal} $V(3,6)$ is normal. Uniqueness of $\Omega$ up to scaling is given by rationality from Lemma \ref{lem:V(3,n) is rational}, normality from Proposition \ref{prop:V(3,n) is normal}, and Remark \ref{uniqueness of canonical form}.

Strictly speaking, we need also to verify the technical topological assumptions in the first paragraph of Definition \ref{def:iterated boundary} for $V(3,6)$ and its iterated boundaries. Namely, its TTN part is the analytic closure of its analytic interior and its analytic interior is an oriented real manifold whose dimension coincides with the Krull dimension of the variety. This statement is automatic for the iterated boundaries of $V(3,6)$ because they are all positroid varieties, which are known to be positive geometries. For $V(3,6)$ itself, we delay this discussion to the last section, where we directly prove the general case.
\end{proof}
\end{theorem}

\section{Constructing Positive Geometries}\label{sec: pos geoms as fibrations over p^1}
In this section, we introduce two key constructions of new positive geometries from existing ones. Combining these two constructions will give the inductive step of the proof of the main theorem.

\begin{definition}[{\bf $X$-fibration over $\mathbb{P}^1$}]
A scheme $Y$ is an $X$-fibration over $\mathbb{P}^1$ if there is a morphism 
$\pi: Y\rightarrow \mathbb{P}^1$ and an open cover $\lbrace V_i\rbrace$ of $\mathbb{P}^1$ such that: 
\begin{enumerate}
\item[(1)] For every $i$, letting $U_i = \pi^{-1}(V_i)$, there is an isomorphism
\[
\phi_i: U_i \stackrel{\sim}{\longrightarrow} X\times V_i. 
\]

\item[(2)] For every pair $i,j$, there is $\rho_{ij}: X\times \Gamma(V_i\cap V_j) \rightarrow X$ such that the diagram 
\[
\begin{tikzcd}
U_i \cap U_j \arrow[r, "\phi_j"] \arrow[d, "\phi_i"']
& {X \times (V_i \cap V_j)} \arrow[d, "{\rho_{ij}\times 1}"] \\
{X \times (V_i \cap V_j)} \arrow[r, "id"]
& {X \times (V_i \cap V_j)}
\end{tikzcd}
\]
commutes. 
The data $\lbrace (V_i,\phi_i)\rbrace$ satisfying the properties (1) and (2) is called a {\it trivialization} of $Y$. 
\end{enumerate}
\end{definition}

In this paper, we will always take $\{V_i\}$ to be the standard affine chart on $\PP^1$ consisting of two affine lines. Next, we prove a technical result about fibrations over $\mathbb{P}^1$ in the context of positive geometries. 

\begin{theorem}\label{pos geom over P1}
Let $(X,X_{\geq0}, \Omega_X)$ be a positive geometry. Assume that $X$ and each iterated boundary component of $X$ is irreducible, normal, and rational.

Let $Y$ be an $X$-fibration over $\mathbb{P}^1$ and $Y_{\geq0}$ is a semialgebraic set in $Y(\mathbb{R})$ that is an orientable real manifold of full dimension.

The pair $(Y,Y_{\geq0})$ is said to be compatible with $(X,X_{\geq0}, \Omega_X)$ if the following hold

\begin{enumerate}
    \item[(i)] $Y$ has a trivialization over the open cover given by the two standard affine lines covering $\mathbb{P}^1$. i.e., $Y=U_1\cup U_2$ where $U_i$ are dense open subschemes of $Y$ with isomorphisms 
 \[
 \phi_1:U_1\to X\times \left\{\left[\frac{x}{y}:1\right]\right\}=X\times \mathbb{A}^1, \qquad \phi_2:U_2\to X\times \left\{\left[1:\frac{y}{x}\right]\right\}=X\times \mathbb{A}^1.
 \] 
The transition isomorphism
 \[
 \phi_2\circ \phi_1^{-1}: X\times \spec \mathbb{C}\left[\frac{x}{y}\right]_{\frac{x}{y}} \to  X\times \spec \mathbb{C}\left[\frac{y}{x}\right]_{\frac{y}{x}}\]
is given by 
\[
(x, t)\mapsto (\rho(x,t),t^{-1}), \qquad x\in X, t\in \mathbb{C}^*
\] 
for a morphism $\rho:X\times \mathbb{C}^*\to X$. Similarly for $\phi_1\circ \phi_2^{-1}$. 

\item[(ii)] Total nonnegativity  is compatible. Namely, $Y_{\geq0} =\phi_1^{-1}(X_{\geq0} \times \mathbb{A}^1_{\geq0})\bigcup \phi_2^{-1}(X_{\geq0} \times \mathbb{A}^1_{\geq0})$.
In fact, there exists such a semialgebraic set $Y_{\geq0}$ if and only if the transition isomorphisms $\phi_2\circ \phi_1^{-1}, \phi_1\circ \phi_2^{-1}$ both send totally nonnegative part of one chart to the totally nonnegative part of the other. In this case, $Y_{\geq0}$ is uniquely defined as the union of the totally nonnegative parts of the charts.

\item[(iii)] Each boundary component $X_i$ of $X$ gives an $X_i$-fibration over $\mathbb{P}^1$ induced by $b_{X_i}:Y_{X_i}\hookrightarrow Y\to \mathbb{P}^1$.
i.e., the transition isomorphisms $\phi_2\circ \phi_1^{-1}, \phi_1\circ \phi_2^{-1}$ restrict to isomorphisms between 
\[
X_i\times \spec \mathbb{C}\left[\frac{x}{y}\right]_{\frac{x}{y}} \simeq  X_i\times \spec \mathbb{C}\left[\frac{y}{x}\right]_{\frac{y}{x}}.
\]
In other words, the map $\rho_t:X\to X, x\mapsto \rho(x,t)$ restricts to a morphism $X_i\to X_i$ for any $t\in \mathbb{C}^*$.
In particular, this implies that total positivity is compatible. Namely, the above condition holds for interiors as well.

\end{enumerate}

If the three conditions above hold, then $(Y,Y_{\geq0})$ has boundaries given by
\begin{itemize}
    \item $Y_{X_i}$ as in $(iii)$, where $X_i$'s are all the boundary components of $X$.
    \item The two $X$-fibers of $\pi:Y\to \mathbb{P}^1$ over the points $\{[1:0]\}$ and $\{[0:1]\}$ in $\mathbb{P}^1$. 
\end{itemize} 

Moreover, suppose that $(X, X_{\geq 0}, \Omega_X)$ is a positive geometry and 

\begin{enumerate}
    \setcounter{enumi}{3}
\item[(iiibis)] Condition (iii) above holds true, with the boundary component $X_i$ of $X$ replaced by any iterated boundary component $Z$ of $X$.
\item[(iv)] The canonical forms on $U_1$ and $U_2$ glue. Namely, $(\phi_2\circ \phi_1^{-1})^*(\Omega_X\wedge d\log(y/x))=- \Omega_X\wedge d\log(x/y)$. Then, $\Omega_Y$ given by ${\Omega_Y}_{|_{U_1}}:=\phi_1^*(\Omega_X\wedge d\log(x/y))$ and ${\Omega_Y}_{|_{U_2}}:=\phi_2^*(-\Omega_X\wedge d\log(y/x))$ is well-defined.
 \end{enumerate}

Then $(Y,Y_{\geq0},\Omega_Y)$ is a positive geometry with iterated boundaries given by

\begin{itemize}
    \item $Y_Z$ as in (iii).
    \item $Z$ embedded in the $X$-fibers of $\pi:Y\to \mathbb{P}^1$ over the points $\{[1:0]\}$ and $\{[0:1]\}$ in $\mathbb{P}^1$, 
    \end{itemize} 
where $Z$ ranges through all iterated boundaries of $X$. Furthermore, $Y$ and each iterated boundary component of $Y$ is irreducible, normal, and rational.
\end{theorem}

\begin{proof}
Condition (ii) guarantees that $Y_{\geq0}$ is a well-defined orientable semialgebraic subset of $Y$, locally on either chart given by total nonnegative part on $X\times \mathbb{A}^1$. 

Since $Y$ is the union of two dense open sets $U_1$ and $U_2$, we see the following equality between analytic interiors: 
\[
\itr Y_{\geq0}=\itr (Y_{\geq0}\cap U_1) \cup \itr (Y_{\geq0}\cap U_2).
\]
To prove this, first notice that the right-hand side sits inside $\itr Y_{\geq 0}$, because it is an open set contained in $ (Y_{\geq0}\cap U_1) \cup  (Y_{\geq0}\cap U_2)=Y_{\geq0}$. Now, for any $p\in \itr Y_{\geq 0}$, we may assume that $p\in U_1$, because $U_1\cup U_2=Y$. Since $p$ is an interior point, there exists a neighborhood $V$ containing $p$; by passing to $V\cap U_1$, we can suppose that $V\subset U_1$. Thus, $p\in\itr (Y_{\geq0}\cap U_1)$. This shows the other containment.  

So the analytic interior of $Y_{\geq0}$ can be computed after restricting to $U_1$ and $U_2$, where on each chart we have the trivializing isomorphisms $\phi_1$ and $\phi_2$ from condition $(i)$. Recall that the inverse functon theorem implies that $\itr(\psi^{-1}(S)) = \psi^{-1}(\itr S)$, for any diffeomorphism $\psi$ and any smooth real manifold $S$. Thus 
\[
\itr (Y_{\geq0}\cap U_1)=(\phi_1)^{-1}(\itr(X_{\geq0} \times \mathbb{A}^1_{\geq0}))=(\phi_1)^{-1}(X_{>0} \times \mathbb{A}^1_{>0})).
\] 
Similarly for $\itr (Y_{\geq0}\bigcap U_2)$. The analytic boundary of $Y_{\geq 0 }$ is then the complement (denoted by a $c$) of the union 
\[
(\phi_1)^{-1}(X_{>0} \times \mathbb{A}^1_{>0}) \cup (\phi_2)^{-1}(X_{>0} \times \mathbb{A}^1_{>0})
\] 
in $Y_{\geq0}$. We see that
\[
\begin{aligned}
 Y_{\geq0}\cap U_1 \cap ((\phi_1)^{-1}(X_{>0} \times \mathbb{A}^1_{>0})) \cup (\phi_2)^{-1}(X_{>0} \times \mathbb{A}^1_{>0})))^c&=Y_{\geq0}\cap U_1\cap ((\phi_1)^{-1}(X_{>0} \times \mathbb{A}^1_{>0}))^c\\&=\phi_1^{-1}(\partial X_{\geq0 }\times \mathbb{A}^1_{\geq0} \cup X_{\geq0}\times  \{0\})\end{aligned}
 \] 
Here the first equality comes from the definition of $\phi_1$ and $\phi_2$. We now show the second equality.   

First, notice that, by (i) and (ii), $Y_{\geq0}\cap U_1=\phi_1^{-1}(X_{\geq0}\times \mathbb{A}^1_{\geq 0})$. Indeed, suppose by contradiction
that there exists $q\in \phi_2^{-1}(X_{\geq0}\times \mathbb{A}^1_{\geq 0})\cap Y_{\geq0}\cap U_1$ and $q\notin \phi_1^{-1}(X_{\geq0}\times \mathbb{A}^1_{\geq 0})$. Since the transition isomorphism $\phi_1\circ\phi_2^{-1}$ respects totally nonnegative parts by (iii), we would
have $\phi_1(q)\in X_{\geq0}\times \mathbb{A}^1_{\geq 0}$, contradiction the second assumption. Hence $Y_{\geq0}\cap U_1=\phi_1^{-1}(X_{\geq0}\times \mathbb{A}^1_{\geq 0})$. Hence

\[
\begin{aligned}
&Y_{\geq0}\cap U_1\cap ((\phi_1)^{-1}(X_{>0} \times \mathbb{A}^1_{>0}))^c
=(\phi_1)^{-1}(X_{\geq0}\times \mathbb{A}^1_{\geq 0}) \cap ((\phi_1)^{-1}(X_{>0} \times \mathbb{A}^1_{>0}))^c\\
&= (\phi_1)^{-1}((X_{\geq0}\times \mathbb{A}^1_{\geq 0}) \cap (X_{>0} \times \mathbb{A}^1_{>0})^c)=\phi_1^{-1}(\partial X_{\geq0 }\times \mathbb{A}^1_{\geq0} \cup X_{\geq0}\times  \{0\}).
\end{aligned}
\] 
Similarly for $Y_{\geq0} \cap U_2$. 

Notice that $\phi_1^{-1}( X_{\geq0}\times  \{0\})$ is the totally nonnegative part of the fiber $\pi^{-1}([0:1])\simeq X$; $\phi_2^{-1}( X_{\geq0}\times  \{0\})$ is the totally nonnegative part of the fiber $\pi^{-1}([1:0])\simeq X$. Finally, $\phi_1^{-1}(\partial X_{\geq0 }\times \mathbb{A}^1_{\geq0})=\bigcup_i \phi_1^{-1}((X_i)_{\geq0}\times \mathbb{A}^1_{\geq0})$, where $X_i$'s are all the boundaries of $X$. By (iii), for each $i$, $\phi_1^{-1}((X_i)_{\geq0}\times \mathbb{A}^1_{\geq0})\cup \phi_2^{-1}((X_i)_{\geq0}\times \mathbb{A}^1_{\geq0})=(Y_{X_i})_{\geq0}$. Thus the boundary of $Y_{\geq0}$ is 
\[
\partial Y_{\geq0}=\pi^{-1}([0:1])_{\geq0 }\cup \pi^{-1}([1:0])_{\geq0 } \cup \bigcup_{i} (Y_{{X_i}})_{\geq0}. 
\]
Now, the Zariski closure of $X_{\geq0}$ is $X$ since $X$ is a positive geometry. Applying this to each fiber $\pi^{-1}([0:1])\simeq  \pi^{-1}([1:0])\simeq X $, we have that the Zariski closures of $\pi^{-1}([0:1])_{\geq0 }, \pi^{-1}([1:0])_{\geq0 } $ are $\pi^{-1}([0:1]), \pi^{-1}([1:0])$, respectively. Also, we have $(Y_{X_i})_{\geq0}=\phi_1^{-1}((X_i)_{\geq0}\times \mathbb{A}^1_{\geq0})\cup \phi_2^{-1}((X_i)_{\geq0}\times \mathbb{A}^1_{\geq0})$ and the Zariski closure of $(X_i)_{\geq0}\times \mathbb{A}^1_{\geq0}$ is $X_i\times \mathbb{A}^1$ since $(X_i,(X_i)_{\geq0})$ is a boundary component. So the Zariski closure of $Y_{\geq0}$ is the union $\phi_1^{-1}(X_i\times \mathbb{A}^1)\cup \phi_2^{-1}(X_i\times \mathbb{A}^1) =Y_{X_i}$.

So far, we showed that the algebraic boundary components of $(Y,Y_{\geq0})$ are $\pi^{-1}([0:1])$, $\pi^{-1}([1:0])$, and the $Y_{X_i}$'s; the analytic boundary components are the corresponding totally nonnegative parts. The iterated boundaries of $\pi^{-1}([0:1])\simeq \pi^{-1}([1:0])\simeq X$ are already part of the data of the positive geometry $X$. As for the iterated boundaries of $Y_{X_i}$, we note that the pair $Y_{X_i}$ and $X_i$ satisfy the same four conditions (i)-(iv) in the statement. Thus, the above argument can be applied again with $Y$ and $X$ replaced by $Y_{X_i}$ and $X_i$. An induction implies the statement about iterated boundaries when condition (iii) is replaced by (iiibis).

We are now ready to show that $(Y,Y_{\geq0})$ is a positive geometry. First, we verify the topological conditions described in the first paragraph of Definition \ref{def:iterated boundary}. Namely, $Y_{\geq0}$ is the analytic closure of its analytic interior $Y_{>0}$ and $Y_{>0}$ is an orientable manifold of dimension equal to the Krull dimension of $Y$. Notice that (ii) and (iii) gurantee that $Y\to \mathbb{P}^1$ restrictes to a fibration $Y_{\geq0}\to \mathbb{P}^1_{\geq0}$ and $Y_{>0}\to \mathbb{P}^1_{\geq0}$. Since $\mathbb{P}^1_{\geq0}$ is homeomorphic to an interval $I$ and thus contractible, $Y_{\geq0}$ and $Y_{>0}$ are homeomorphic to the trivial bundle $X_{\geq0}\times I$ and $X_{>0}\times I$. Thus, the topological condition follows from the assumption that it holds for $X_{\geq0}$.

Condition (iv) gives a candidate canonical form $\Omega_Y$. By induction, each boundary component (either $X$ or $Y_{X_i}$) is already a positive geometry. We have to show that $\Omega_Y$ is the unique form -- with only simple poles exactly along the boundary components -- whose residue along those components give the canonical form on each of them. The uniqueness part follows from Remark \ref{uniqueness of canonical form}. The form $\Omega_Y$ is easily seen to have simple poles on the boundary and no extra poles, because its restriction to $U_1$ and $U_2$ both have simple poles along the desired boundaries and no other poles. Finally, for the residue along $Y_{X_i}$, notice that taking a residue is a local computation and may be perfomed locally on the chart $U_1$. The residue is thus 
\[
\phi_1^*(\Res (\Omega_X\wedge d\log (x/y))=\phi_1^*(\Res \Omega_X\wedge d\log(x/y))=\phi_1^*( \Omega_{X_i}\wedge d\log(x/y)),
\] 
where the first equality holds because we are taking the residule along $X_i\times \mathbb{A}^1$ and so $d\log \frac{x}{y}$ is independent. Now, the right-hand side is the canonical form on $Y_{X_i}$ restricted to $U_1\bigcap Y_{X_i}$ by inductive hypothesis. As for the residue along the fiber $\pi^{-1}([0:1])$, we consider the chart $U_1$ and then the fiber is cut out by the function $x/y$, which vanishes in degree $1$. So the residue can be obtained by removing $\wedge d\log(x/y)$, thus producing $\Omega_X$, the canonical form on $X$ indeed. We have thus shown that $Y$ is a positive geometry. 

As for the last sentence in the statement, it is clear from condition (i) that $X$ irreducible, normal, rational implies the same properties for $Y$. Again, for each iterated boundary $Z$, $Y_Z\to \mathbb{P}^1$ has $Z$-fibers and the same four conditions in the statement hold for each pair $Y_Z$ and $Z$. Thus, $Z$ irreducible, normal, rational implies the same properties for $Y_Z$. All iterated boundaries of $Y$ are isomorphic to either $Z$ or $Y_Z$. This finishes the proof. 
\end{proof}

We are going to apply this theorem to Grassmannians and subvarieties in the Grassmannians. We first discuss a simple example.

\begin{example}\label{Hirzebruch} ({\bf Hirzebruch surfaces are positive geometries})

    Let $(X,X_{\geq0},\Omega_X)=(\mathbb{P}^1,(\mathbb{P}^1)_{\geq 0},d\log (\frac{x}{y}))$ be the $1$-dimensional projective simplex. Consider the Hirzebruch surface $F_n=\mathbb{P}(\mathcal{O}\oplus \mathcal{O}(n))$ for $n\geq0$. Then $F_n$ is a $\mathbb{P}^1$-bundle over $\mathbb{P}^1$ with the following explicit description of trivializing charts.
    
    Let $X=\proj \mathbb{C}[a,b]$. Using the notations established in Theorem \ref{pos geom over P1}, let the base space $\mathbb{P}^1$ have projective coordinates $\proj \mathbb{C}[x,y]$. Then the transition isomophism is given by $$\rho([a:b], t)=[a:t^n b] $$

    The positive part on $\mathbb{P}$ is given by the condition that all coordinates are positive. So, it's easy to see that the two transition isomorphism send positive part to positive part. Furthermore, the two boundary components of $X_{\geq0}=\mathbb{P}^1_{\geq 0}$ are given by $[1:0],[0:1]$. And they are indeed stable under $\rho$. Namely, $\rho([1:0],t)=[1:0],\rho([0:1],t)=[0:1]$ for any $t\in \mathbb{C}^*$. Finally, the transition isomorphism $\phi_2\circ \phi_1^{-1}$ sends $([a:b],t)\mapsto ([a:t^nb],t^{-1})$. And we verify that indeed, $d\log (\frac{a}{t^n b})\wedge d\log (\frac{1}{t})=-d\log (\frac{a}{b})\wedge d\log t$. So the pullback from the two charts glue to be a rational form on $F_n$.

    Thomas Lam pointed out that the positive part of $F_n$ is just $\mathbb{R}_{\geq0}\times \mathbb{R}_{\geq0}$. If we trivialize the $\mathbb{P}^{1}$-bundle over the affine line in $\mathbb{P}^1$ away from $[1:-1]$, we have an affine open set $\mathbb{P}^1\times \mathbb{A}^1$ containing the positive part. And the canonical form on $F_n$ restricted to this affine open is just the product $d\log (\frac{a}{b})\wedge d\log t$ of the canonical forms on the two components being projective or affine lines. Thus, one can easily see that the residue computation at poles must work out.

    The content of our theorem says that under mild conditions, the positive part of this above affine open set can still be realized as a semi-algebraic set in $F_n$. Furthermore, if you start with the affine open and take the correct form therein, you can extend it to the surface $F_n$ without introducing a pole along the fiber over $[1:-1]$. The assumptions (ii),(iii),(iv) in our theorem are essentially conditions on the chart morphisms so that naively pulling back the canonical form gives the correct answer. They are mainly to prevent chart choices such as gluing a $\mathbb{P}^1$-bundle using the transition map $\rho:([a:b],t)\mapsto [a+b:b]$. In this way, We end up with a space abstractly isomorphic to $F_1$. However, the pole $\{[0:1]\}\times \mathbb{A}^1$ on $U_1$ is sent to $[1:1]\times \mathbb{A}^1$, which is not a pole on $U_2$. The union of the preimages under the trivialization isomorphisms of the positive parts on the charts do not form a semi-algebraic set, and if we naively pullback canonical forms from the chart, they certainly do not glue.
    
   In the special case where $n=1$, $F_1$ equals the blowup of $\mathbb{P}^2$ at the origin. The projection $\pi_2:F_1\subset \mathbb{P}^2\times \mathbb{P}^1\to \mathbb{P}^1 $ to the second component gives $Y=F_1$ as a $\mathbb{P}^1$-bundle over $\mathbb{P}^1$. If we instead consider the projection $\pi_1:F_1\to \mathbb{P}^2$ to the first component, we get the blowup map $\Bl_{[0:0:1]} \mathbb{P}^2\to \mathbb{P}^2$. 

We explicitly compute the canonical form on $F_1=\{([\alpha:\beta:\gamma],[x:y])| \alpha y=\beta x\}\subset \mathbb{P}^2\times \mathbb{P}^1$ according to our theorem. 
We can let $\phi_1^{-1}$ send the point $([a:b], \frac{x}{y})\in X\times \mathbb{A}^1 $ to the point $([\frac{ax}{y}:a:b],[x:y])\in F_1$. Thus, $U_1$ consists of points in $F_1$ such that $y\not=0$. We take the form $d\log\frac{b}{a}\wedge d\log \frac{x}{y}$ and pullback via $\alpha=\frac{ax}{y},\beta=a,\gamma=b$. Since $\frac{b}{a}=\frac{\gamma}{\beta}, \frac{x}{y}=\frac{\alpha}{\beta}$, we have that the form on $F_1$ is equal to $d\log\frac{\gamma}{\beta}\wedge d\log \frac{\alpha}{\beta}$. 

For ease of notation, set $X_0=\frac{\alpha}{\gamma}, X_1=\frac{\beta}{\gamma}$ affine coordinates on $\mathbb{P}^2$ pulled-back on $\Bl \mathbb{P}^2$. We have $d\log\frac{1}{X_1}\wedge d\log \frac{X_0}{X_1}=-d\log X_1\wedge (\frac{X_1}{X_0} \cdot \frac{X_1dX_0-X_0dX_1}{X_1^2})$. Since we are taking the wedge with $d\log X_1$, we can ignore the $dX_1$ in the second term. We end up with $-d\log X_1\wedge (\frac{X_1}{X_0}\cdot \frac{X_1 dX_0}{X_1^2})=-d\log X_1\wedge d\log X_0$ is the canonical form on $\mathbb{P}^2$ pulled-back on $\Bl \mathbb{P}^2$.

It is known that well-behaved blowups of projective polytopes are positive geometries \cite{brauner2024wondertopes}. By their result, $F_1$ is indeed a positive geometry whose canonical form is equal to the pullback from $\mathbb{P}^2$ via $\pi_1$. It is worth noting that our theorem derives the positive geometry strucure from pulling back via $\pi_2$. And the next theorem we present gives the converse direction of \cite{brauner2024wondertopes} in the special case of $F_1\to \mathbb{P}^2$.

\end{example}

\begin{definition}\label{dfn:trivial subvariety}
    Let $(X,X_{\geq0})$ and $(Y,Y_{\geq0})$ be as in the first half of Theorem \ref{pos geom over P1}. A subvariety $S$ in $X$ is said to be trivial in $Y$ if the $S$-family over $\mathbb{P}^1$ induced as a subfamily of $Y$ by condition $(3)$ of Theorem \ref{pos geom over P1} is the product $S\times \mathbb{P}^1$. 
\end{definition}

\begin{definition}\label{retract}
    Let $(X,X_{\geq0})$ and $(Y,Y_{\geq0})$ be as in the first half of Theorem \ref{pos geom over P1}. Let $b_Z:Y_Z\hookrightarrow Y$ be as in condition (3) in Theorem \ref{pos geom over P1}. Let $S$ be an iterated boundary of $X$ trivial in $Y$.

 A reduced, irreducible, normal projective variety $R$ is said to be a {\it retract of $Y$ along $S$} if it comes with the following data
 \begin{itemize}
     \item A morphism $\psi: Y\to R$;
     \item Two closed embeddings $\iota_1,\iota_2:X\hookrightarrow R$ such that the precomposition with $\iota_S: S\hookrightarrow X$ agrees, i.e., $\iota_1\circ \iota_S = \iota_2\circ \iota_S$.      \end{itemize}
Moreover, the above morphims are required to satisfy the following properties. 
 \begin{enumerate}
    \item The morphism $\iota: Y\to R\times \mathbb{P}^1$ defined by $y\mapsto (\psi(y),\pi(y))$ a closed embedding. In particular, $\psi$ is a closed embedding composed with a projection, and thus is proper. So the following diagram commutes:
\[
\begin{tikzcd}
Y \arrow[r, "\iota"] \arrow[dr, "\psi"'] 
& R \times \mathbb{P}^1 \arrow[d, "\mathrm{pr}_1"] \\
& R
\end{tikzcd}
\]

\item the morphism $\psi$ restricted to $Y\setminus Y_{S}=Y\setminus(S\times \mathbb{P}^1)$ is an isomorphism onto its image, which is the open subscheme of $R$ given by the complement of $\iota_S:S\hookrightarrow R$. Note that $\psi(Y)=R$. 

\item The morphism $\iota_{Y_S}:S\times \mathbb{P}^1=Y_S\hookrightarrow Y \to R$ factors through $S$ by projecting to the first component and then embed through $\iota_S$. This amounts to say that the following diagram commutes:

\[
\begin{tikzcd}
S\times \mathbb{P}^1=Y_S \arrow[r, hook, "\iota_{Y_S}"] \arrow[d, "\mathrm{pr}_1"]
& Y \arrow[d, "\psi"] \\
S \arrow[r, hook]
& R
\end{tikzcd}
\]

\item $X\times \{[0:1]\}\hookrightarrow Y\to R$ is the closed embedding $\iota_1:X\hookrightarrow R$ and $X\times \{[1:0]\}\hookrightarrow Y\to R$ is the closed embedding $\iota_2:X\hookrightarrow R$.
\end{enumerate}
\end{definition}

\begin{example}\label{base case of induction of the next theorem}
    When $S=X$, $Y=Y_S$ equals the product $X\times \mathbb{P}^1$. Condition (2) and the assumption that $R$ is a variety requires the closed embedding $\iota_S:S\hookrightarrow R$ be an isomorphism, i.e., $R=X$. Thus, the retract of $X\times \mathbb{P}^1$ along $X$ is $X$. When $S=\emptyset$, a retract of $Y$ is $Y$ itself.
\end{example}

\begin{theorem} \label{retract of pos geo over P^1}
 Let $(X,X_{\geq0})$ and $(Y,Y_{\geq0})$ be as in the first part of Theorem \ref{pos geom over P1}. Let $b_Z:Y_Z\hookrightarrow Y$ be as in condition (iiibis) in Theorem \ref{pos geom over P1}.  Let $S$ be an iterated boundary of $X$ trivial in $Y$ and let $\psi:Y\to R$ and $\iota_{S}:S\hookrightarrow R$ be a retract of $Y$ along $S$.  

 Let $R_{\geq 0}$ be the image of $Y_{\geq 0}$ under $\psi:Y\to R$. Let $(Y_i,(Y_i)_{\geq0})$ be the algebraic and analytic boundary components of $Y$. Then, 

\begin{enumerate}
    \item[(i)] The boundary components of $(R,R_{\geq0})$ are exactly $(\psi(Y_{i}), \psi((Y_{i})_{\geq0}))$ for all $Y_i\not=Y_S$. In this case, $\psi(Y_{i})$ is the Zariski closure of $\psi((Y_{i})_{\geq0})$.
\end{enumerate}
Assume furthermore that $(X,X_{\geq0}, \Omega_X)$ is a positive geometry and the assumptions in the second part of Theorem \ref{pos geom over P1} holds. Thus, Theorem \ref{pos geom over P1} implies that $(Y,Y_{\geq0}, \Omega_Y)$ is a positive geometry with some canonical form $\Omega_Y$. Assume that $R$ and all iterated boundaries of $(R,R_{\geq0})$ are normal. Then, we have
\begin{enumerate}
    \setcounter{enumi}{1}
    \item[(ii)] $(R, R_{\geq 0},\psi_{*}\Omega_Y )$ is a normal positive geometry. 
\end{enumerate}
 \end{theorem}

\begin{proof}
If $S=X$, Example \ref{base case of induction of the next theorem} shows that the statement is satisfied. Assume $S\not=X$ and let $X_i$ be the boundary divisors of $X$. For any point $p\in R_{\geq 0}$ not in $\iota_S:S\hookrightarrow R$, condition (2) of Definition \ref{retract} implies that $p$ is in the interior of $R_{\geq0}$ if and only if its unique isomorphic preimage is in the interior of $Y_{\geq 0}$. 
    
 Now, take $p\in  R_{\geq 0}$ that is in the image of the morphism $\iota_S:S\hookrightarrow R$. Conditions (2) and (3) in Definition \ref{retract} imply that the preimage of $p$ under $\psi$ is $\{p\}\times \mathbb{P}^1\hookrightarrow S\times \mathbb{P}^1\hookrightarrow Y$. Note that the assumption $p\in R_{\geq0}:=\psi(Y_{\geq0})$ implies that some preimage of $p$ is in $Y_{\geq0}$. Thus, we must have that $\{p\}\times \mathbb{P}^1\hookrightarrow S_{\geq0}\times \mathbb{P}^1$ equals the preimage of $p$ under $\psi$. In particular, it contains $\{p\}\times {\mathbb{P}^1}_{\geq0}\subset  S_{\geq0}\times \mathbb{P}^1_{\geq0}$. The subvariety $S\times \mathbb{P}^1$ is an iterated boundary component of $Y$ by Theorem \ref{pos geom over P1}. So $p$ is the image of some point in the analytic boundary of $Y_{\geq0}$. Indeed, we show that $p$ must be in the analytic boundary of $R_{\geq0}$ as well. Suppose otherwise; then there exists an analytic open set $U\subset  R_{\geq 0}$ containing $p$. Since $\psi$ is the composition of a closed immersion with a projection, it is analytic continuous on the real points. So the preimage of $U$ is analytic open in $Y(\mathbb{R})$. The intersection of the preimage of $U$ with $Y_{\geq 0}$ is thus analytic open in $Y_{\geq 0}$ and contains all preimages of $p$ lying in $Y_{\geq0}$. Recall that $p$ has preimage $\{p\}\times \mathbb{P}^1\hookrightarrow S_{\geq0}\times \mathbb{P}^1\hookrightarrow Y$. Thus, we have an interval of preimages of $p$ contained in $S_{\geq0}\times \mathbb{P}^1_{\geq0}$ and thus we constructed an analytic open set in $Y_{\geq 0}$ containing them. This contradicts the fact that $S\times \mathbb{P}^1$ is an iterated boundary component of $Y_{\geq0}$.
    
 So far we have established that the analytic boundary of $ R_{\geq0}$ is the image of the analytic boundary of $Y_{\geq 0}$ under $\psi:Y\to R$, which equals the union $\bigcup \psi((Y_i)_{\geq0})$ over all boundary components of $Y$. However, we point out that not all images $\psi((Y_i)_{\geq0})$ are necessarily analytic boundary components of $R_{\geq 0}$. Notice that $\dim Y_{\geq0}=\dim R_{\geq0}$ by conditions (2) and (3) in Definition \ref{retract}. So, if $\dim \psi((Y_i)_{\geq0})=\dim (Y_i)_{\geq0}$, then $\dim \psi((Y_i)_{\geq0})=\dim (Y_i)_{\geq0}=\dim Y_{\geq0}-1=\dim R_{\geq0}-1$, and so $\psi((Y_i)_{\geq0})$ is indeed an analytic boundary component for $R_{\geq 0}$. However, some images might have lower dimensions and be contained in other boundary components. Theorem \ref{pos geom over P1} divides all the $Y_i$'s in two families. For each case, we determine when they form a boundary component. 

\textbf{Case 1}: $Y_i$ is a boundary component of $Y$ of the form $\pi^{-1}([0:1])=X\times [0:1]$ or $\pi^{-1}([1:0])=X\times [1:0]$. By condition (4) in Definition \ref{retract}, both have isomorphic image under $\psi$ as a subvariety in $R$ mapping positive parts to positive parts. Thus, dimensions are preserved and their images form boundary components.

\textbf{Case 2}: $Y_i$ is a boundary divisor on $Y$ of the form $Y_{X_i}$, for some boundary divisor ${X_i}$ on $X$. Recall that we assumed that $S\not=X$. So $S$ has codimension at least one in $X$. 
    \begin{enumerate}
        \item  If $X_i\not=S$, $X_i\setminus S$ is dense in $X_i$. In particular, $Y_{X_i}$ contains a dense open subset away from $Y_S=S\times \mathbb{P}^1$. Since $\psi$ is an isomorphism outside of $Y_S$, the scheme theoretic image of $Y_{X_i}$ has the same dimension as $Y_{X_i}$ and so do the positive parts.  We conclude that $\psi(Y_{X_i})$ is a boundary divisor on $\psi(Y)=R$. 
        
        \item If $X_i=S$, Condition $(3)$ in Definition \ref{retract} forces the image of $Y_S=S\times \mathbb{P}^1$ to be $S$. Thus, $\psi((Y_S)_{\geq0})=S_{\geq0}$ has dimension equal to $\dim (Y_S)_{\geq0}-1$. Indeed, $S$ is already contained in both $\iota_{1,2}:X\hookrightarrow R$, which are the images from {\bf Case 1} by condition (4) in Definition \ref{retract}.
        Thus, the image of $(Y_S)_{\geq0}$ does not form a boundary component. 
    \end{enumerate}
    
We conclude that as long as $Y_i\not=Y_S$, the image of $(Y_i)_{\geq0}$ under $\psi$ is an analytic boundary of $R_{\geq0}$ and has the same dimension as $(Y_i)_{\geq0}$. In this case, we also showed that the scheme theoretic image of $Y_i$ has the same dimension. Since each algebraic boundary $Y_i$ is by definition the Zariski closure of $(Y_i)_{\geq0}$, each $Y_i$ is reduced and irreducible. Thus, its scheme theoretic image is the Zariski closure of $\psi(Y_i)$ and is reduced and irreducible. The morphism $\psi$ is proper by condition (1) in Definition \ref{retract}, and so closed. Hence $\psi(Y_i)$ (being a Zariski closed set in $R$, endowed with the reduced scheme structure) is the scheme theoretic image of $Y_i$ that is reduced, irreducible, contains $\psi((Y_i)_{\geq 0})$, and these two have the same dimension. Therefore $\psi(Y_i)$ is the Zariski closure of $\psi((Y_{i})_{\geq0})$. This shows (i).

To show (ii), we first verify the topological conditions described in the first paragraph of Definition \ref{def:iterated boundary}. Namely, we need to show that $R_{\geq0}$ is a semialgebraic subset and equals the analytic closure of its analytic interior $R_{>0}$ and $R_{>0}$ is an orientable manifold of dimension equal to the Krull dimension of $R$. First, $R_{\geq0}$ is by choice the image of $Y_{\geq0}$. By Condition (i) in Definition \ref{retract}, $Y_{\geq0}$ stays semialgebraic after the embedding in $R\times \mathbb{P}^1$ and thus has semialegbraic image in $R$ by the Tarski–Seidenberg theorem. It is furthermore analytically closed because the map $\psi:Y(\mathbb{R})\to R(\mathbb{R})$ is a closed embedding followed by a projection away from a compact space $\mathbb{RP}^1$, and is thus universally closed. Notice from previous discussion that the analytic boundary $\partial R$ is the isomorphic image of $\psi(\partial Y)$ since $Y_S$ is contained in $\partial Y$. So $\partial R$ being an orientable manifold of the correct dimension follows from the fact that $Y_{\geq0}$ is. Finally, $R_{\geq0}$ equals the analytic closure of $R_{>0}$ because $R_{\geq0}$ is by choice the image of $Y_{\geq0}$. Then, any sequence in $Y_{>0}$ approaching some preimage of a point in $R_{\geq0}$ gives a sequence in $R_{>0}$ approaching the point in $R_{\geq0}$.

We now prove (ii) by induction on $\dim X$. Fix a retract $\psi:Y\to R$ with $\iota_{1,2}:X\hookrightarrow R$; it follows from the proof of part (i) that $(R,R_{\geq0})$ has boundary components divided into two families. To show that $\psi_{*}(\Omega_Y)$ is the canonical form making $(R,R_{\geq0})$ a positive geometry, we have to show that it is the unique form with only simple poles exactly along the boundary components, whose residue along those components is the canonical form on each boundary component. The uniqueness follows from Remark \ref{uniqueness of canonical form}. Notice that $\psi_{*}\Omega_Y$ contains no extra poles other than the boundary components of $R$ because of the following: 
    \begin{itemize}
        \item If $\psi_{*}\Omega_Y$ has a pole along $Z\subset R$, then $\Omega_Y$ has a pole along $\psi^{-1}(Z)\subset Y$.
        
        \item The form $\Omega_Y$ contains no poles other than the boundary components of $Y$. 
        
        \item The morphism $\psi$ restricts to an isomorphism outside of $Y_S$ and maps $Y_S$ onto $S$. 

       \item $\partial R=\psi(\partial Y)$ and $S\subset \partial R$.
    \end{itemize}
    
Since $\psi$ is an isomorphism outside $Y_S$, the poles outside $S\subset R$ are inside $\partial Y$. On the other hand, the poles 
on $S$ are in $\partial R = \psi(\partial Y)$. Hence $\psi_{*}\Omega_Y$ has no extra poles. It remains to show that each boundary component is a positive geometry whose canonical form is given by the residue of $\psi_{*}\Omega_Y$ along the corresponding component. This suffices because the residue being nonzero automatically implies that they must be simple poles. 

\textbf{Case 1}: The boundary component is $\iota_1:X\times \{[0:1]\}\hookrightarrow Y\to R$ or $\iota_2:X\times \{[1:0]\}\hookrightarrow Y\to R$. In both cases, the boundary component is isomorphic to $X$ and is already known to be a positive geometry with canonical form $\Omega_X$. Since the residue computation is a local computation and thus can be done on a dense open subset, we can do the residue computation outside of $S\hookrightarrow R$ because $S\not=X$ by assumption. Now, $\psi$ restricts to an isomorphism outside of $Y_S\to S$, so the validity of the residue computation for $X$ as a boundary of $Y$ implies that that of the residue computation for $X$ as a boundary of $R$. 

\textbf{Case 2}: The boundary component is $\psi(Y_{X_i})$ for some boundary component $X_i\not=S$ of $X$. We first establish that $(\psi(Y_{X_i}), \psi((Y_{X_i})_{\geq0}), \psi_{*}\Omega_{Y_{X_i}})$ is a positive geometry, where $\psi$ is understood as the restriction of $\psi:Y\to R$ on $Y_{X_i}\to \psi(Y_{X_i})$. First, notice that
\begin{itemize}
    \item The pair $(Y_{X_i},(Y_{X_i})_{\geq0})$ and $(X_i,(X_i)_{\geq 0})$ are related as in Theorem \ref{pos geom over P1}.
    \item $S\cap X_i$ is a subvariety of $X_i$ that is trivial in $Y_{X_i}$, because $S$ is a subvariety of $X$ trivial in $Y$.
\end{itemize}  
Therefore, if we can show that $\psi(Y_{X_i})$ is a retract of $Y_{X_i}$ along $S\cap X_i$ in the sense of Definition \ref{retract}, this decreases $\dim X$. Now, we invoke induction, with two possible cases. Either there is a step of the induction procedure such that the trivial subscheme satisfies $S\cap X_i = X_i$ or there is none. In the first case, we use again Example \ref{base case of induction of the next theorem} showing 
that $(\psi(Y_{X_i}), \psi((Y_{X_i})_{\geq0}), \psi_{*}\Omega_{Y_{X_i}})$ is a positive geometry. Otherwise, in the second case, we invoke 
induction with base case $\dim X=0$ and $Y\cong \mathbb{P}^1$. 

Consider $\psi:Y_{X_i}\to \psi(Y_{X_i})$ given by the restriction of $\psi:Y\to R$. Also, precomposing $X_i\hookrightarrow X$ with $\iota_1: X\times \{[0:1]\}\hookrightarrow Y\to R$ and $\iota_2: X\times \{[1:0]\}\hookrightarrow Y\to R$ gives two maps $\iota_{1,2}'$. Notice that for both $\iota_{1,2}$, the first map $X\hookrightarrow Y$ embeds a copy of $X$ as the fiber of $Y\to \mathbb{P}^1$ over $\{[0:1]\}$ and $\{[1:0]\}$. Thus, after precomposing with $X_i\hookrightarrow X$, $\iota'_{1,2}$ embeds a copy of $X_i$ as the fiber of $Y_{X_i}\to \mathbb{P}^1$ over $\{[0:1]\}$ and $\{[1:0]\}$. Then, their images under $\psi:Y\to R$ land in $\psi(Y_{X_i})$. $Y_{X_i}$ is reduced because $X_i$ is reduced, and since $\psi$ is proper, $\psi(Y_{X_i})$ is the scheme theoretic image. So we have morphisms $\iota'_{1,2}:X_i\hookrightarrow Y_{X_i}\to \psi(Y_{X_i})$. Notice that precomposing $\iota'_{1,2}$ with $S\cap X_i$ agrees simply because precomposing $\iota_{1,2}$ with $S$ agrees. Now, we verify the four conditions in Definition \ref{retract} required to be satisfied by $\psi$ and $\iota_{1,2}$: 
	
\begin{enumerate}
    \item The image of $Y_{X_i}$ under the closed embedding $\iota:Y\hookrightarrow R\times \mathbb{P}^1$ is contained in $\psi(Y_{X_i})\times \mathbb{P}^1$. So $\iota$ restricts to a closed embedding $Y_{X_i}\to R\times \mathbb{P}^1$.
    \item $\psi$ restricted to $Y_{X_i}\setminus Y_{S\cap X_i}$ is an isomorphism onto the correct open subshcheme of $\psi(Y_{X_i})\subset R$ because $\psi$ restriced to $Y\setminus Y_S$ is an isomorphism onto the correct image in $R$.
    \item $\psi$ restricted to $Y_{S\cap X_i}$ is the projection to the first component because so is $\psi$ restricted to $Y_S$.
    \item This is satisfied by construction. 
\end{enumerate}

We conclude that $(\psi(Y_{X_i}), \psi((Y_{X_i})_{\geq0}), \psi_*\Omega_{Y_{X_i}})$ is a positive geometry by induction on $\dim X$. 

It remains to verify that $\psi_*\Omega_{Y_{X_i}}$ is the residue of $\psi_*\Omega_Y$ along $\psi(Y_{X_i})$. Again, $\psi$ restricts to an isomorphism away from $Y_S$, which is dense in $Y_{X_i}$ because $S\not=X_i$. Since taking the residue of a form is a local computation, the residue of $\psi_*\Omega_Y$ along $\psi(Y_{X_i})$ can be computed away from $Y_S$. Thus, the residue computation is valid for $\psi_*\Omega_Y$ because so is for $\Omega_Y$, since $Y$ is a positive geometry. 
\end{proof}

\begin{example}({\bf Projective cones over rational normal curves are positive geometries})

    In the above proof, it appears that contracting a $\mathbb{P}^1$-worth of $S$ causes no effect on the positive geometry structure. The following example makes it manifest.

    Going along Example \ref{Hirzebruch}, let $X=\mathbb{P}^1, Y=F_n$. The Hirzebruch surface $F_n$ can be constructed by taking the projective cone in $\mathbb{P}^{n+1}$ over a degree $n$ rational normal curve in $\mathbb{P}^n$ and take the blowup at the origin. We can simply take $R$ to be the projective cone and the morphism $Y\to R$ to be the blowup map. So, we are contracting the exceptional divisor. 

    When $n=1$, $(F_1)_{\geq0}$ has $4$ boundary divisors, among which the exceptional divisor over $[0:0:1]$ becomes the origin in $R=\mathbb{P}^2$. Residue calculations away from the origin work for $R$ because they work for $F_1$. Residue calculations near the origin work for $R$ because the origin occurs as the boundary of $\{x=0\},\{y=0\}$, which must work because they are positive geometries isomorphic to their strict transforms in the blowup.

    Theorem \ref{pos geom over P1} and Theorem \ref{retract of pos geo over P^1} show that  $\mathbb{P}^1$ being a positive geometry implies that the blowup $\Bl_{[0:0:1] }\mathbb{P}^2$ is a positive geometry, which in turn implies that $\mathbb{P}^2$ is a positive geometry. 
\end{example}

\section{Canonical Form on the Positroid Variety $Z(A|B)$}\label{sec: canonical form of Z(A|B)}
In this section, we explicitly compute the canonical form on the codimension $1$ boundaries of collinear configurations and verify a residue computation.

As stated in Theorem \ref{Grassmannian is a positive geometry}, the totally nonnegative Grassmannian is a positive geometry and the totally nonnegative positroid varieties are all its iterated boundaries. Thus, each positroid variety has a canonical form, which can be computed by taking iterated residues of the canonical form on the Grassmannian. In this section, we apply the BCFW-bridge algorithm from \cite{AHBCGPT16} to compute the canonical form on a special class of positroid varieties.

The real points of a positroid variety contain a dense real torus parametrized by a {\it reduced plabic graph}. The {\it Postnikov's boundary measurement map} sends $\mathbb{R}_{>0}^n$ bijectively to the positive points of the open positroid variety. Muller and Speyer \cite{MS17} constructed a {\it twist map} which extends this to an open embedding of the algebraic torus into the open positroid variety. Under this twist map, each edge weight of the reduced plabic graph is a rational function on the positroid variety, and the canonical form on this variety is the wedge product of all the $d\log$'s of all such edge weights.

In this way, one can obtain a formula for the canonical form on any positroid variety given by a reduced plabic graph. The BCFW-bridge construction is a way to generate a reduced plabic graph for a given positroid variety \cite{AHBCGPT16}.

In this section, we work out an explicit formula for the canonical form on the positroid variety $Z(12\cdots m|m+1,m+2\cdots ,n)$ going through the above machinery.

\begin{proposition}\label{residue along Z(A|B) is good}
    The canonical form on the positroid variety $Z(12\cdots m|m+1,m+2\cdots ,n)$ agrees with the residue $\Res_{[123]=0} \Omega_{V(3,n)}$. 
\end{proposition}
\begin{proof}
Consider the bounded affine permutation $f\in \mathcal B(3,n)$ (see Definition \ref{def: bounded affine perm}) given by
\begin{align*}
    f(x)= x+2 \ \text{ for } \ 1\leq x\leq m-2 \ \text{ and } \  m+1\leq x\leq n-2\\
\quad f(m-1)= m+2,\quad  f(m)= n+1,\quad f(n-1)=n+2,\quad f(n)=m+n+1.
\end{align*} 
In window notation, $f=[3,\cdots, m,m+2, \overline{1},m+3,\cdots n, \overline{2},\overline{m+1}]$, where $\overline{i}=i+n$ and we see that every number modulo $n$ indeed appears once in the window. The corresponding Grassmann necklace (Definition \ref{def: necklace}) is 
\[
I(f) = (I_1,\ldots, I_n) = 
\]
\[
(\{1,2, m+1\}, \{2,3 ,m+1\},\cdots , \{m-1,m, m+1\}, \{m,m+1, m+2\}, \{m+1,m+2, \overline{1}\},\{m+2,m+3, \overline{1}\},\cdots, 
\]
\[
\cdots,\{n-1,n, \overline{1}\},\{n,\overline{1}, \overline{2}\}).
\]
Thus the corresponding positroid variety is $\Pi_f=Z(1\cdots m|m+1\cdots n)$.

Let $s_j$ be the transposition swapping $j$ and $j+1$. The BCFW-bridge construction \cite[\S 5.2]{AHBCGPT16} factorizes the given affine bounded permutation as follows
\[
f=g\circ s_{m+1}\circ \cdots \circ s_{n-1}\circ s_{2}\circ \cdots s_{n-2} \cdot \circ s_{1}\circ \cdots \circ s_{m-1},
\]

where $g$ is the affine bounded permutation given by 
\[
g(x)= x+n = \overline{x}, \text{ for }\ x=1, 2, m+1, \mbox{ and  } g(x)=x \ \text{ for }\  3\leq x\leq m \ \text{ and }\ m+2\leq x\leq n.
\]

Following \cite[\S 5.2]{AHBCGPT16}, the decomposition gives a recipe for constructing a reduced plabic graph for this positroid, where each such transposition corresponds to adding a weighted edge, which results in a move on the matrix. Take edge weight variables $$p_{m+1},\cdots p_{n-1}, q_{2},\cdots q_{n-2},r_1,\cdots r_{m-1},$$ one for each transposition. After composing all the moves, we reach a parametrization of a torus in the open positroid variety $\Pi_f^\circ$, which we describe as follows. 

Fix an interval $[a,b]$. For ease of notation, let $p_{[a,b]}:=\prod_{i\in [a,b]} p_i$ and similarly for $q_{[a,b]}$ and $r_{[a,b]}$. If $a>b$, $[a,b]$ is understood as an empty interval and $p_{[a,b]}=1$ is the empty product. Also, let $(p,q)_{[a,b]}:=\sum_{i\in [a-1,b]} p_{[a,i]} q_{[i+1,b]}$ and similarly for $(q,r)_{[a,b]}$. For example, we have $q_{[2,5]}=q_2q_3q_4q_5$, and $(p,q)_{[6,7]}=p_6p_7+p_6q_7+q_6q_7$. 

The BCFW-bridge construction gives the following parametrization for $\Pi_f^\circ$ 
$$P=\begin{pmatrix}
    1& r_{[1,1]}& r_{[1,2]} &\cdots & r_{[1,m-1]}& 0&\cdots&\cdots&\cdots &0\\
    0&1 & (q,r)_{[2,2]}&\cdots&(q,r)_{[2,m-1]}&q_{[2,m]}&q_{[2,m+1]}&\cdots &q_{[2,n-2]} &0\\
    0&\cdots  &\cdots  &\cdots &0&1&(p,q)_{[m+1,m+1]}&\cdots &(p,q)_{[m+1,n-2]}&p_{[m+1,n-1]}
\end{pmatrix},$$
where on rows $1,2,3$, the $1$ sits on columns $1,2,m+1$, respectively. Notice that in this special case, one can solve for the edge weight variables as rational functions in the Pl\"ucker coordinates, as predicted by Muller-Speyer's twist map. 

Row operations on the matrix does not change points in the Grassmannian. We multiply row $3$ by $\frac{1}{p_{[m+1,n-1]}}$, add ($-r_{[1,1]}=-r_1$ times row $2$) to row $1$, and switch rows $1$ and $3$. We obtain the matrix representative
\[
P'=\begin{pmatrix}
    0&0& \alpha_1& \cdots &\alpha_{n-3} &1\\
    0&1& \beta_1& \cdots &\beta_{n-3} &0\\
    1&0& \gamma_1& \cdots &\gamma_{n-3} &0\\    
\end{pmatrix},
\]
where 
\begin{itemize}
    \item $\alpha_i=0$, for $1\leq i \leq m-2$. $\alpha_{m-1}=\frac{1}{p_{[m+1,n-1]}}$. $\alpha_i=\frac{(p,q)_{[m+1,i+1]}}{p_{[m+1,n-1]}}$ for $m\leq i\leq n-3$.

    \item $\beta_i=(q,r)_{[2,i+1]}$ for $1\leq i\leq m-2$. $\beta_i=q_{[2,i+1]}$ for $m-1\leq i\leq n-3$.
    
    \item $\gamma_i=r_{[1,i+1]}-r_1 (q,r)_{[2,i+1]}$ for $1\leq i\leq m-2$. $\gamma_i= -r_1q_{[2,i+1]}$ for $m-1\leq i\leq n-3$.
\end{itemize}
It is well-known \cite{AHBCGPT16} that the canonical form $\Omega_{\Pi_f}$ on $(\Pi_f,\Pi_{f,\geq0})$ equals the wedge of $d\log$ of all the edge weight variables $p,q,r$'s. 

Proposition \ref{the better form of the form} gives a formula for $\Omega_{V(3,n)}$. By cyclic symmetry, the same formula holds when the matrix variables $\alpha,\beta,\gamma$'s are set up as in the matrix $P'$. Now, by the ideal description of $Z(1\cdots m|m+1\cdots n)$ [cite positroid paper], the Pl\"ucker coordinate $[123]$ pulled back to $V(3,n)$ vanishes in degree $1$ along $Z(1\cdots m|m+1\cdots n)$. Thus, the residue $\Res_{[123]=0} \Omega_{V(3,n)}$ along $Z(1\cdots m|m+1\cdots n)$ is given by removing $d\log [123]=d\log \alpha_1$ from the expression of $\Omega_{V(3,n)}$. This gives $\Res_{[123]=0} \Omega_{V(3,n)}$ in matrix variables $\alpha,\beta,\gamma$'s.

Finally, to establish the proposition, we evaluate the matrix variables $\alpha,\beta,\gamma$ as edge weight variables $p,q,r$'s according to the evaluations provided above. A nontrivial computation verifies that indeed $\Res_{[123]=0} \Omega_{V(3,n)}=\Omega_{Z(1\cdots m|m+1\cdots n)}$, which we now explain.

We start from the expression 
\begin{equation*}
    \begin{aligned}
        \Res_{[123]=0} \Omega_{V(3,n)}=&\frac{ \gamma_{m-1}}{\alpha_{m-1} \beta_{n-3}\gamma_1\gamma_{n-3}} \frac{1}{\prod_{i=1}^{m-2}(\beta_i\gamma_{i+1}-\gamma_i\beta_{i+1}) \prod_{i=m-1}^{n-4} (\alpha_i\beta_{i+1}-\beta_i\alpha_{i+1})}\cdot \\
        &\left( \bigwedge_{i=1}^{m-2} d\beta_i \wedge d\gamma_i \right) \wedge \left( \bigwedge_{i=m-1}^{n-3} d\alpha_i \wedge d\beta_i \right) \wedge d\gamma_{n-3}.
    \end{aligned}
\end{equation*}
By algebraic manipulations, we have
 $$\begin{aligned}
     \prod_{i=1}^{m-3}(\beta_i\gamma_{i+1}-\gamma_i\beta_{i+1})&=\prod_{i=1}^{m-3} ((q,r)_{[2,i+1]}(r_{[1,i+2]}-r_1 (q,r)_{[2,i+2]})- (q,r)_{[2,i+2]}(r_{[1,i+1]}-r_1 (q,r)_{[2,i+1]})) \\
     &=\prod_{i=1}^{m-3}-r_{[1,i+1]}q_{[2,i+2]},\\
    \text{and} \quad \beta_{m-2}\gamma_{m-1}-\gamma_{m-2}\beta_{m-1}&=
(q,r)_{[2,m-1]}( -r_1q_{[2,m]})-q_{[2,m]}(r_{[1,m-1]}-r_1 (q,r)_{[2,m-1]})\\
&=-r_{[1,m-1]}q_{[2,m]}.
 \end{aligned}$$
Combining them, we conclude that 
$$\prod_{i=1}^{m-2}(\beta_i\gamma_{i+1}-\gamma_i\beta_{i+1})=\prod_{i=1}^{m-2}-r_{[1,i+1]}q_{[2,i+2]}.$$
On the other hand, 
$$ \begin{aligned}
    \prod_{i=m}^{n-4} (\alpha_i\beta_{i+1}-\beta_i\alpha_{i+1})
    &= \prod_{i=m}^{n-4}\left(\frac{(p,q)_{[m+1,i+1]}q_{[2,i+2]}}{p_{[m+1,n-1]}} - \frac{(p,q)_{[m+1,i+2]}q_{[2,i+1]}}{p_{[m+1,n-1]}}      \right)\\
&=\prod_{i=m}^{n-4} -\frac{    p_{[m+1,i+2]}   q_{[2,i+1]}}{p_{[m+1,n-1]}} ,\\
\text{and}\quad \alpha_{m-1}\beta_{m}-\beta_{m-1}\alpha_{m}&=
\frac{q_{[2,m+1]}}{p_{[m+1,n-1]}}-\frac{(p,q)_{[m+1,m+1]}q_{[2,m]}}{p_{[m+1,n-1]}}
= \frac{-p_{m+1}q_{[2,m]}}{p_{[m+1,n-1]}}
\end{aligned} $$
Combining them, we conclude that 
$$\prod_{i=m-1}^{n-4} (\alpha_i\beta_{i+1}-\beta_i\alpha_{i+1})= \prod_{i=m-1}^{n-4} -\frac{    p_{[m+1,i+2]}   q_{[2,i+1]}}{p_{[m+1,n-1]}}.$$

Thus, the rational function in the expression for $\Res_{[123]=0} \Omega_{V(3,n)}$ becomes 
$$ \begin{aligned}
    &\frac{ -r_1q_{[2,m]}}{\frac{1}{p_{[m+1,n-1]}} q_{[2,n-2]}(-r_1q_2)(-r_1 q_{[2,n-2]})} \frac{1}{\prod_{i=1}^{m-2}(-r_{[1,i+1]}q_{[2,i+2]}) \prod_{i=m-1}^{n-4} (-\frac{    p_{[m+1,i+2]}   q_{[2,i+1]}}{p_{[m+1,n-1]}})} \\
    &=(-1)^{n-5} \frac{\prod_{i=m+1}^{n-1} p_{[i, n-1]} }{q_{[2,n-2]}\prod_{i=1}^{m-1} r_{[1,i]} \prod_{i=2}^{n-2} q_{[2,i]}},
\end{aligned}$$
which is a Laurent polynomial. Now, we compute the form.

By the usual Leibniz rule, $dx\wedge d(xy)=dx\wedge (ydx+xdy)=dx\wedge(xdy)$. This means that whenever we are wedging with $dx$, we can replace $d(xy)$ with $xdy$. Similarly, whenever we are wedging with $dx$, we may replace $dy$ by $xd(y/x)$. For clarity, we call such substitutions {\it extended Leibniz rules}.

First, since $\gamma_{n-3}=-r_1 q_{[2,n-2]}=-r_1\beta_{n-3}$. We have $d\beta_{n-3}\wedge d\gamma_{n-3}=d\beta_{n-3}\wedge (\beta_{n-3}dr_1)$ by the extended Leibniz rules. Thus, we can first replace $d\gamma_{n-3}$ by $\beta_{n-3}dr_1$. The form thus becomes $\beta_{n-3}\left( \bigwedge_{i=1}^{m-2} d\beta_i \wedge d\gamma_i \right) \wedge \left( \bigwedge_{i=m-1}^{n-3} d\alpha_i \wedge d\beta_i \right) \wedge dr_1$.

By the extended Leibniz rules, $d\beta_i\wedge d(r_1\beta_i)=\beta_id\beta_i\wedge dr_1$, which would have zero contribution after wedging with $r_1$. This means that we may replace $d\beta_i\wedge d\gamma_i$ by $d\beta_i\wedge d(\gamma_i +r_1\beta_i)$. Also, for any $m\leq i\leq n-3$, we can replace $d\alpha_i$ by $\alpha_{m-1}d(\frac{\alpha_i}{\alpha_{m-1}})$. Also, $d\alpha_{m-1}=-\alpha_{m-1}^2 d(1/\alpha_{m-1})$

After such a substitution, we have
$$\begin{aligned}
    \bigwedge_{i=1}^{m-2} d\beta_i \wedge d\gamma_i &\mapsto\bigwedge_{i=1}^{m-2} d\beta_i \wedge d(\gamma_i+r_1\beta_i)
    =\bigwedge_{i=1}^{m-2} d((q,r)_{[2,i+1]}) \wedge d(r_{[1,i+1]})
    \\
     \bigwedge_{i=m-1}^{n-3} d\alpha_i \wedge d\beta_i 
     &\mapsto -\alpha_{m-1}^{n-m}d(1/\alpha_{m-1})\wedge d\beta_{m-1}\wedge \bigwedge_{i=m}^{n-3} d(\alpha_i/\alpha_{m-1}) \wedge d\beta_i\\
     &=-\alpha_{m-1}^{n-m}d(p_{[m+1,n-1]} )\wedge d(q_{[2,m]}) \wedge   \bigwedge_{i=m}^{n-3} d((p,q)_{[m+1,i+1]}))\wedge d(q_{[2,i+1]})\\
\end{aligned}$$

Now, notice that in the form we have a wedge product of $dr_{1}\wedge dr_{[1,2]}\wedge \cdots \wedge dr_{[1,m-1]}$. By the extended Leibniz rules, we can substitute the last term $dr_{[1,m-1]}$ by $r_{[1,m-2]}dr_{m-1}$, and then replace the second to last term $dr_{[1,m-2]}$ by $r_{[1,m-3]}dr_{m-2}$, and continue in a similar fashion. So 
\[
dr_{1}\wedge dr_{[1,2]}\wedge \cdots \wedge dr_{[1,m-1]}=r_1^{m-2}r_2^{m-3}\cdots r_{m-2}dr_1\wedge dr_2\wedge \cdots \wedge d{r_{m-1}}.
\]

In the presence of $dr_i$ for all $i\in [m-1]$, we may apply the extended Leibniz rules to substitute $d((q,r)_{[2,2]})=d(q_2+r_2)\mapsto dq_2$. Then, in the presence of $dq_2$ and all $dr_i$, we may replace $d((q,r)_{[2,3]})=d(q_2q_3+q_2r_3+r_2r_3)\mapsto d(q_2q_3)$. The process continues and we see that in the presence of $dr_i$ for $i\in [1,m-1]$, every $d((q,r)_{[2,i+1]})$ can be replaced by its first term $d(q_{[2,i+1]})$.

Now, we have $dq_{2}\wedge dq_{[2,3]} \wedge \cdots \wedge dq_{[2,m-1]}$ coming from $d\beta_i\wedge d\gamma_i$. These can be combined with the $dq$-intervals from $d\alpha_i\wedge d\beta_i$ to form a sequence 
\[
dq_{2}\wedge dq_{[2,3]} \wedge \cdots \wedge dq_{[2,n-2]}= q_2^{n-4}q_3^{n-5}\cdots q_{n-3} dq_{2}\wedge dq_{3} \wedge \cdots \wedge dq_{n-2},
\]
by the same reasoning above. Again, with a similar argument, we can replace every $d((p,q)_{[a,b]}))\mapsto dp_{[a,b]}$ in the presence of all the $dq$'s. Similarly, the wedge product of all the $dp_{[a,b]}$'s can be substituted by a monomial in the $p$'s multiplied by the wedge product of all the $dp$'s.

In conclusion, we see that the form is equal to a monomial times the wedge product of the differential of all the $p,q,r$-variables, times $1/\alpha_{m-1}^{n-m}$; the latter is a monomial in the $p$-variables. A final computation reveals that we are left with the square-free product of all the variables in the denominator. This finishes the proof. 
\end{proof}

\begin{example}
    The computation is best illustrated through the following example. Let $m=4,n=7$ and we compute the form on $Z(1234|567)$. Since $n=7$, we have modulo $7$ cyclic symmetry and we denote $n+7$ by $\overline{n}$. The subscheme $Z(1234|567)$ of $\Gr(3,7)$ is given by the positroid whose Grassmann necklace is $(125,235,345,456,56\overline{1},67\overline{1},7\overline{12})$. The corresponding bounded affine permutation is thus $f=[346\overline{1}7\overline{25}]$ and we have $Z(1234|567)=\Pi_f$.

    The BCFW-recursion gives coordinates on a general closed positroid variety $\Pi_f$ that gives a formula for the canonical form on $\Pi_f$. We follow [cite].

The BCFW-bridge construction breaks up the permutation as the following product 
$$f=g\cdot s_{5}s_{6}\cdot s_{2}s_3s_4 s_{5}\cdot s_{1}s_2 s_{3}$$
where $g$ is the affine bounded permutation given by 
$$g(x)=\overline{x}\ \text{ for }\ x=1,2,5, \quad g(x)=x \ \text{ for }\  x=3,4,6,7$$
We have a matrix parametrization of the open positroid variety $\Pi_f^\circ$ given by 
$$\begin{pmatrix}
    1 & r_1 & r_1r_2 & r_1r_2r_3 & 0  & 0 & 0 \\
    0 & 1 & q_2+r_2 & q_{23}+q_2r_3+r_{23} & q_2q_3q_4 & q_2q_3q_4q_5  & 0\\
    0 & 0 & 0 & 0 & 1 & p_5+q_5 & p_5p_6
\end{pmatrix}$$
and the canonical form is given by 
$(\bigwedge_{i=1}^3 d\log r_i)\wedge (\bigwedge_{j=2}^5 d\log q_j)\wedge (\bigwedge_{k=5}^6 d\log p_k)$.

%On the other hand, as part of the residue computation, we need to compute the wedge product of the following forms {\color{red}DS: To be done!}
    
\end{example}

\section{$V(3,n)$ is a Positive Geometry}\label{sec:final proof}

In this section, we prove our main Theorem \ref{thm:mainintro}. 
\begin{lemma}\label{order of vanishing}
    The Pl\"ucker coordinate $p_{123}$ (a global section of $\mathcal{O}_{V(3,n)}(1)$) vanishes at order one along divisors $Z(1\cdots m|m+1\cdots n)$ for $ 3\leq m\leq n-2$ and $ZV(23x)_{\red}$.
\end{lemma}
\begin{proof}
It is immediate that $p_{123}$ vanishes on the given divisors on $V(3,n)$. The order of vanishing is a local computation, and thus we can restrict to the chart $p_{12n}=1$. On the latter, we fix coordinates
\begin{equation}\label{local coords}
\begin{pmatrix}
    0&0& \alpha_1& \cdots &\alpha_{n-3} &1\\
    0&1& \beta_1& \cdots &\beta_{n-3} &0\\
    1&0& \gamma_1& \cdots &\gamma_{n-3} &0\\    
\end{pmatrix}
\end{equation}
Note that, in these coordinates, $p_{123}=\alpha_1$. Notice that this is a different choice of chart than in Section \ref{section: candidate form}. However, by cyclic symmetry, all results in Section \ref{section: candidate form} can be directly translated to these coordinates, hence the naming.

\textbf{Case 1:} We show that the order of vanishing along $Z(1\cdots m|m+1\cdots n)$ is one. For simplicity of notation, write $V$ and $W$ for restriction of $V(3,n)$ and $Z(1\cdots m|m+1\cdots n)$ on this chart, respectively. Then the local ring
$\mathcal{O}_{V,W}=(R/I)_J$, where $R=k[\alpha_i,\beta_i,\gamma_i]$, $I$ is the prime ideal of $V$ (described in Theorem \ref{prop:interpretation of V}) and $J$ is the prime ideal in $R/I$ generated by the images of 
\begin{itemize}
    \item $\alpha_i$ for $1\leq i\leq m-2$
    \item $\beta_j\gamma_k-\beta_k\gamma_j$ for $m-1\leq j,k\leq n-3.$
\end{itemize}
The ideal $J$ is prime because $W$ is a positroid variety. Furthermore, $W$ is a divisor in $V$. By the aforementioned description, $I$ is generated by the determinants
\[
\Delta_{ijk}=\det \begin{pmatrix}
    \alpha_i\beta_i  &\alpha_j\beta_j  &\alpha_k\beta_k\\
    \alpha_i\gamma_i&\alpha_j\gamma_j&\alpha_k\gamma_k\\
    \beta_i \gamma_i&\beta_j \gamma_j&\beta_k \gamma_k
\end{pmatrix}.
\]
By Proposition \ref{prop:V(3,n) is normal}, $V$ is normal and so we see that $\mathcal{O}_{V,W}=(R/I)_J$ is a discrete valuation ring with maximal ideal $J(R/I)_J$. The element $\alpha_1$ is in $J$, and we have to show that $\alpha_1\not\in J^2(R/I)_J$. 
We expand $\Delta_{ijk}$ using the first column. The first two terms are a multiple of $\alpha_i$ and the last term is $\beta_i\gamma_i \alpha_j\alpha_k(\beta_j\gamma_k-\beta_k\gamma_j)$. Since $\Delta_{ijk}$ is zero in $R/I$, this implies that $\beta_i\gamma_i \alpha_j\alpha_k(\beta_j\gamma_k-\beta_k\gamma_j)$ is a multiple of $\alpha_i$ in $R/I$.

Let $1\leq i\leq m-2$ and $m-1\leq j,k\leq n-3$. It is easy to see that $\beta_i,\gamma_i, \alpha_j,\alpha_k$ are not in the prime ideal $J$ by explicitly constructing points in $W$ of the form
\[
\begin{pmatrix}
    0&0&0&0&\cdots& 0 & n-m&n-m-1&\cdots &1\\
    0&1&1&1&\cdots &1& 0&0&\cdots &0 \\
    1&0&1&2& \cdots &m-2 & 0&0&\cdots &0  
\end{pmatrix}.
\]
Thus, since $J$ is prime, their product is not in $J$. As $\beta_i\gamma_i \alpha_j\alpha_k$ is not an element in the maximal ideal of the local ring $(R/I)_J$, it is invertible in this ring. Choosing $i=1$ in $\Delta_{ijk}$, we then conclude that $\beta_j\gamma_k-\beta_k\gamma_j$ is a multiple of $\alpha_1$ in $(R/I)_J$.

Now, by symmetry, if $\alpha_1\in J^2(R/I)_J$, then $\alpha_i\in J^2(R/I)_J$ for all $1\leq i\leq m-2$. The above argument shows that this would imply that $\beta_j\gamma_k-\beta_k\gamma_j\in J^2(R/I)_J$ for all $m-1\leq j,k\leq n-3$. Since they are the generators of $J$, this would imply that $J\subset J^2$ and so $J=J^2$. 

Let $U$ be the dense open subscheme of $W$ of smooth points in both $V$ and $W$.  Let $\mathcal J/\mathcal J^2$ be the conormal sheaf of $U\subset V$. By the conormal sheaf short exact sequence \cite[Theorem 8.17]{Hartshorne}, we see that $\mathcal J/\mathcal J^2$ is locally free of rank one on $U$. However, the discussion above implies that $\mathcal J/\mathcal J^2$ should be zero everywhere, thus providing a contradiction.

\textbf{Case 2:}  We show that the order of vanishing along $ZV(23x)_{\red}$ is one. For simplicity of notation, write $V$ and $D$ for restriction of $V(3,n)$ and $ZV(23x)_{\red}$ on this chart, respectively. Then the local ring
$\mathcal{O}_{V,D}=(R/I)_J$, where $R=k[\alpha_i,\beta_i,\gamma_i]$, $I$ is the prime ideal generated by the determinants $\Delta_{ijk}$, and $J$ is the prime ideal in $R/I$ defining
$D$ in $V$. From Theorem \ref{thm:matrix description of colliding boundary}, we see that $ZV(23x)_{\red}=ZV(23x)$ on our local chart. Therefore $J$ 
is generated by the images in $R/I$ of 
\begin{itemize}

\item the $2\times 2$ minors of the $3\times 2$ submatrix of \eqref{local coords} consisting of the second and third column: $\alpha_1,\gamma_1$. 

\item the determinants $\Delta_{ijk}\in I$. 
\end{itemize}
So $J = (\alpha_1,\gamma_1)$ and the conclusion follows from the involution symmetry exchanging $\alpha_i$ with $\gamma_i$ on the ring $\mathcal{O}_{V,D}=(R/I)_J$. 
\end{proof}

\begin{proposition}\label{pi1pi2 are good}
The following statements hold true.

 \begin{enumerate}
 \item[(i)]  The morphism $
\pi_2 :\mathcal{I}_{st}(A_0;I_1,\cdots, I_r)\subset V(A_0;I_1\cdots, I_r)'\times \mathbb{P}^1 \to \mathbb{P}^1
$ as in Proposition \ref{prop:bundle over P^1} satisfies the assumptions in the first half of Theorem \ref{pos geom over P1}.

\item[(ii)]  The morphism $
\pi_1:\mathcal{I}_{st}(A_0;I_1,\cdots, I_r)\subset V(A_0;I_1\cdots, I_r)'\times \mathbb{P}^1 \to V(A_0;I_1\cdots, I_r)'.
$ as in Proposition \ref{prop:iso of sheaves} satisfies the assumptions in the first half of Theorem \ref{retract of pos geo over P^1}.
 \end{enumerate}  
\end{proposition}

\begin{proof}
We induct on $n$ to show that both of the above statements are true. (Recall that $A_0, I_1,\ldots, I_r$ in the statement define a partition of $[n]$.)

First, assume that the above statements hold for any $n<n_0$, for a fixed $n_0$, statements (i) and (ii) -- combined with the first half of Theorem \ref{pos geom over P1} and Theorem \ref{retract of pos geo over P^1} -- imply the following facts inside $V(3,n)$ for any $n<n_0$:
\begin{enumerate}
    \item $\mathcal{I}_{st}(A_0;I_1,\cdots, I_r)$, along with its natural totally nonnegative part, has boundaries given by the $V(A_0,I_1',\cdots,I_r)'$-fibers over $[0:1]$ and $[1:0]$ under $\pi_2$; for each boundary $X_i$ of $V(A_0,I_1',\cdots,I_r)'$, $\pi_2$ induces an $X_i$-fibration.
    \item $V(A_0,I_1,\cdots,I_r)'$, along with its natural totally nonnegative part, has boundaries given by the image of each boundary of $\mathcal{I}_{st}(A_0;I_1,\cdots, I_r)$ under $\pi_1$ such that the dimension does not drop.
\end{enumerate}

We claim that the above two statements imply that for any $n<n_0$, $V(A_0,I_1,\cdots,I_r)'$, along with its natural totally nonnegative part, has boundaries given by:

\begin{enumerate}
   \item Fixing any $i\in [r]$, we replace $I_i,I_{i+1}$ with $I_i\bigsqcup I_{i+1}$ in $V(A_0,I_1,\cdots,I_r)'$, where the index is taken modulo $r$. Namely, we have $V(A_0,I_2,\cdots ,I_{r-1}, I_r\bigsqcup I_1)'$ when $i=r$.
   \item Fixing a partition of $[r]$ into two disjoint cyclic intervals $A\bigsqcup B$, each of length at least $2$, one has the boundaries of collinear configurations $\Pi(A_0; I_a \text{ for } a\in A| I_b \text{ for } b\in B)$.
   \item For each $\# I_i\geq2$ and $a\in I_i$, $V(A_0\bigsqcup \{a\},I_1,\cdots, I_i-\{a\},\cdots,I_r)'$.
\end{enumerate}

To see this, notice that when each $I_i$ has size $1$, this is the content of Theorem \ref{decomposition of boundary}. When some $I_i$ has size at least $2$, by cyclic symmetry we can assume that $\#I_1\geq2$. Recall that $V(A_0,I_1',\cdots,I_r)'$ by definition lives in $V(3,n-1)$ where two adjacent elements of $I_1$ are identified. Since the boundaries of $V(A_0,I_1',\cdots,I_r)'$ are those above by induction, then the boundaries of $\mathcal{I}_{st}(A_0,I_1,\cdots,I_r)$ lie in the following two classes: 

\begin{enumerate}
    \item Fixing a boundary $X_i$ of $V(A_0,I_1',\cdots,I_r)'$ means fixing either one of the three conditions above: merging two sets of columns, dividing all columns into two disjoint subsets and require they correspond to collinear points, or deleting an extra column. Note that we are doing this to $V(A_0,I_1',\cdots,I_r)'$, namely we have $n-1$ indices and ``$st$" is treated as a single index. Taking the sub-fibrations in $\mathcal{I}_{st}$ simply amounts to duplicate the column formally labeled by ``$st$" into two parallel columns now labeled by $s$ and $t$. Now, taking the image in $V(A_0,I_1,\cdots,I_r)$ under $\pi_1$ means that we forget about the point in $\mathbb{P}^1$, keeping track of the ratio between columns $s$ and $t$. We have the same original condition, except that now $s$ and $t$ are treated as two different columns in the same set $I_1$. 
    \item The fibers over $[1:0]$ and $[0:1]$ are the subschemes of $\mathcal{I}_{st}$ where column $t$ or $s$ is zero, respectively. Taking the image, we have $V(A_0\bigsqcup \{s\},I_1-\{s\},\cdots,I_r)'$ or $V(A_0\bigsqcup \{t\},I_1-\{t\},\cdots,I_r)'$.
\end{enumerate}
Hence, induction on $n<n_0$ shows that the boundaries of $V(A_0;I_1,\cdots, I_r)'$ are either of the form (1), (2) or (3). 

So far, we established that if the statements (i) and (ii) hold for $n<n_0$, then we have a description of boundaries of $V(A_0;I_1,\cdots, I_r)'\subset V(3,n)$ for $n<n_0$. Below, in the inductive proof of (i) and (ii), we shall use this consequence.

We now provide the inductive proof of (i) and (ii). 

\noindent (i). Consider the morphism $
\pi_2 :\mathcal{I}_{st}(A_0;I_1,\cdots, I_r)\subset V(A_0;I_1\cdots, I_r)'\times \mathbb{P}^1 \to \mathbb{P}^1
$. We verify the assumptions in the first half of Theorem \ref{pos geom over P1} with $X=V(A_0;I_1'\cdots, I_r)'$ and $Y=\mathcal{I}_{st}(A_0;I_1,\cdots, I_r)$.
\begin{enumerate}
    \item As shown in Proposition \ref{prop:bundle over P^1}, $\pi_2$ is trivialized on two charts $U_1$ and $U_2$.
    \item As in the proof of Proposition \ref{prop:iso of sheaves}, the transition morphism from chart $U_1$ to $U_2$ is given by multiplying column $s$ by $y/x$ to get column $t$. Since $y/x\geq0$ on the totally nonnegative part, the transition morphism preserves the totally nonnegative part.
    \item We need to check this condition for each boundary $X_i$ of the fiber $V(A_0;I_1'\cdots, I_r)'$. By inductive hypothesis, we have a description of boundaries of $X_i$ by merging two sets of columns, dividing all columns into two disjoint subsets and require they correspond to collinear points, or deleting an extra column. Notice that all three possibilities are invariant under the transition map, because the transition map is simply rescaling a column by scalar $y/x$. On the intersection of the two charts, the scalar $y/x$ is invertible. And thus the transition morphism restricts to an isomorphism between $X_i$-fibers.
\end{enumerate}

\noindent (ii). Consider the morphism $
\pi_1:\mathcal{I}_{st}(A_0;I_1,\cdots, I_r)\subset V(A_0;I_1\cdots, I_r)'\times \mathbb{P}^1 \to V(A_0;I_1\cdots, I_r)'.$ We verify the assumptions in the first half of Theorem \ref{retract of pos geo over P^1}, with $X=V(A_0;I_1'\cdots, I_r)'$, $Y=\mathcal{I}_{st}(A_0;I_1,\cdots, I_r)$,  $R=V(A_0;I_1\cdots, I_r)'$. We think of $S$ as a variety abstractly isomorphic to $V(A_0\bigcup \{s,t\}, I_1-\{s,t\},\cdots, I_r)'$. We let $S$ embed into $X$ as its closed subscheme given by setting columns $s$ (and $t$) to zero. Notice that $S$ is indeed trivial in $Y$ in the sense of Definition \ref{dfn:trivial subvariety}, because the transition map is column scaling but these two columns are already zero. To define a retract as in Definition \ref{retract}, we need the following extra data of two morphisms $\iota_{1,2}:X\hookrightarrow R$. We have two embeddings $X\hookrightarrow Y$ as fibers over $[1:0]$ and $[0:1]$ of $\pi_2$. They are given by requiring column $t$ or $s$ to be zero. $\iota_{1,2}$ is defined by composing them with $\pi_1:Y\to R$. We check the definition of retract:
\begin{enumerate}
    \item The morphism $Y\to R\times \mathbb{P}^1$ being a closed embedding is true by construction of $Y$ as the incidence subscheme $\mathcal{I}_{st}$.
    \item The isomorphism outside $S\times \mathbb{P}^1\subset Y$ (where both columns $s$ and $t$ are zero) is clear. Columns $s$ and $t$ not simultaneously zero implies that we can recover the unique point in $\mathbb{P}^1$ giving the ratio of these two columns.
    \item The commutativity follows essentially by definition of the maps along with the fact that $V(A_0\bigcup \{s,t\}, I_1-\{s,t\},\cdots, I_r)'$ is isomorphic to the closed subscheme of $\mathcal{I}_{st}(A_0;I_1,\cdots, I_r)$ given by setting columns $s$ and $t$ both to zero.

    \item True by choice of $\iota_{1,2}$.
\end{enumerate}
This finishes the proof. 
\end{proof}

In the proof above, we saw a classification of the boundaries of $V(A,I_1,\cdots,I_r)'$. By iteratedly taking boundaries of boundaries and employing Theorem \ref{decomposition of boundary} as base case, we achieve the classification of iterated boundaries of $V(3,n)$: 

\begin{corollary} \label{iterated boundaries of V(3,n)}
The iterated algebraic boundaries of the ABCT variety $V(3,n)$ are:
\begin{enumerate}
    \item[(i)] the boundaries of collinear configurations $\Pi(A_0; I_a \text{ for } a\in A| I_b \text{ for } b\in B)$ and their positroid subvarieties.
    \item[(ii)] boundaries of colliding configurations $V(A_0;I_1,\cdots, I_r)_{\red}=V(A_0;I_1,\cdots, I_r)'$.
\end{enumerate}    
The iterated analytic boundaries are given by the totally nonnegative points of the algebraic boundaries. The iterated boundaries form a poset under containment.
\end{corollary}

\begin{remark}\label{iterated boundaries of colliding boundaries}
In the statement of  Corollary \ref{iterated boundaries of V(3,n)} it is hiding the seemingly more general description of iterated boundaries of $V(A_0;I_1,\cdots, I_r)_{\red}\subset V(3,n)$. They are exactly those iterated boundaries of $V(3,n)$ contained in $V(A_0;I_1,\cdots, I_r)_{\red}$. These are given by:
\begin{enumerate}
    \item the boundaries of collinear configurations $\Pi(A_1; J_a \text{ for } a\in A| J_b \text{ for } b\in B)$ and their positroid subvarieties, where $A_1$ contains $A_0$ and $J_a, J_b$ are disjoint unions of $I_i-A$.
    \item boundaries of colliding configurations $V(A_0;J_1,\cdots, J_s)_{\red}$, where $A_1$ contains $A_0$ and $J_j$ are disjoint unions of $I_i-A$.
\end{enumerate}
    
\end{remark}

\begin{theorem}[Main Theorem]
Let $\theta_2$ denote the rational Veronese map $\Gr(2,n)\dashrightarrow V(3,n)$. The following hold.

\begin{enumerate}
\item[(i)] The triple $(V(3,n), V(3,n)_{\geq0}, \Omega_{V(3,n)})$ is a positive geometry, where $V(3,n)$ is the ABCT variety, $V(3,n)_{\geq0}$ is the intersection $V(3,n)\bigcap \Gr_{\geq0}(3,n)$ or equivalently the analytic closure of the image ${\theta_2(\Gr(2,n)_{>0})}$, and finally the canonical form $\Omega_{V(3,n)}$ is the pushforward of $\Omega_{\Gr(2,n)}$ along $\theta_2$. \\

\item[(ii)] The iterated algebraic boundaries of the ABCT variety $V(3,n)$ are certain positroid varieties in $\Gr(3,n)$ and certain other positroid varieties in $\Gr(3,n)$ intersected with $V(3,n)$ with the reduced subscheme structure. A detailed description is given in Definition \ref{dfn:iterated boundaries two types}. The iterated analytic boundaries are the corresponding iterated algebraic boundaries intersected with $\Gr_{\geq0}(3,n)$.\\

\item[(iii)] The poset structure of iterated boundaries is the induced sub-poset on the two classes of positroids as above. The poset is naturally graded by the dimension of the iterated boundary. \\

\item[(iv)] The variety $V(3,n)$ and all its iterated algebraic boundaries are reduced, irreducible, rational, normal, Cohen-Macaulay, regular in codimension $1$, and have expected dimensions from the perspective of point configurations.\\

\item[(v)]  For any iterated boundary $V(A_0;I_1\cdots, I_r)_{\red}$ in $V(3,n)$ such that $s\not=t\in I_1$, the morphism $
\pi_2 :\mathcal{I}_{st}(A_0;I_1,\cdots, I_r)\subset V(A_0;I_1\cdots, I_r)_{\red}\times \mathbb{P}^1 \to \mathbb{P}^1
$ as in Proposition \ref{prop:bundle over P^1} satisfies the assumptions of Theorem \ref{pos geom over P1}.\\

\item[(vi)] For any iterated boundary $V(A_0;I_1\cdots, I_r)_{\red}$ in $V(3,n)$ such that $s\not=t\in I_1$, the morphism $
\pi_1:\mathcal{I}_{st}(A_0;I_1,\cdots, I_r)\subset V(A_0;I_1\cdots, I_r)_{\red}\times \mathbb{P}^1 \to V(A_0;I_1\cdots, I_r)_{\red}
$ as in Proposition \ref{prop:iso of sheaves} satisfies the assumptions of Theorem \ref{retract of pos geo over P^1}.\\

\end{enumerate}
\end{theorem}
\begin{proof}
(ii) is Corollary \ref{iterated boundaries of V(3,n)} and (iv) is Theorem \ref{geometry of V and its boundaries}. To show (iii), we need to show that the poset of iterated boundaries is the induced sub-poset of positroids. This is clear, because if the positroids appearing in Corollary \ref{iterated boundaries of V(3,n)} that we are intersecting with $V(3,n)$ are exactly those such that if several columns give points on a line, then all columns give points on a union of two lines. So intersecting with $V(3,n)$ does not change containment conditions. Namely, for $\Pi_1,\Pi_2$ positroids of this type, $\Pi_1\subset\Pi_2$ if and only if $\Pi_1\cap V(3,n)\subset \Pi_2\cap V(3,n)$. 

Now, we show that (i)+(v)+(vi) hold by induction on $n$. When $n=5,6$, in \S \ref{n=5,6 case} we establish that $V(3,5)$ and $V(3,6)$ are indeed positive geometries with the canonical form given by pushforward. Furthermore, their iterated boundaries are all positroid varieties and thus (v) and (vi) hold. 

Fix $n\geq7$. We assume that (i), (v), (vi) holds inside $V(3,m)$ for all $m<n$. We first show that this implies (v) and (vi) holds inside $V(3,n)$.

Notice that (v) and (vi) combined is a stronger version of Proposition \ref{pi1pi2 are good}. Recall from the proof of Proposition \ref{pi1pi2 are good} that the morphism $\pi_2$ in (v) realizes $\mathcal{I}_{st}(A_0;I_1,\cdots, I_r)\subset V(3,n)$ as a fibration over $\mathbb{P}^1$ with fibers isomorphic to $V(A_0;I_1',\cdots , I_r)'\subset V(3,n-1)$. The morphism $\pi_1$ in (vi) realizes $V(A_0;I_1,\cdots, I_r)'$ as a retract of $\mathcal{I}_{st}(A_0;I_1,\cdots, I_r)\subset V(3,n)$.
Thus, if (v) and (vi) are true, once combined with Theorem \ref{pos geom over P1} and Theorem \ref{retract of pos geo over P^1}, they imply that:
\begin{itemize}
    \item If $V(A_0;I_1',\cdots , I_r)'\subset V(3,n-1)$ is a positive geometry, then $\mathcal{I}_{st}(A_0;I_1,\cdots, I_r)\subset V(3,n)$  is a positive geometry.
    \item If $\mathcal{I}_{st}(A_0;I_1,\cdots, I_r)\subset V(3,n)$ is a positive geometry, then $V(A_0;I_1,\cdots, I_r)'\subset V(3,n)$ is a positive geometry.
\end{itemize}
These last two statements tell us that $V(A_0;I_1,\cdots, I_r)\subset V(3,n)$ is a positive geometry whenever $V(A_0;I_1',\cdots , I_r)'\subset V(3,n-1)$ is a positive geometry. Notice that here the column index set has cardinality $n-1$, and so it is reduced by one.

To show (v), we only need to verify that the morphism $
\pi_2 :\mathcal{I}_{st}(A_0;I_1,\cdots, I_r)\subset V(A_0;I_1\cdots, I_r)_{\red}\times \mathbb{P}^1 \to \mathbb{P}^1
$ as in Proposition \ref{prop:bundle over P^1} satisfies the assumptions of the second half of Theorem \ref{pos geom over P1}. Recall from the proof of Proposition \ref{pi1pi2 are good} that $\pi_2$ realizes $\mathcal{I}_{st}(A_0;I_1,\cdots, I_r)$ as a fibration over $\mathbb{P}^1$ with fibers isomorphic to $V(A_0;I_1',\cdots , I_r)'\subset V(3,n-1)$. 

\begin{itemize}
    \item We can conclude that the fiber $V(A_0;I_1',\cdots , I_r)'\subset V(3,n-1)$ is a positive geometry. This is done by repeatedly using the above discussion to merge two columns in the same $I_i$ into one column, appealing to (v)+(vi) for all $m<n$, combined with Theorem \ref{pos geom over P1} and \ref{retract of pos geo over P^1}. The process terminates until we reach $V(3,m')$ for a smaller $m'$, which is indeed a positive geometry by inductive hypothesis that (i) holds for $m'<n$.
    \item From Remark \ref{iterated boundaries of colliding boundaries}, we know that all iterated boundaries of $V(A_0,I_1',\cdots,I_r)'$ are either positroid varieties or other boundaries of colliding configurations. They are invariant under rescaling a column because they are either a positroid variety or a positroid variety intersected with $V(3,n)$, both of which are invariant under column rescaling.
    \item We need to show that the canonical form glues. That is, the transition morphism preserves canonical form. Note that the transition morphism sends $(X,t)\to (X',t^{-1})$, where $X\in V(A_0,I_1',\cdots,I_r)'\subset V(3,n-1)$ and $X'$ is obtained by rescaling a column in $X$ by $t$. When $I_1',I_2,\cdots ,I_r$ all have size $1$, the canonical form is invariant under the transition morphism by Proposition \ref{form invariance under column scaling} and the fact that $d\log (t^{-1})=-d\log t$. We show that this is the case in general by induction on $\sum \#(I_i-1)$. Assume without loss of generality that $I_1'$ has size at least $2$. Since (v) and (vi) holds for $n-1$, we again have that the canonical form on $V(A_0,I_1',\cdots,I_r)'$ is equal to $\Omega_{V(A_0,I_1'',\cdots,I_r)'}\wedge d\log h$, where $V(A_0,I_1'',\cdots,I_r)'\subset V(3,n-2)$ is obtained by merging two columns in $I_1'$ and $h$ is the ratio between the two parallel columns in $I_1'$. Now, if $\Omega_{V(A_0,I_1'',\cdots,I_r)'}$ is invariant under rescaling a column by $h$ up to multiple of $dh$, then $\Omega_{V(A_0,I_1'',\cdots,I_r)'}\wedge d\log h   $ is invariant under rescaling a column by $t$ up to multiple of $dt$. This reasoning reduces $\sum \#(I_i-1)$ by $1$ and we are done.
\end{itemize}
The above bullet points verify (v). Statement (vi) follows because assumptions in the first half of Theorem \ref{retract of pos geo over P^1} is verified in Proposition \ref{pi1pi2 are good}, assumptions in the second half is covered in (v), and finally the normality assumption is covered in (iv).

So far, we showed that (i)+(v)+(vi) holds for smaller $n$ implies that (v)+(vi) holds for $n$, which in turns implies that all iterated boundaries of colliding configurations $V(A_0;I_1,\cdots, I_r)'\subset V(3,n)$ are positive geometries. From here, we now show that (i) holds for $n$.

    To show the statement of positive geometry in (i), we first verify the topological conditions described in the first paragraph of Definition \ref{def:iterated boundary}. Namely, we need to show that $V(3,n)_{\geq0}$ is the analytic closure of its analytic interior and the analytic interior is an orientable manifold of dimension equal to the Krull dimension of $R$. Notice that by Theorem \ref{conj:equality of boundaries}, the analytic interior of $V(3,n)_{\geq0}$ is exactly the intersection $V(3,n)(\mathbb{R})\bigcap \Gr(3,n)_{>0}$, where the latter is the subset the the real Grassmannian where all the Pl\"ucker coordinates are strictly positive. Recall from Subsection \ref{subsection:geometry of V} that we used $\itr V$ to denote the open subscheme of $V(3,n)$ defined by the nonvanishing of all Pl\"ucker coordinates. Then, the analytic interior of $V(3,n)_{\geq0}$ is exactly the intersection $\itr V(\mathbb{R})\bigcap \Gr(3,n)_{>0}$. By Lemma \ref{int V is smooth}, $\itr V$ is a smooth variety, and thus its real points $\itr V(\mathbb{R})$ is a manifold. Then, the analytic interior $\itr V(\mathbb{R})\bigcap \Gr(3,n)_{>0}$ being an open subset contained in it is also a manifold. The dimension count of $\itr V(\mathbb{R})\bigcap \Gr(3,n)_{>0}$ follows from the fact that the rational map $\theta:\Gr(2,n)\dashrightarrow V(3,n)$ is defined on $\Gr(2,n)_{>0}$, finite-to-one, and has image in $V(3,n)_{>0}$, as in Remark \ref{D_n symmetry} and the discussion right above Proposition \ref{pushforward formula}. To show that it is orientable, notice that the canonical form $\Omega_{\Gr(2,n)}$ has no zeros or poles in $\Gr(2,n)_{>0}$ and is defined over $\mathbb{R}$. Also, $\Gr(2,n)_{>0}$ is exactly the preimage of $V(3,n)_{>0}$ under $\theta$. Thus, the pushforward of $\Gr(2,n)_{>0}$ via $\theta$ also has no zeros or poles on $V(3,n)_{>0}$ and is defined over $\mathbb{R}$. One can see this explicitly using the formula from Proposition \ref{the better form of the form} that on the chart $p_{123}\not=0$, all poles and zeros are along the vanishing of some Pl\"ucker coordinate. By cyclic symmetry, the same statement holds on any chart $p_{i,i+1,i+2}\not=0$. And the statement follows from the fact that $G_{\geq0}(3,n)\setminus G_{>0}(3,n)$ equals the union of positroid divisors $\bigcup Z(p_{i,i+1,i+2})$, as discussed in the beginning of Section \ref{sec:boundary components}. Thus, our candidate canonical form is a nowhere vanishing holomorphic form on $V(3,n)_{>0}$, showing orientability. Finally, $V(3,n)_{\geq0}$ is the analytic closure of $V(3,n)_{>0}$ by Proposition \ref{analytic closure}.

We already identified all iterated boundaries of $V(3,n)$ in (ii) and showed that they are normal in (iv). Thus, to check (i), it remains to prove that our candidate canonical form $\theta_*(\Omega_{\Gr(2,n)})$ verifies the recursive property in Definition \ref{def: positive geom}. Note that, by Remark \ref{uniqueness of canonical form}, the uniqueness of the canonical form is guaranteed by the fact that all iterated boundaries are rational and normal, which holds by (iv).

By Theorem \ref{decomposition of boundary} and Theorem \ref{conj:equality of boundaries}, the boundary components of $V(3,n)$ are exactly $Z(A|B)$ and $V(a,a+1,l:l\in [n])'=V(a,a+1,l)'$. The recursive property involves verifying three steps.

\begin{enumerate}
    \item Each boundary divisor described in Theorem \ref{decomposition of boundary} is a positive geometry.
    \item The candidate form $\Omega_{V(3,n)}$ on $V(3,n)$ has no other poles outside of the boundary divisors. 
    \item The candidate form $\Omega_{V(3,n)}$ has simple poles along each boundary divisor described above and the residue equals the canonical form on the corresponding boundary divisor.
\end{enumerate}

\noindent \textbf{Step 1:} The variety $Z(A|B)$ is a positroid variety, which is known to be a positive geometry \cite[\S 7]{AHBCGPT16}. The boundary of colliding configurations $ZV(34l)'$ is already shown above to be a positive geometry as a byproduct of showing (v)+(vi).\\

\noindent \textbf{Step 2:} The first equality in Proposition \ref{inductive formula of canonical form on V(3,n)} gives a formula for $\Omega_{V(3,n)}$ in coordinates on the chart $\{p_{123}\not=0\}$. So it does not see poles along $ZV(123)$.

On the complement of $ZV(123)$, the formula implies that any pole of $\Omega_{V(3,n)}$ is contained in the union of $ZV(\beta_1\gamma_2-\beta_2\gamma_1)= ZV(145)$, $ZV(\beta_1)=ZV(124)$, $ZV(\gamma_1)=ZV(134)$, or it is a pole of $\varphi^*_4\Omega_{V(3,n-1)}$. The poles of $\varphi^*_4\Omega_{V(3,n-1)}$ are exactly the preimages under $\varphi^*_4$ of poles of $\Omega_{V(3,n-1)}$. Note that by inductive hypothesis, $\Omega_{V(3,n-1)}$ has poles exactly along all the boundary divisors of $\Omega_{V(3,n-1)}$.

Thus, we conclude that 
\[
Z(\Omega_{V(3,n)})\subset ZV(123)\cup ZV(145)\cup ZV(124)\cup ZV(134)\cup \varphi_4^{-1}(Z(\Omega_{V(3,n-1)})).
\]

Similarly, using the second equality, we have that 
\begin{equation}\label{inductive on poles}
    Z(\Omega_{V(3,n)})\subset ZV(123)\cup ZV(245)\cup ZV(124)\bigcup ZV(234)\cup \varphi_4^{-1}(Z(\Omega_{V(3,n-1)})).
\end{equation}

By induction on $n$, it is easy to show that the vanishing of a single Pl\"ucker coordinate on $V(3,n)$ is the scheme theoretic union of divisors of the form $Z(A|B)$ for a partition $A\sqcup B$ of $[n]$ into two subsets of size at least $2$ and $ZV(a,b,x:x\in [n])$ for a choice $a\not=b\in [n]$. We now employ the above containment and $D_n$-symmetry to show that if $A,B$ are not simultaneously cyclic intervals in the first case or if $a,b$ are not consecutive numbers in the second case, then they are not contained in $Z(\Omega_{V(3,n)})$.

{\bf Case 1:} Consider the component $\Xi =ZV(a,b,x:x\in [n])$, where $a,b$ are not consecutive. Namely, $a,b$ are separated by two nonempty cyclic intervals. Since $n\geq 7$, one of the two intervals contains at least three elements. By $D_n$-symmetry, we can take $b=2$ and let the longer interval come after $b$. Thus, $a\not=1,3,4,5$. Now, a Pl\"ucker coordinate $p_I$ vanishes on $ZV(abx)$ if and only if $a,b\in I$. This implies that $ZV(123), ZV(245),ZV(124), ZV(234)$ do not contain $ZV(abx)$ since $a$ is disjoint from $[1,5]$. Furthermore, forgetting column $4$ still gives $ZV(abx:x\in [n-1])$ where $a,b$ are non-adjacent, which is not a pole of $\Omega_{V(3,n-1)}$ by inductive hypothesis. So $ZV(abx)$ is not contained in the component $\varphi_4^{-1}(Z(\Omega_{V(3,n-1)}))$, either. Thus, inclusion (\ref{inductive on poles}) shows that $\Xi$ is not contained in $Z(\Omega_{V(3,n)})$.

{\bf Case 2:} Consider the component $\Xi=Z(A|B)$ such that no cyclic intervals in $A$ or $B$ have length at least $2$. In this case, $n$ has to be even, so that $n\geq 8$. Assume that $A$ consists of all the odd numbers and $B$ consists of all the even numbers. Thus, a Pl\"ucker coordinate $p_I$ vanishes on $Z(A|B)$ if and only if all elements in $I$ have the same parity. This implies that $ZV(123), ZV(245), ZV(124), ZV(234)$ do not contain $Z(A|B)$. Furthermore, forgetting column $4$ still gives $Z(A|B)$ where $A$ still has at least four disconnected cyclic intervals, which is not a pole of $\Omega_{V(3,n-1)}$ by inductive hypothesis. So $Z(A|B)$ is not contained in the component $\varphi_4^{-1}(Z(\Omega_{V(3,n-1)}))$, either. Thus, (\ref{inductive on poles}) shows that $\Xi$ is not contained in $Z(\Omega_{V(3,n)})$.

{\bf Case 3:} Consider the component $\Xi = Z(A|B)$ such that $A$ and $B$ each contains a cyclic interval of size at least $2$ such that they are adjacent. By $D_n$-symmetry, take $1,2\in A$ and $3,4\in B$. A Pl\"ucker coordinate $p_I$ vanishes on $Z(A|B)$ if and only if $I\subset A$ or $I\subset B$. So if $I$ has nonempty intersection with both $\{1,2\}$ and $\{3,4\}$, $p_I$ cannot vanish on $Z(A|B)$. This implies that $ZV(123), ZV(245), ZV(124), ZV(234)$ do not contain $Z(A|B)$. Furthermore, forgetting column $4$ does not merge cyclic intervals because $4$ is contained in a cyclic interval of size at least $2$ in $B$. So forgetting column $4$ still gives a partition of $[n-1]$ into two sets which are not cyclic intervals. This is not a pole of $\Omega_{V(3,n-1)}$ by inductive hypothesis. So $Z(A|B)$ is not contained in the component $\varphi_4^{-1}(Z(\Omega_{V(3,n-1)}))$, either. Thus, inclusion (\ref{inductive on poles}) shows that $\Xi$ is not contained in $Z(\Omega_{V(3,n)})$.

{\bf Case 4:} What is left is the component $\Xi = Z(A|B)$ such that $A$ contains an interval of size at least $2$ and for each cyclic interval contained in either $A$ or $B$ of size at least $2$, the adjacent cyclic intervals from the other half of the partition must have length $1$. By $D_n$-symmetry, let $A$ contains an interval of size at least $2$ that starts at $3$. Namely, we can assume that $2\in B$ and $3,4\in A$. By assumtion, the maximal cyclic interval contained in $B$ where $2$ lies in must have length $1$, because otherwise we would be in Case 3. This implies that $1\in A$. A Pl\"ucker coordinate $p_I$ vanishes on $Z(A|B)$ if and only if $I\subset A$ or $I\subset B$. So if $I$ has nonempty intersection with both $\{2\}$ and $\{1,3,4\}$, $p_I$ cannot vanish on $Z(A|B)$. This implies that $ZV(123), ZV(245), ZV(124), ZV(234)$ do not contain $Z(A|B)$. Furthermore, forgetting column $4$ does not merge cyclic intervals because $4$ is contained in a cyclic interval of size at least $2$ in $A$. So forgetting column $4$ still gives a partition of $[n-1]$ into two sets which are not cyclic intervals. This is not a pole of $\Omega_{V(3,n-1)}$ by inductive hypothesis. So $Z(A|B)$ is not contained in the component $\varphi_4^{-1}(Z(\Omega_{V(3,n-1)}))$, either. Thus, inclusion (\ref{inductive on poles}) shows that $\Xi$ is not contained in $Z(\Omega_{V(3,n)})$.\\

\noindent \textbf{Step 3:} By cyclic symmetry, it suffices to show the statement for $Z(A|B)$ and $ZV(34x)$.

Notice that the vanishing of a Pl\"ucker coordinate has an irreducible component given by $Z(A|B)$. The vanishing of $p_{234}$ has an irreducible component given by $ZV(34x)$. So we can compute the residue by removing $d\log$ of such a Pl\"ucker coordinate. After we verify that it agrees with the canonical form prescribed by {\bf Step 1}, we must have simple poles, because otherwise the resulting residue given by removing $d\log$ of such a Pl\"ucker coordinate will not be a rational form. In both cases, we can remove $d\log p_{234}$ because it vanishes at order one by Lemma \ref{order of vanishing}.

{\bf Case 1:} Consider $ZV(34x)$. By Proposition \ref{inductive formula of canonical form on V(3,n)}, removing $d\log$ of $p_{234}=\alpha_1$ one obtains $\frac{-\alpha_2\beta_1}{\alpha_1\beta_2-\alpha_2\beta_1}\varphi_4^* \Omega_{V(3,n-1)}\wedge d\log \beta_1$. Since $\alpha_1=p_{234}$ vanishes on $ZV(34x)$, the rational function $\frac{-\alpha_2\beta_1}{\alpha_1\beta_2-\alpha_2\beta_1}=1$ on $ZV(34x)$. So the residue of $\Omega_{V(3,n)}$ along $ZV(34x)$ is given by $\varphi_4^* \Omega_{V(3,n-1)}\wedge d\log \beta_1$. Notice that $\alpha_1=p_{234}, \gamma_1=p_{134}$ both vanish on $ZV(34x)$. So $\beta_1$ is precisely the ratio of the fourth column divided by the third column. This is exactly the canonical form on $ZV(34x)$ prescribed by {\bf Step 1}. 

{\bf Case 2:} The $Z(A|B)$ case is verified by a direct computation in Proposition \ref{residue along Z(A|B) is good}.
\end{proof}

\section{$V(n-3,n)$ is a Positive Geometry}\label{paozhuan}
The geometric structure of $V(k,n)$ for $k\geq4$ is in general unclear. When $k=n-2$, $V(k,n)=\Gr(k,n)$ is trivial. In this section, we explain the first nontrivial case of $k=n-3$.

As in \cite[Remark 4.10]{ARS24}, we make the observation that $V(d+1,n)\cong V(n-d-1,n)$ as varieties. We briefly explained the argument for self-containment.

Fix consider the standard bilinear form on $V=\mathbb{C}^{n }$ and we have the corresponding duality
\[ \perp:\Gr(d+1,n) \to \Gr(n-d-1,n), \qquad [U] \mapsto [U^{\perp}] \]
In Pl\"ucker coordinates, we have $p_{I}(U)=(-1)^I p_{I^c}(U^\perp)$, where $I$ is any subset of $[n]$ of size $d+1$, $I^c$ is the complement in $[n]$, and $(-1)^I$ is the sign of the permutation with one line notation $[I,I^c]$. The isomorphism, restricted to the open subvariety $\Gr(d+1,n)^\circ$ where all Pl\"ucker coordinates do not vanish, descends to the point configuration space via the Gelfand-MacPherson correspondence. For a configuration of $n$ points in $\mathbb{P}^d$, the image is a configuration of $n$ points in $\mathbb{P}^{n-d-2}$. This map is called the Gale transform. Goppa's duality, as in \cite[Corollary 3.2]{EP00} and the remarks afterward, states that the Gale transform of $n$ distinct points on a common rational normal curve in $\mathbb{P}^{d}$ is again $n$ distinct points on a common rational normal curve in $\mathbb{P}^{n-d-2}$.

Recall that $V(d+1,n)$ is the Zariski closure of the image of the map $\theta_d:\Gr(2,n)^\circ\to \Gr(d,n)^\circ $. From definition of $\theta_d$, a point $[M]\in \Gr(d,n)^\circ$ is in the image of $\theta_d$ if and only if the Gelfand-MacPherson correspondence gives $n$ distinct points on a common rational normal curve in $\mathbb{P}^{d}$. Thus, the duality isomorphism between the Grassmannians sends the image of $\theta_d$ isomorphically to the image of $\theta_{n-d-2}$. 

Now, we consider the automorphism on any Grassmannian $$\operatorname{alt}:\Gr(k,n)\to \Gr(k,n)$$  by multiplying the $j$-th column by $(-1)^{j-1}$. Since column rescaling does not change the corresponding point configuration, the composition $\operatorname{alt}\circ \perp$ is still an isomorphism between the images of $\theta_d$ and $\theta_{n-d-2}$, and thus restricts to an isomorphism between their Zariski closure $V(d+1,n) \cong V(n-d-1,n)$. The merit of composing with the sign change $\operatorname{alt}$ is that we have $p_I(U)=p_{I^c}(\operatorname{alt}(U^\perp))$. The truth of this statement can be found in \cite[Section 7]{Hochster75} and \cite[Theorem 2.2.8]{oxley06}.

\begin{theorem}
The ABCT variety $V(n-3,n)$ is a positive geometry.
\end{theorem}
\begin{proof}
    After the sign change, the isomorphism $\operatorname{alt}\circ \perp:V(d+1,n) \cong V(n-d-1,n)$ restricts to an isomorphism between their totally nonnegative parts. In particular, by taking $d=2$, we conclude the theorem as a corollary of the main theorem.
\end{proof}

The above theorem is a straightforward consequence of the main theorem and Goppa's duality. However, there are interesting combinatorial interpretations of the iterated boundaries of $V(n-3,n)_{\geq0}$, as we illustrate in the example below for $n=7$.

\begin{table}[h]
\centering
\renewcommand{\arraystretch}{2.8} % Increased padding for visual elegance
\begin{tabular}{|c|c|}
\hline
\textbf{$7$ points on a conic in $\mathbb{P}^2$} & \textbf{$7$ points on a cubic in $\mathbb{P}^3$} \\ \hline
% --- (1,1) Plane Conic ---
\begin{tikzpicture}[scale=1.0, baseline=(current bounding box.center)]
    \draw[thick, blue!80!black] (0,0) ellipse (1.8cm and 1.1cm);
    \foreach \n in {1,2,...,7} {
        \pgfmathsetmacro{\angle}{(\n-1)*360/7 + 90} % Rotated for better symmetry
        \filldraw (1.8*cos{\angle}, 1.1*sin{\angle}) circle (1.5pt);
        \node at ({2.2*cos{\angle}}, {1.4*sin{\angle}}) {\small \n};
    }
\end{tikzpicture} 
& 
% --- (1,2) Elegant Twisted Cubic ---
\begin{tikzpicture}[x={(-1.2cm,-0.3cm)}, y={(1.5cm,-0.2cm)}, z={(0cm,0.8cm)}, scale=0.8, baseline=(current bounding box.center)]
    % The curve: r(t) = (t, t^2, 0.4*t^3)
    % A wider domain and lower Z-scale creates a graceful "S" sweep
    \draw[thick, red!80!black, smooth, domain=-1.7:1.7, samples=80] 
        plot ({\x}, {\x*\x}, {0.4*\x*\x*\x});
        
    % 7 Points distributed widely to prevent label collision
    % Points are calculated at t = -1.5, -1.0, -0.5, 0, 0.5, 1.0, 1.5
    \foreach \t [count=\i] in {-1.6, -1.4, -1.1, -0.7, 0, 1.0, 1.5} {
        \pgfmathsetmacro{\zval}{0.4*\t*\t*\t}
        \filldraw ({\t}, {\t*\t}, {\zval}) circle (1.8pt);
        
        % Logic to place labels away from the curve's concavity
        \ifnum \i<4
            \node[anchor=south, xshift=-3pt] at ({\t}, {\t*\t}, {\zval}) {\small \i};
        \else\ifnum \i=4
            \node[anchor=north, yshift=-2pt] at ({\t}, {\t*\t}, {\zval}) {\small \i};
        \else\ifnum \i=6
            \node[anchor=north, yshift=-2pt] at ({\t}, {\t*\t}, {\zval}) {\small \i};
        \else
            \node[anchor=south, xshift=-3pt] at ({\t}, {\t*\t}, {\zval}) {\small \i};
        \fi\fi\fi
    }
\end{tikzpicture} \\ \hline

\begin{tikzpicture}[scale=1]
    % Define the intersection point (origin)
    \coordinate (O) at (0,0);

    % Line 1: Contains points 1, 2, 3
    % Oriented at 150 degrees (pointing up-left)
    \draw[thick, blue!80!black] (-3.5, 2.02) -- (0.5, -0.28); 
    
    % Points on Line 1 (1, 2, 3)
    \foreach \pos [count=\i] in {-3, -2, -1} {
        \pgfmathsetmacro{\x}{\pos * cos(30)}
        \pgfmathsetmacro{\y}{\pos * -sin(30)}
        \filldraw (\x, \y) circle (1.8pt) node[anchor=south west, xshift=-2pt] {\small \i};
    }

    % Line 2: Contains points 4, 5, 6, 7
    % Oriented at 30 degrees (pointing up-right)
    \draw[thick, blue!80!black] (-0.5, -0.28) -- (3.5, 2.02);

    % Points on Line 2 (4, 5, 6, 7)
    \foreach \pos [count=\j] in {0.5, 1.5, 2.5, 3.5} {
        \pgfmathsetmacro{\val}{int(\j + 3)}
        \pgfmathsetmacro{\x}{\pos * cos(30)}
        \pgfmathsetmacro{\y}{\pos * sin(30)}
        \filldraw (\x, \y) circle (1.8pt) node[anchor=south east, xshift=2pt] {\small \val};
    }
    
\end{tikzpicture} & 

\begin{tikzpicture}[scale=1.5]
    % 1. The Conic (Definite Ellipse, not a circle)
    \draw[thick, red!80!black] (0,0) ellipse (1.5cm and 0.8cm);
    
    % 2. The Tangent Line (Slope 2)
    % Point of Tangency T: (-1.449, 0.206)
    % Line Equation: y = 2x + 3.1048
    \draw[thick, red!80!black] (-2, -0.895) -- (-1.15, 0.805);
    
    % 3. Points 1, 2, 3 on the Line (Approaching tangency)
    % Coordinates calculated using y = 2x + 3.1048
    \filldraw (-1.9, -0.695) circle (0.8pt) node[left, xshift=-2pt] {\small 1};
    \filldraw (-1.7, -0.295) circle (0.8pt) node[left, xshift=-2pt] {\small 2};
    \filldraw (-1.55, 0.005)  circle (0.8pt) node[left, xshift=-2pt] {\small 3};

    % 5. Points 4, 5, 6, 7 on the Ellipse (Counter-clockwise flow)
    % Angles chosen to follow the curve around from the tangent point (~172 degrees)
    \foreach \angle [count=\j] in {130, 45, -45, -130} {
        \pgfmathsetmacro{\pnum}{int(\j + 3)}
        \pgfmathsetmacro{\px}{1.5*cos(\angle)}
        \pgfmathsetmacro{\py}{0.8*sin(\angle)}
        \filldraw (\px, \py) circle (0.8pt);
        % Position labels radially outside for clarity
        \node at ({1*\px}, {1.25*\py}) {\small \pnum};
    }
\end{tikzpicture}

\\ \hline

\begin{tikzpicture}[scale=1]
    % Define the intersection point (origin)
    \coordinate (O) at (0,0);

    % Line 1: Contains points 1, 2, 3
    % Oriented at 150 degrees (pointing up-left)
    \draw[thick, blue!80!black] (-3.5, 2.02) -- (0.5, -0.28); 
    
    % Points on Line 1 (1, 2, 3)
    \foreach \pos [count=\i] in {-3, -2} {
        \pgfmathsetmacro{\x}{\pos * cos(30)}
        \pgfmathsetmacro{\y}{\pos * -sin(30)}
        \filldraw (\x, \y) circle (1.8pt) node[anchor=south west, xshift=-2pt] {\small \i};
    }

    % Line 2: Contains points 4, 5, 6, 7
    % Oriented at 30 degrees (pointing up-right)
    \draw[thick, blue!80!black] (-0.5, -0.28) -- (3.5, 2.02);

    % Points on Line 2 (4, 5, 6, 7)
    \foreach \pos [count=\j] in {0.5, 1.25, 2, 2.75,3.5} {
        \pgfmathsetmacro{\val}{int(\j + 2)}
        \pgfmathsetmacro{\x}{\pos * cos(30)}
        \pgfmathsetmacro{\y}{\pos * sin(30)}
        \filldraw (\x, \y) circle (1.8pt) node[anchor=south east, xshift=2pt] {\small \val};
    }
    
\end{tikzpicture}

&
\begin{tikzpicture}[x={(-1.2cm,-0.3cm)}, y={(1.5cm,-0.2cm)}, z={(0cm,0.8cm)}, scale=0.8, baseline=-1.5cm]
    % The curve: r(t) = (t, t^2, 0.4*t^3)
    % A wider domain and lower Z-scale creates a graceful "S" sweep
    \draw[thick, red!80!black, smooth, domain=-1.7:1.7, samples=80] 
        plot ({\x}, {\x*\x}, {0.4*\x*\x*\x});
        
    % 7 Points distributed widely to prevent label collision
    % Points are calculated at t = -1.5, -1.0, -0.5, 0, 0.5, 1.0, 1.5
    \foreach \t [count=\i] in {-1.6,  -1.1, -0.7, 0, 1.0, 1.5} {
        \pgfmathsetmacro{\zval}{0.4*\t*\t*\t}
        \filldraw ({\t}, {\t*\t}, {\zval}) circle (1.8pt);
        
        % Logic to place labels away from the curve's concavity
        \ifnum \i=1
            \node[anchor=east, xshift=-3pt] at ({\t}, {\t*\t}, {\zval}) {\small \i=2};
        \else\ifnum \i=2
            \node[anchor=north, yshift=-2pt] at ({\t}, {\t*\t}, {\zval}) {\small 3};
        \else\ifnum \i=3
            \node[anchor=north, yshift=-2pt] at ({\t}, {\t*\t}, {\zval}) {\small 4};
        
        \else\ifnum \i=4
            \node[anchor=south east, yshift=-2pt] at ({\t}, {\t*\t}, {\zval}) {\small 5};
     \else\ifnum \i=5
            \node[anchor=north, yshift=-2pt] at ({\t}, {\t*\t}, {\zval}) {\small 6};
             \else
            \node[anchor=south east, yshift=-3pt] at ({\t}, {\t*\t}, {\zval}) {\small 7};
        
        \fi\fi\fi\fi\fi
    }
\end{tikzpicture}

\\ \hline
% --- (1,1) Plane Conic ---
\begin{tikzpicture}[scale=1.0]
    \draw[thick, blue!80!black] (0,0) ellipse (1.8cm and 1.1cm);
    \foreach \n in {1,2,...,5} {
        \pgfmathsetmacro{\angle}{(\n-1)*360/7 + 90} % Rotated for better symmetry
        \filldraw (1.8*cos{\angle}, 1.1*sin{\angle}) circle (1.5pt);
        \node at ({2.2*cos{\angle}}, {1.4*sin{\angle}}) {\small \n};
    }

{
        \pgfmathsetmacro{\angle}{(6-1)*360/7 + 90} % Rotated for better symmetry
        \filldraw (1.8*cos{\angle}, 1.1*sin{\angle}) circle (1.5pt);
        \node at ({2.2*cos{\angle}}, {1.4*sin{\angle}}) {\small 6=7};
    }
    
\end{tikzpicture} 

&

\begin{tikzpicture}[scale=1.5]
    % 1. The Conic (Definite Ellipse, not a circle)
    \draw[thick, red!80!black] (0,0) ellipse (1.5cm and 0.8cm);
    
    % 2. The Tangent Line (Slope 2)
    % Point of Tangency T: (-1.449, 0.206)
    % Line Equation: y = 2x + 3.1048
    \draw[thick, red!80!black] (-2, -0.895) -- (-1.15, 0.805);
    
    % 3. Points 1, 2, 3 on the Line (Approaching tangency)
    % Coordinates calculated using y = 2x + 3.1048
    \filldraw (-1.9, -0.695) circle (0.8pt) node[left, xshift=-2pt] {\small 7};
    \filldraw (-1.7, -0.295) circle (0.8pt) node[left, xshift=-2pt] {\small 6};

    % 5. Points 4, 5, 6, 7 on the Ellipse (Counter-clockwise flow)
    % Angles chosen to follow the curve around from the tangent point (~172 degrees)
    \foreach \angle [count=\j] in {130, 90, 45, -45, -130} {
        \pgfmathsetmacro{\pnum}{int(\j )}
        \pgfmathsetmacro{\px}{1.5*cos(\angle)}
        \pgfmathsetmacro{\py}{0.8*sin(\angle)}
        \filldraw (\px, \py) circle (0.8pt);
        % Position labels radially outside for clarity
        \node at ({1*\px}, {1.25*\py}) {\small \pnum};
    }
\end{tikzpicture}
\\ \hline
\end{tabular}
\vspace{0.5cm}
\caption{Boundary Correspondence between $V(3,7)$ and $V(4,7)$}
\end{table}
\begin{example}
In this example, we explicitly compute all boundaries of $(V(4,7),V(4,7)_{\geq0})$.

The boundaries of $(V(3,7), V(3,7)_{\geq0})$ comes in three families up to cyclic rotation. They are $Z(123|4567), Z(34567), Z(67x:x\in [7])_{\red}$. We directly compute the image under $\operatorname{alt}\circ \perp$.

The boundary $Z(123|4567)$ corresponds to the collinear configurations where $c_1,c_2,c_3$ are collinear and $c_4,c_5,c_6,c_7$ are collinear. This is given by the vanishing of $[123]$ and any $3$-element subset of $\{4,5,6,7\}$. Thus, the image is given by the vanishing of $[4567]$ and any $4$-element subset containing $\{1,2,3\}$. This corresponds to a configuration $d_1,\cdots,d_7$ in $\mathbb{P}^3$ such that $d_1,d_2,d_3$ are collinear and $d_4,d_5,d_6,d_7$ are coplanar.

The boundary $Z(34567)$ corresponds to the collinear configurations where $c_3,c_4,c_5,c_6,c_7$ are collinear. This is given by the vanishing of any $3$-element subset of $\{3,4,5,6,7\}$. Thus, the image is given by the vanishing of any $4$-element subset containing $\{1,2\}$. This corresponds to a configuration $d_1,\cdots,d_7$ in $\mathbb{P}^3$ such that $d_1$ and $d_2$ collide. In this case, $6$ points in general position determine a twisted cubic.

The boundary $Z(67x:x\in [7])_{\red}$ corresponds to the configurations where $c_6$ and $c_7$ collide and all points lie on a common conic in $\mathbb{P}^2$. This is set-theoretically defined by the vanishing of any $3$-element subset containing $\{6,7\}$ and a quartic polynomial $[123][345][561][246]-[234][456][612][135]$ for the conic condition. Thus, the image is given by the vanishing of any $4$-element subset of $\{1,2,3,4,5\}$ and the vanishing of $[4567][1267][2347][1357]-[1567][1237][3457][2467]$. The former dictates for the corresponding configuration $d_1,\cdots,d_7$ in $\mathbb{P}^3$ that $d_1,d_2,d_3,d_4,d_5$ must be coplanar. It is a straightforward computation that the quartic relation can be interpreted as ``the line through $d_6,d_7$ passes through the conic determined by the coplanar points $d_1,d_2,d_3,d_4,d_5$."

The above table is a visualization of the conditions on the point configurations given by $V(4,7)$. We conclude that the configurations on a twisted cubic degenerates in one of two ways. We can either ``let two points collide" or ``let the twisted cubic degenerate to a conic and a line such that they intersect." We conjecture that for any $n$, boundaries of $V(4,n)$ are of the form we just described.

\end{example}

\section{Open Problems}
Some follow-up open problems of this article are as follows. 
\begin{enumerate}
    \item Find a manifestly cyclically invariant global expression of the canonical form $\Omega_{V(3,n)}$ instead of in matrix coordinates. See Remark \ref{open problem motivation} for motivations.

    \item The ABCT variety $V(k,n)$ is conjectured to be a positive geometry for all $3\leq k\leq n-3$ by Lam \cite{Lam24}. We verified the conjecture for $k=3,n-3$. Does it still hold for $4\leq k\leq n-4$? The geometry of $V(k,n)$ is more complicated, and we lack results resembling those in \cite{ARS24}. Also, the degeneration strata at the boundary will be much more complicated. Our results and techniques introduced for $V(3,n)$ might be relevant for the general case. As shown in the previous section, we understand $ V(k,k+3)$ well. In particular, boundaries of $V(k,k+3)$ can be given the interpretation of a condition on point configurations, which conjecturally is the same condition describing boundaries for higher $n$.
\end{enumerate}

{
\bibliographystyle{alphaurl}
\bibliography{Biblio}
}

\end{document}